\documentclass [10.5pt,oneside,a4paper,mathscr]{amsart}
\usepackage{geometry}
\newgeometry{margin=1.1in}

\usepackage{amsfonts, amsmath, amssymb, amsthm,mathtools, stmaryrd}
\usepackage[numbers]{natbib}  
\usepackage{verbatim}
\usepackage[normalem]{ulem}
\usepackage{tikz-cd}
\usepackage{enumitem}

\usepackage{array}

\usepackage{hyperref}

\theoremstyle{plain}
\newtheorem{thm}{Theorem}[section]
\newtheorem{cor}[thm]{Corollary}
\newtheorem{lem}[thm]{Lemma}
\newtheorem{prop}[thm]{Proposition}

\numberwithin{equation}{section}

\newtheorem{conjecture}{Conjecture}  
\newtheorem{conj}[conjecture]{Conjecture}  

\theoremstyle{definition}
\newtheorem{defn}[thm]{Definition}
\newtheorem{example}[thm]{Example}
\newtheorem{lemma}[thm]{Lemma}
\newtheorem{rmk}[thm]{Remark}
\newtheorem{question}[thm]{Question}

\theoremstyle{remark}

\newcommand{\BC}{{\mathbb{C}}}

\newcommand{\BE}{{\mathbb{E}}}

\newcommand{\BG}{{\mathbb{G}}}
\newcommand{\BH}{{\mathbb{H}}}

\newcommand{\BQ}{{\mathbb{Q}}}
\newcommand{\BR}{{\mathbb{R}}}

\newcommand{\BZ}{{\mathbb{Z}}}

\newcommand{\CA}{{\mathcal A}}

\newcommand{\CC}{{\mathcal C}}
\newcommand{\CD}{{\mathcal D}}
\newcommand{\CE}{{\mathcal E}}
\newcommand{\CF}{{\mathcal F}}

\newcommand{\CH}{{\mathcal H}}

\newcommand{\CL}{{\mathcal L}}
\newcommand{\CM}{{\mathcal M}}

\newcommand{\CO}{{\mathcal O}}

\newcommand{\CR}{{\mathcal R}}

\newcommand{\Ft}{{\mathfrak{t}}}

\newcommand{\Fz}{{\mathfrak{z}}}

\newcommand{\pt}{{\mathsf{p}}}

\DeclareFontFamily{OT1}{rsfs}{}
\DeclareFontShape{OT1}{rsfs}{n}{it}{<-> rsfs10}{}
\DeclareMathAlphabet{\curly}{OT1}{rsfs}{n}{it}
\renewcommand\hom{\curly H\!om}

\newcommand\End{\operatorname{End}}
\newcommand{\Aut}{\operatorname{Aut}}
\newcommand{\p}{\mathbb{P}}

\newcommand{\Mbar}{{\overline M}}
\newcommand{\vir}{{\text{vir}}}

\newcommand{\Pic}{\mathop{\rm Pic}\nolimits}

\newcommand{\GW}{\mathsf{GW}}

\newcommand{\A}{\mathrm{Aff}}

\newcommand\ev{\operatorname{ev}}
\newcommand{\AHMod}{\mathsf{AHMod}}

\newcommand{\QMod}{\mathsf{QMod}}
\newcommand{\Mod}{\mathsf{Mod}}

\newcommand{\RR}{\mathbb{R}}

\newcommand{\ZZ}{\mathbb{Z}}

\newcommand{\Orth}{\mathrm{O}^+}
\newcommand{\Orthnoplus}{\mathrm{O}}

\newcommand{\id}{\mathrm{id}}
\newcommand{\Tr}{\mathrm{Tr}}

\newcommand{\ch}{\mathsf{ch}}

\newcommand{\wt}{\mathsf{wt}}

\newcommand{\AHJ}{\mathsf{AHJac}}
\newcommand{\QJac}{\mathsf{QJac}}

\newcommand{\Jac}{\mathsf{Jac}}
\newcommand{\SL}{\mathrm{SL}}
\newcommand{\GL}{\mathrm{GL}}

\newcommand{\ct}{\mathrm{ct}}

\newcommand{\RM}{R^{\textup{Maa\ss}}}
\newcommand{\LM}{L^{\textup{Maa\ss}}}
\newcommand{\LJ}{L^{\textup{Jac}}}

\newcommand{\RO}{\mathsf{R}}
\newcommand{\LO}{\mathsf{L}}

\newcommand{\Mon}{\mathrm{Mon}}
\newcommand{\DMon}{\mathrm{DMon}}

\newcommand{\taut}{{\mathrm{taut}}}

\newcommand{\tr}{{\mathrm{tr}}}

\newcommand{\rk}{{\mathrm{rk}}}

\newcommand{\Lift}{{\mathrm{Lift}}}
\newcommand{\Fricke}{{\mathrm{Fr}}}

\renewcommand{\Im}{{\mathrm{Im}}}

\begin{document}
	\baselineskip=14.5pt
	\title[Quasi-modular forms for the orthogonal group]{Quasi-modular forms for the orthogonal group and Gromov-Witten theory of Enriques surfaces}

	\author{Georg Oberdieck}
\address{Universit\"at Heidelberg, Mathematisches Institut}
\email{georgo@uni-heidelberg.de}
	
	\author{Brandon Williams}
\address{Universit\"at Heidelberg, Mathematisches Institut}
\email{bwilliams@mathi.uni-heidelberg.de}

	\date{\today}

\begin{abstract}
We develop the theory of almost-holomorphic and quasimodular forms for orthogonal groups of a lattice of signature $(2,n)$ through orthogonal lowering and raising operators.
The interactions with the regularized theta lift of Borcherds is a central theme.
Our main results are: (i) 
the constant-term morphism, which sends an almost-holomorphic modular form to its associated quasimodular form, is an isomorphism,
(ii) description of spaces of quasimodular forms in terms of vector-valued modular forms,
(iii) the lowering and raising operators satisfy equivariance properties with the theta lift,
(iv) a weight-depth inequality which is a necessary and sufficient criterion for the theta lift of an almost-holomorphic modular form to be almost-holomorphic,
(v) an explicit formula for the series expansion of the lift of any almost-holomorphic modular form,
(vi) the Fourier-Jacobi coefficients of an orthogonal quasimodular form are quasi-Jacobi forms.

As a geometric application, we conjecture that the Gromov-Witten potentials of Enriques and bielliptic surfaces are orthogonal quasimodular forms and satisfy holomorphic anomaly equations with respect to the lowering operators on quasimodular forms. We show that parallel statements for an arbitrary K3 or abelian-surface fibration do not hold.
\end{abstract}

	\maketitle
	
	\setcounter{tocdepth}{1} 
	\tableofcontents

\section{Introduction}
In this paper we develop the theory of quasimodular forms for the orthogonal group with a view towards the theta lift of Borcherds and with applications to the Gromov-Witten theory of Enriques and bielliptic surfaces.
The first part of the introduction concerns the theory of quasimodular forms, the second part the Gromov-Witten theory.

\subsection{Quasimodular forms for orthogonal groups}
Quasimodular forms for $\SL_2$ are generalizations of classical modular forms on the upper half plane $\BH = \{ \tau \in \BC| \Im(\tau) > 0 \}$ that were introduced by Kaneko and Zagier \cite{KZ}. They can either be defined in an ad-hoc way as polynomials in the second Eisenstein series
\[ G_2(\tau) = -\frac{1}{24} + \sum_{n \geq 1} \sum_{d|n} d q^n, \quad q = e^{2 \pi i \tau} \]
in which the coefficients are usual modular forms,
or intrinsically as the constant term (or holomorphic part) of almost-holomorphic modular forms. We refer to Section~\ref{sec:classical case}
for a review. 

The algebra of quasimodular forms has very useful properties. It is isomorphic to the algebra of almost-holomorphic modular forms under the morphism that takes the constant term. The algebra is closed under taking derivatives and the derivative operator extends to an $\mathfrak{sl}_2$ Lie algebra action, that is induced by 
the Maa{\ss} lowering and raising operators on smooth functions on the upper half plane.
In fact, a function on the upper half plane is almost-holomorphic if it is annihilated by a sufficiently high power of the Maa{\ss} lowering operator
\[ \LM = -2 i \mathrm{Im}(\tau)^2 \frac{\partial}{\partial \overline{\tau}}. \]
Quasimodular forms appear naturally as Taylor coefficients of theta series, or more generally of Jacobi forms \cite{EZ} and as 0-th Fourier coefficients of elliptic functions \cite{HAE, GM}, and therefore appear in the theory of partitions \cite{ZagierPartitions,vI}, in representation theory \cite{BO}, in mirror symmetry \cite{Dijkgraaf} and many other instances of enumerative geometry, e.g. \cite{OP,HAE}

More generally, modular forms can be defined in the setting of a Hermitian symmetric space $G/K$ with a discrete subgroup $\Gamma \le G$.
In this paper, we consider Hermitian symmetric spaces of Cartan type IV, which are associated to the groups $O(2,n)$. Concretely, let $M$ be an integral lattice of signature $(2, n)$, $n \in \mathbb{N}$.
The symmetric space can be realized as the period domain
\[ \CD = \{ Z \in \p(M_{\BC}) | Z \cdot Z = 0, Z \cdot \overline{Z} > 0 \}^{+}, \]
where the $+$ stands for one of the two components.
Consider the two Hodge bundles $\CL$ and $\CE$ on $\CD$ of rank $1$ and $n$, respectively,
whose fibers over the point $Z \in \CD$ are given by
\[ \CL|_{Z} = \BC Z, \quad \CE|_{Z} = \BC Z^{\perp} / \BC Z. \]
The modular group is a finite-index subgroup $\Gamma \le \mathrm{O}(M)$ preserving $\CD$, and modular forms of weight $k$ and rank $s$ are defined as $\Gamma$-invariant sections of $\CL^{\otimes k} \otimes \CE^{\otimes s}$. (When $n \leq 2$, a condition at cusps is also required, cf. \cite{Ma}.)
A significant difference between the theory for $\SL_2$ and $O(2,n)$ is the presence of the second Hodge bundle $\CE$, whose sections give rise to vector-valued modular forms. Since $\CE$ is closely related to the tangent bundle of $\CD$ (Lemma~\ref{lemma:tangent bundle}, sections of powers of $\CE$ naturally appear when differentiating scalar-valued (i.e. rank $0$) modular forms and therefore play a central role in our approach to the raising and lowering operators. In other words, even if one is only interested in scalar-valued quasimodular forms, the theory of vector-valued modular forms is crucial.

To define quasimodular forms for orthogonal groups, we first need to make sense of almost-holomorphic modular forms. In Section 2, we use the canonical $G$-invariant K\"ahler metric on $\CD$ to define lowering and raising operators $\LO$ and $\RO$ on the spaces of smooth sections of $\CL^{\otimes k} \otimes \CE^{\otimes s}$. By definition, almost-holomorphic modular forms are then $\Gamma$-invariant sections that are annihilated after sufficiently many applications of $\LO$. In particular, they are nearly-holomorphic modular forms in the general sense of Shimura \cite{Shimura2000}.
Quasimodular forms are then defined as their constant term (or holomorphic part) in an appropriate sense. For $n=1$, where orthogonal modular forms are just classical modular forms for $\SL_2$ due to the isogeny from $\SL_2(\BR)$ to $\mathrm{SO}(2,1)$, our definition specializes to the usual notion of lowering operator and quasimodular forms for $\SL_2$.

When $M$ admits a direct sum decomposition $M \cong U \oplus L$, where $U = \binom{0 \ 1}{1\ 0}$ is the hyperbolic plane, these definitions can be made more explicit in the tube domain model $\CC$ for $\CD$, which is provided by the biholomorphic map $\varphi: \CC \to \CD$, $z \mapsto (-\frac{1}{2} z^2,1, z)$, where
\[ \CC = \{ z \in L_{\BC} | \mathrm{Im}(z)^2 > 0 \}^{+}. \]
With respect to the global coordinates $z$, a function is almost-holomorphic if and only if it is a polynomial in the non-holomorphic variables
\[ \nu_j = \frac{\partial}{\partial z_j} \log ( \mathrm{Im}(z)^2 ), \quad j=1,\ldots,n \] with holomorphic functions as coefficients.
Taking the holomorphic part then simply corresponds to taking the constant term of this polynomial.
Similar results also hold for the expansion at general $0$-dimensional cusps of $\CD$, i.e. for isotropic vectors in $M$ that do not necessarily split a hyperbolic plane (Section~\ref{subsec:general cusp}).

We prove a series of results about quasimodular forms analogous to the classical case of $\SL_2$-quasimodular forms.
The first says that an almost-holomorphic modular form is determined by its holomorphic part:

\begin{thm}[Theorem~\ref{thm:constant term}] The constant-term morphism
\[ \ct : \AHMod_{k,s}(\Gamma) \to \QMod_{k,s}(\Gamma), \quad F \mapsto F|_{\nu=0}, \]
that sends a almost-holomorphic orthogonal modular form to its associated quasimodular form, is an isomorphism.
\end{thm}

The next discusses the commutator of the lowering and raising operators:

\begin{thm}[Proposition~\ref{prop:commutation relation}]
The lowering and raising operators act on the spaces of almost-holomorphic modular forms as derivations by
\[ \LO_k : \AHMod_{k,s}(\Gamma) \to \AHMod_{k-1,s+1}(\Gamma),
\quad \RO_k : \AHMod_{k,s}(\Gamma) \to \AHMod_{k+1,s+1}(\Gamma) \]
and satisfy the identity
\[ {[} \LO, \RO ] F = \frac{k}{2} F \otimes g
+ \frac{1}{2} \sum_{r=1}^{s} \sigma_{r, s+1}( F \otimes g) - \sigma_{r, s+2}(F \otimes g), \] 
where $g$ is the metric and $\sigma_{ij}$ is the transposition interchanging the $i$-th and $j$-th tensor power of $\CE$.
\end{thm} 

An analogue of the second Eisenstein series $G_2(\tau)$ usually does not exist for $O(2,n)$. However, if we allow poles along a divisor $\CH$ which is the zero divisor $\{ f=0 \}$ of a modular form $f$ of positive weight, then the quasimodular form $\partial(f)/f$ which is the holomorphic part of the almost-holomorphic modular form $\RO(f)/f$ is a useful substitute.
In particular, in Section~\ref{subsec:log ahol modular forms} we define the notion of {\em logarithmic} almost-holomorphic modular forms with respect to the pair $(\CD, \CH)$ as weight $k$ modular forms $F$ for which $\LO^d F$ is allowed to have a pole of order $\leq k-d$ along $\CH$ for all $d \geq 0$ but must be almost-holomorphic elsewhere. We obtain the following decomposition:

\begin{thm} There is a natural isomorphism
\[ \AHMod^{\log}_{k,s}(\Gamma) \cong \bigoplus_{d=0}^{k} \Mod_{k-d,s+d}^{\log}(\Gamma)^{S_d}, \]
where $\Mod_{k-d,s+d}^{\log}(\Gamma)^{S_d} \subset \Mod_{k-d,s+d}^{\log}(\Gamma)$ is the subspace of logarithmic rank $s+d$ forms that are invariant under the action of $S_d$ permuting the last $d$ factors.
\end{thm}

In particular, this says that all logarithmic quasimodular forms (including holomorphic quasimodular forms) can be expressed in terms of usual vector-valued modular forms.
The theory of vector-valued orthogonal modular forms was recently developed by Ma \cite{Ma}, who proved important vanishing results. However, almost no results about explicit computations of spaces of vector-valued orthogonal modular forms are known. This makes the computation of spaces of orthogonal quasimodular forms challenging, even in situations where the ring of classical scalar-valued orthogonal modular forms is as simple as possible (the case of a free polynomial algebra, as discussed for example in \cite{WW}).

The most important construction of modular forms on $\mathrm{O}(2, n)$ is the theta lift, which was defined in the work of Oda and Rallis--Schiffmann and others. We follow here Borcherds \cite{Borcherds1998} who extended the theta lift to allow singularities of certain types.
The theta lift takes as input an almost-holomorphic modular form $f$ for the finite Weil representation attached to $M$ of weight $\kappa = k+1 - n/2$, where $k \geq 0$, with poles only at the cusp with at worst exponential growth, and outputs the regularized integral 
$$\mathrm{Lift}(F)(Z) = \frac{1}{2} \Big( \frac{i}{2} \langle Z, \overline{Z} \rangle \Big)^{-k} \int^{\mathrm{reg}} [ F, \Theta_k ] \, \frac{\mathrm{d}x \, \mathrm{d}y}{y^2},$$
where $\Theta_k(\tau, Z)$ is the weight $k$ theta function associated to $M$. We refer to Section~\ref{sec:theta lifts} for necessary background and a review.
Our first main result identifies sufficient and necessary conditions for the lift $\Lift(f)$ to be almost-holomorphic:

\begin{thm}[Theorem \ref{thm:list of almost-holomorphic modular form}]
Let $f$ be an almost-holomorphic modular form of weight $\kappa=k+1-\frac{n}{2}$ and depth $d>0$. If $k \geq 2d$, then $\Lift(f)$ is an almost-holomorphic modular form of depth $2d$.
If in addition $M = U \oplus L$ splits a hyperbolic plane, then $\Lift(F)$ is almost-holomorphic if and only if $k \geq 2d$.
\end{thm}

The proof is based on equivariance formulas of the theta lift with respect to the raising and lowering operators on $\SL_2$ and $O(2,n)$.
We state here the most important of these formulas, which may be of independent interest. 
\begin{thm}[Theorem \ref{thm:L on theta lift}]
If $F$ is an almost-holomorphic modular form of weight $\kappa = k + 1 - n/2$ with $k \ge 2$, then the image of the theta lift $\mathrm{Lift}(F)$ under the orthogonal lowering operator $\LO_k$ is $$\LO_k[\mathrm{Lift}(F)] = -\frac{1}{2\pi} \cdot \RO_{k-2}[\mathrm{Lift}(\LM F)],$$
where $\LM$ is the Maa{\ss} lowering operator on $\SL_2$ and $\RO_{k}$ is the orthogonal raising operator.
\end{thm}

For applications it will be useful to consider Fourier expansions of quasimodular forms. 
We prove the following two results in this direction:

\begin{thm} If we have a decomposition $M \cong U \oplus
	U(r) \oplus N(-r)$, where $N$ is an even positive-definite lattice, then every rank $0$ quasimodular form $f \in \QMod_{k,0}(\Gamma)$ admits a Fourier-Jacobi expansion
	\[ f = \sum_{m} e^{2 \pi i m \tilde{\tau}} f_m(\Fz,\tau) \]
	where $f$ is a quasi-Jacobi form of weight $k$ and index $m$ for $N$.
\end{thm}

This is an analogue of the Fourier-Jacobi expansion of orthogonal modular forms. The details of notation and the precise Jacobi group that appears are explained in Section~\ref{sec:FJ expansion and quasimodular forms}. There we also consider the case of vector-valued modular forms
and discuss a compatibility result between the orthogonal lowering operator $\LO$ and the two types of lowering operators on quasi-Jacobi forms.
Quasi-Jacobi forms are the holomorphic parts of almost-holomorphic Jacobi forms and have been considered before: for example in
\cite{Libgober} or \cite[Appendix A]{Enriques}.

Any quasimodular form $f$ can be expanded in the tube domain in the Fourier series
\[ f = \sum_{\alpha \in U(\Gamma)^{\vee}_{\geq 0}} c_{\alpha} e^{2 \pi i (\alpha, z)} \]
where $U(\Gamma) \subset L_{\BQ}$ is a certain full rank sublattice depending on $\Gamma$, see Section~\ref{subsec:Fourier expansion of quasimodular forms}.
Our last result concerns the Fourier expansions of theta lifts:

\begin{thm}[Cor.\ref{cor:expansion of lift of almostholomorphic mod forms}] \label{thm:expansion quasimod lift}
Assume that $M = U \oplus L$ and let $F(\tau)$ be an almost-holomorphic modular form of weight $\kappa = k + 1 - \mathrm{rank}(L)/2$ with $k \ge 0$ for the Weil representation of $M$ and
consider the associated quasimodular form
\[ \ct(F) = \sum_{\mu \in L'/L} \sum_{n \in \mathbb{Q}} c(\mu; n) q^n \mathfrak{e}_{\mu}.
\]
If we have the depth condition $k \geq 2d$ with $d>0$, so that $\mathrm{Lift}(F)$ is almost-holomorphic, then the associated orthogonal quasimodular form has Fourier expansion 
\[
		\ct(\mathrm{Lift}(F))
		= 
		\frac{1}{2} \zeta(1-k) c(0,0) + 
		\sum_{\substack{\mu \in L^{\vee} \\ \mu > 0}} c(\mu, \mu^2 / 2)  \sum_{\delta=1}^{\infty} \delta^{k-1} e^{2\pi i \langle \delta \mu, z \rangle}.
\]
\end{thm}

This expansion was first obtained by Borcherds for modular forms. Remarkably, the formula for quasimodular forms is exactly the same.

Our proof of Theorem~\ref{thm:expansion quasimod lift} uses a more general result. We describe explicitly the series expansion of the theta lift of any almost-holomorphic modular form, whether the lift is almost-holomorphic or not. The obtained expansion can be found in Theorem~\ref{thm:expansion of theta lift} below, but is too long to concisely state it in the introduction.

\subsection{Related work on quasimodular forms}
Many of our results are closely related to earlier work in the field.
Shimura defined \emph{nearly-holomorphic modular form} for arbitrary Hermitian symmetric spaces in (\cite{Shimura2000}, III.13.11).
Our definition of almost-holomorphic modular forms agrees with his, and part of our work in Sections~\ref{sec:Orthogonal modular forms} and~\ref{sec:tube domain and quasimodular forms} can be viewed as fleshing out his construction in the tube domain and the affine cone (as Shimura only considers the symplectic and unitary case in detail).

Our work also relies heavily on the theory of vector-valued modular forms for the orthogonal group and the geometry of the bundle $\CE$ which was recently investigated in great detail by Ma in \cite{Ma}.

In \cite{Zemel}, Zemel introduced quadratic raising and lowering operators on scalar-valued modular forms for orthogonal groups,
and also proved an equivariance formula for the theta lift.
We show that his operators, which are denoted $\CR,\CL$ here, are related to our raising and lowering operators $\RO,\LO$ by
\[ \CR = \tr( \RO \circ \RO), \quad \CL = \tr( \LO \circ \LO). \]
Our equivariance formulas recover and generalize certain identities proved by Zemel (cf. Proposition~\ref{prop:liftzemel}). For the applications to almost-holomorphic and quasimodular forms that we pursued here, it is however necessary to work with the vector-valued operators $\LO,\RO$.

The special geometry of the period domain of Enriques surfaces was considered in the context of mirror symmetry by Grimm-Klemm-Marino-Weiss in \cite{GKMW}. In particular, they work out the covariant derivative of the K\"ahler metric (our raising operator $\RO$) in the tube domain and discuss certain constructions of almost-holomorphic modular forms. Making the discussion of \cite{GKMW} rigorous and understanding its connection with Gromov-Witten theory was a main motivation of this work.

Quasimodular forms were studied for other types of modular forms such as Siegel modular forms \cite{PSS,KPSW,RZZ} or Hilbert modular forms \cite{MR2377572}. The motivation is usually related to enumerative geometry.

The ``Gauss--Manin connection in disguise" program of Mosavati and collaborators gives a geometrization of quasimodular forms in many classical contexts. The idea is to view quasimodular forms as algebraic functions on natural enhanced period domains. We will not consider these ideas in this paper, but instead refer to the book \cite{MR4375351} for a survey and more details.

\subsection{Gromov-Witten theory}
Gromov-Witten invariants are defined as integrals over the moduli space of stable maps from nodal genus $g$ algebraic curves to a smooth projective target variety $X$ of degree $\beta \in H_2(X,\BZ)$,
\[
\left\langle \tau_{k_1}(\gamma_1) \cdots \tau_{k_n}(\gamma_n) \right\rangle^{X}_{g,\beta}
=
\int_{[ \Mbar_{g,n}(X,\beta) ]^{\vir}} \prod_{i=1}^{n} \ev_i^{\ast}(\gamma_i) \psi_i^{k_i}.
\]
The invariant may be viewed as counting curves in $X$ passing through cycles of class $\gamma_i \in H^{\ast}(X)$ with a "tangency multiplicity" $k_i \geq 0$.
We refer to Section~\ref{sec:basic definitions} for all details.

If $X$ is a Calabi-Yau manifold, there are deep predictions by
Bershadsky-Cecotti-Ooguri-Vafa (BCOV)
of physical origin \cite{BCOV} and rooted in mirror symmetry that link the genus $g$ Gromov-Witten potentials,
\[ F_g(\tau_{k_1}(\gamma_1) \dots \tau_{k_n}(\gamma_n)) = \sum_{\beta \in H_2(X,\BZ)} \langle \tau_{k_1}(\gamma_1) \dots \tau_{k_n}(\gamma_n) \rangle^X_{g,\beta} Q^{\beta}, \]\
where $Q^{\beta}$ is a formel variable, to functions satisfying modular properties.

For Calabi-Yau threefolds, where $\Mbar_{g,0}(X,\beta)$ is of (virtual) dimension zero and we may consider the series $F_g := F_g()$ defined without any incidence conditions, this prediction has the following shape: let $X^{\vee}$ be a Calabi-Yau threefold mirror to $X$, let $\CM$ be its moduli space of complex deformations, and let $\CL$ be the Hodge line bundle.
Then there should be a canonical non-holomorphic section $F_g^{\ast}$ of $\CL^{2g-2}$ such that its asymptotic expansions around a point of maximally-unipotent monodromy is equal to $F_g$ under a canonical transformation of coordinates (the mirror map). 
The non-holomorphic functions $F_g^{\ast}$ can thus be viewed as modular objects, where the group $\Gamma$ under which they transform is given by the monodromy of $X^{\vee}$. The series $F_g$ are their quasimodular limits.
Similar predictions exist for Calabi-Yau manifolds of arbitrary dimension, if one takes into account the presence of descendent insertions $\tau_k(\gamma)$. We refer to \cite{CosLi} for this more general case.

It is very attractive to consider these conjectures when the moduli space $\CM$ of complex deformations of the mirror is the quotient of a Hermitian symmetric domain. The expected modularity of $F_g^{\ast}$ should then be linked with the theory of classical (almost-holomorphic) modular forms of Hermitian symmetric domains.
The simplest examples are Calabi-Yau manifolds of dimension $1$ (elliptic curves) and dimension $2$ (lattice polarized K3 and abelian surfaces).
For Calabi-Yau threefolds however this case is rarely satisfied \cite{FL,LY,MR4375351}.
For elliptic curves, which are mirror to itself, the mirror symmetry prediction says that $F_g^{\ast}$ are almost-holomorphic modular forms for $\SL_2(\BZ)$ and $F_g$ are their quasimodular limits \cite{Dijkgraaf}.
This has been proven in full generality by Okounkov and Pandharipande \cite{OP}.
A BCOV interpretation of their results was given in \cite{LiEll}.
In dimension 2, lattice-polarized K3 surfaces are mirror to K3 surfaces with a mirror-dual lattice polarization \cite{Dolgachev}, and thus the associated moduli spaces $\CM$ are arithmetic quotients of period domains associated to an orthogonal group of a lattice of signature $(2,n)$ (a Hermitian symmetric of Cartan type IV). One thus naturally expects that $F_g^{\ast}$ are (almost-holomorphic) orthogonal modular forms and the $F_g$ are their quasimodular limits.
This is true but for trivial reasons: since K3 surfaces may by deformed to K3 surfaces of Picard rank $0$, all Gromov-Witten invariants of K3 surfaces vanish, so the claim is vacuous.\footnote{A reduced non-trivial Gromov-Witten theory for K3 surfaces has been defined and intensively studied in recent years, but the modular behaviour of the generating series is only partial and the interaction with mirror symmetry is not clear.}

A natural generalization of the above geometric setup is that of Calabi-Yau fibrations.
We are motivated here by the following question:

\begin{question} \label{question:CY fibration}
Let $\pi : X \to B$ be a Calabi-Yau fibration, that is a morphism of smooth projective varieties (or orbifolds) with equidimensional fibers and trivial relative canonical bundle $\omega_{\pi} \cong \CO_X$.
If $\max(3g-3+n,\alpha)>0$, what are the modular properties of the $\pi$-relative Gromov-Witten potentials
\begin{equation} \label{Fgrel} F_{g,\alpha}(\tau_{k_1}(\gamma_1) \cdots \tau_{k_n}(\gamma_n)) = \sum_{\substack{\beta \in H_2(X,\BZ) \\ \pi_{\ast} \beta = \alpha}} \langle \tau_{k_1}(\gamma_1) \cdots \tau_{k_n}(\gamma_n) \rangle^X_{g,\beta} Q^{\beta}, \end{equation}
where $\alpha \in H_2(B,\BZ)$ is the degree in the base?
\end{question}

The condition $\max(3g-3+n,\alpha)>0$, i.e. $3g-3+n>0$ or $\alpha$ effective, ensures that no unstable cases appear in the sum and thus all terms are well-defined.

For a Calabi-Yau fibration $\pi$, the generic fiber is a smooth Calabi-Yau manifold. If $\pi$ has relative dimension one, then $\pi$ is an elliptic fibrations and assuming that $\pi$ has a section, 
the $F_{g,\alpha}$ were conjectured to be quasimodular forms for $\SL_2(\BZ)$ in \cite{HAE,RES}.
In this paper we are interested in the case of relative dimension $2$, where $\pi$ is a K3 or abelian surface fibration,
and want to establish links to orthogonal quasimodular forms.
As a starting point we will restrict ourselves to the two simplest cases: the Enriques and bielliptic surfaces.

Recall that an Enriques surface is the quotient $Y = X/\langle \tau \rangle$ of a K3 surface $X$ by a fixed-point free involution.
We may view Enriques surfaces as K3 fibrations over the orbifold point,
$\pi : Y \to B \BZ_2$, and thus view it as a special case of Question~\ref{question:CY fibration}.
Since $B \BZ_2$ has vanishing $H_2$, the relative Gromov-Witten potentials are just the absolute potentials of the Enriques surface.

Our main conjecture is then the following:

\begin{conj}[Special form of Conjecture~\ref{conj:main} in Section~\ref{subsec:modularity conjecture}]
The Gromov-Witten potentials of the Enriques surfaces
$F_g( \tau_{k_1}(\gamma_1), \ldots, \tau_{k_n}(\gamma_n))$
are components of vector-valued logarithmic quasimodular form for the orthogonal group $O^{+}(M)$ with respect to the pair $(\CD,\CH_{-2})$, where $M$ is the Enriques lattice $U \oplus U(2) \oplus E_8(-2)$ and $\CH_{-2} \subset \CD$ is
the norm $-2$ Noether-Lefschetz divisor.
\end{conj}

We refer to Section~\ref{subsec:modularity conjecture} and Part A for the precise explanation of all the terms.
The precise weights and vector-valued structure is given in Conjecture~\ref{conj:main},
where the modularity is more generally predicted on the level of Gromov-Witten classes.
In Conjecture~\ref{conj:HAE} we also present a holomorphic anomaly equation that 
determines the dependence of the completion $F_g^{\ast}$ of $F_g$ on the non-holomorphic parameters.
We give a wide range of evidence for the conjecture.
Most notably, we show that the linear Hodge integrals of the Enriques surfaces, which were computed in \cite{Enriques}, are the quasimodular forms associated to theta lifts of almost-holomorphic modular form for $\Gamma_0(2)$. Hence they satisfy the conjecture.
We will also explain that the conjecture holds on the level of Fourier-Jacobi expansion (based on \cite{Enriques}) and discuss the implied modularity of the Enriques Calabi-Yau threefold in Section~\ref{subsec:Enriques CY3}.

A parallel example to the Enriques surface is the bielliptic surfaces,
which are quotients of abelian surfaces by fixed point free actions. This case is similar to the Enriques surface, but also considerably simpler because their modularity is essentially implied by the modularity conjectures for elliptic fibrations given in \cite{HAE,RES}. We interpret these predictions as saying that the Gromov-Witten potentials of bielliptic surfaces are orthogonal quasimodular forms for suitable orthogonal groups determined by their group of autoequivalences (Conjecture~\ref{conj:bielliptic surface}). 

In the last section we consider the case of lattice-polarized K3 fibrations $\pi : X \to B$
where the base $B$ is a smooth curve and the fibers of $\pi$ are at most nodal.
For fiber curve classes (those $\beta$ with $\pi_{\ast} \beta = 0$) all Gromov-Witten invariants were then computed by Pandharipande-Thomas \cite{PT_KVV} (based on earlier work of Maulik-Pandharipande \cite{MPGWNL}).
Using their results we show that the relative Gromov-Witten potentials \eqref{Fgrel} for fiber classes have a natural completion which is {\em not} almost-holomorpic.
We also discuss the
fiber Gromov-Witten potentials of the abelian-surface fibered Banana Calabi-Yau threefold determined by Bryan-Pietromonaco \cite{Banana} with a parallel result.
The modular properties of general Calabu-Yau surface fibrations as in Question~\ref{question:CY fibration} appear to be more complicated then just quasimodularity in our sense. This will be taken up in future work.

\subsection{Acknowledgements}
We thank Thomas Blomme, Jan Bruinier and Aaron Pixton for useful discussions.
The first author was supported by the starting grant 'Correspondences in enumerative geometry: Hilbert schemes, K3 surfaces and modular forms', No 101041491 of the European Research Council.

\part{Quasimodular forms for the orthogonal group}
	
\section{The classical case}
\label{sec:classical case}
We recall the classical definition of quasimodular forms for $\mathrm{SL}_2(\mathbb{Z})$ introduced by \cite{KZ}.
Let $\BH = \{ \tau \in \BH | \Im(\tau)>0 \}$ be the upper half-plane.
The group $\SL_2(\BR)$ acts on $\BH$ by $\gamma \cdot \tau = \frac{a \tau + b}{c \tau + d}$ whenever $\gamma = \binom{a\ b}{c\ d}$. The weight $k$ slash operator on functions on $\BH$ is defined by
\[ (f|_{k} \gamma)(\tau) = (c \tau+d)^{-k} f\left( \frac{a \tau + b}{c \tau + d} \right). \]
We have $(f|_{k}\gamma_1) |_{k} \gamma_2 = f|_{k} (\gamma_1 \gamma_2)$ for any $\gamma_1, \gamma_2 \in \SL_2(\BR)$.

The Maa{\ss} lowering and raising operators for functions on  $\BH$ are defined by
\begin{equation} \label{def:Maas raisin and lowering operators}
\begin{gathered}
\RM_k f = 2i \frac{\partial}{\partial \tau} f + \frac{k}{y}f \\
 \LM_k f = -2iy^2 \frac{\partial}{\partial \overline{\tau}} f. 
 \end{gathered}
 \end{equation}
The operators satisfy $\LM_k(y^{k}) = k y^{k+1}$ where $y = \Im(\tau)$ and
\[ 
\RM_k(f)|_{k+2} \gamma = \RM_{k}( f|_{k} \gamma), \quad
\LM_k(f)|_{k-2} \gamma = \LM_{k}( f|_{k} \gamma).
\]

A smooth function $f : \BH \to \BC$ is called {\em almost-holomorphic} if there exists an $\ell \geq 1$ such that $(\LM)^{\ell}(f)=0$. Equivalently, $f$ is a polynomial in $1/y$ with holomorphic functions as coefficients. The slash operators $f \mapsto f|_{k} \gamma$ preserve the ring of almost-holomorphic functions. An almost-holomorphic modular form of weight $k$ for a congruence subgroup $\Gamma \subset \SL_2(\BZ)$ is an almost-holomorphic function $f$ satisfying:
\begin{itemize}
\item $f|_{k} \gamma = f$ for all $\gamma \in \Gamma$ and
\item (cusp condition) for all $\gamma \in \SL_2(\BZ)$, we have the expansion $f|_{k} \gamma = \sum_{i=0}^{\ell} f_i(\tau) y^{-i}$ for some $\ell$ where each $f_i(\tau)$ has a Fourier expansion $\sum_{n} a_n q^n$ with $a_n=0$ for $n<0$.
\end{itemize}
An almost-holomorphic modular form that is holomorphic is called a modular form.
A {\em quasimodular form} of weight $k$ for $\Gamma$ is a holomorphic function $f:\BH \to \BC$ for which there exists an almost-holomorphic modular form $\widehat{f}=\sum_{i} f_i(\tau) y^{-i}$ of weight $k$ for $\Gamma$ with $f_0 = f$. We say that $f$ is the constant term (or holomorphic part) of $\widehat{f}$, and we call $\widehat{f}$ a non-holomorphic completion of $f$. Let $\AHMod_k(\Gamma), \QMod_k(\Gamma), \Mod_k(\Gamma)$ be the spaces of almost-holomorphic modular, quasimodular, and modular forms of weight $k$ respectively. Let $$\ct : \AHMod_k(\Gamma) \to \QMod_k(\Gamma)$$ be the map that sends an almost-holomorphic modular form to its constant term (viewed as a polynomial in $y^{-1}$); this is called the constant-term map and it is an isomorphism, see for example \cite[Prop 3.4]{BO}. In particular, the non-holomorphic completion of a quasimodular form is unique. 
If we let $$\Mod(\Gamma) = \oplus_{k} \Mod_k(\Gamma), \quad \QMod(\Gamma) = \oplus_k \QMod_k(\Gamma)$$ be the algebras of modular and quasimodular forms, then we have the concrete description
\[ \QMod(\Gamma) = \Mod(\Gamma) \otimes_{\BC} \BC[G_2] \]
where $G_2 = -\frac{1}{24} + \sum_{n \geq 1} \sum_{d|n} d q^n$ is the normalized second Eisenstein series.
The non-holomorphic completion of $G_2$ is $$\widehat{G}_2(\tau) = G_2(\tau)+1/8 \pi y.$$

The Maa{\ss} raising and lowering operators define maps on almost-holomorphic modular forms
\begin{align*} \RM_k : \AHMod_k(\Gamma) &\to \AHMod_{k+2}(\Gamma, \\ \LM_k : \AHMod_k(\Gamma) &\to \AHMod_{k-2}(\Gamma). \end{align*}
We define $\RM$ on the graded ring $\AHMod(\Gamma)$ by $\RM|_{\AHMod_{k}(\Gamma)} = \RM_k$ and similarly for $\LM$.
They satisfy the commutation relation 
\[ [\LM,\RM]=-\wt, \]
where the weight operator $\wt$ acts on $\AHMod_k(\Gamma)$ as multiplication by $k$.
We also let $\LM,\RM$ act on quasimodular forms by $$\ct \circ \LM \circ \ct^{-1}, \quad \ct \circ \RM \circ \ct^{-1}.$$ Then we have $\LM(G_2)=\LM(\widehat{G_2}) = -1/8 \pi$. We may therefore formally write
$$\LM = -\frac{1}{8 \pi} \frac{d}{dG_2}.$$
	
	\section{Preliminaries}
\subsection{Notation}
We write $e(x) := e^{2 \pi i x}$ for any complex number $x \in \BC$.

	\subsection{Eichler transvections} \label{subsec:Eichler transvection}
Given an inner product space $M=(V, (-,-))$ consisting of a real vector space $V$ and an inner product $(-,-)$, and two vectors $v,w \in V$ we let $v \wedge_M w \in \mathfrak{so}(V)$ denote the endomorphism,
usually called Eichler transvection, given by
\[ (v \wedge_M w) x = (w,x) v - (v,x) w \text{ for all } x \in V. \]
If $M$ is clear from the notation, we omit the subscript $M$ in $v \wedge_M w$.
    
	\subsection{Connections}
	Let $X$ be a complex manifold
	and let $E$ be a holomorphic vector bundle on $X$.
	We let $\CA(X,E)$ and $\Gamma(X,E)$ denote the spaces of $C^{\infty}$ and
	holomorphic sections of $E$ respectively.
	Since the transition functions are holomorphic,
	there is a well-defined anti-holomorphic derivative
	\[ \overline{\partial} : \CA(X,E) \to \CA(X, E \otimes \overline{\Omega}_X). \]
	If $U \subset X$ is an open set with coordinates $z_1, \ldots, z_n$,
	and we identify sections of $E|_{U}$ with a $\BC^r$-valued function $s = \sum_{i} s_i e_i$ using a trivialization $\CE|_{U} \cong U \times \BC^r$, then we have
	\[ \overline{\partial}(s) = \sum_{i} \frac{\partial s_i}{\partial \overline{z}_j} e_i \otimes d \overline{z}_j. \]
	
	On the other hand, the holomorphic derivative in this chart defined by
	\[ \partial(s) = \sum_{i} \frac{\partial s_i}{\partial z_j} e_i \otimes d z_j \]
	depends on the choice of trivialization.
	To obtain a global derivative (which gives a section of $E \otimes \Omega_X$) one uses a connection,
	which is defined to be a $\BC$-linear morphism $\nabla : \CA(X,E) \to \CA(X,E \otimes \Omega_X)$
	satisfying the Leibniz rule $\nabla(f \cdot s) = \partial(f) s + f \cdot \nabla(s)$
	for any $C^{\infty}$-functions $f$ and any section $s$.\footnote{Strictly speaking, we only consider the $(1,0)$-part of a connection here.}
	The simplest way to construct a connection is as follows.
	Assume that we have an isomorphism of vector bundles $\psi : E \to \overline{F}$, where $F$ is a  holomorphic vector bundle.
	The holomorphic derivative on the anti-holomorphic bundle $\overline{F}$ is well-defined
	and yields a morphism $\partial : \CA(X,\overline{F}) \to \CA(X, \overline{F} \otimes \Omega_X)$. Then we have the associated connection
	\[ \nabla_{\psi} := (\psi^{-1} \otimes \id_{\Omega_X}) \circ \partial \circ \psi : \CA(X,E) \to \CA(X, E \otimes \Omega_X). \]
	
	\begin{rmk}
		A hermitian metric on $E$ is an hermitian scalar product $h_x$ on each fiber $E_x$ which depends smoothly on $x \in X$. The pair $(E,h)$ is called a Hermitian vector bundle.
		As explained for example in \cite[Prop.4.2.14]{HuybrechtsComplexGeometry},
		for any Hermitian vector bundle there is a unique connection
		$$\nabla :  \CA(X,E) \to \CA(X, E \otimes \Omega_X)$$ 
		which is compatible with the Hermitian metric in the sense that
		$\partial h(s_1,s_2) = h(\nabla s_1, s_2)$.
		Assuming that $U \subset X$ is an open set on which $E$ is trivialized, $E|_{U} \cong U \times \BC^r$,
		let $e_i$ denote the basis in this trivialization and $h_{ij} = h(e_i, e_j)$. Then 
		for any section $s = \sum_i s_i e_i$ of $E|_{U}$ we have
		\[ \nabla s = \sum_i \partial(s_i) e_i + \sum_{i,j} s_i A_{ij} e_j \]
		where
		$A_{ij} = \sum_{\ell} \overline{h}^{i \ell} \partial( \overline{h}_{\ell j} )$
		and $h^{ij} = (h^{-1})_{ij}$ are the coefficients of the inverse matrix.
		This connection can be also obtained from the earlier construction
		by considering the isomorphism $E \to \overline{E^{\ast}}$ given by $v \mapsto \overline{h( \overline{-}, v})$.
	\end{rmk}

	\section{Orthogonal modular forms}
    \label{sec:Orthogonal modular forms}

In this section we construct the orthogonal lowering and raising operators $\LO,\RO$ and study their properties. We then define almost-holomorphic modular forms and their logarithmic versions (with poles of a certain kind).
We work both in the period domain $\CD$ as well as $\BG_m$-equivariantly in the affine cone $\A(\CD)$ over it.

	\subsection{Hermitian symmetric domain}
	Let $M$ be an even integral lattice of signature $(2,n)$. Given a ring $R$ we write $M_{R} = M \otimes_{\BZ} R$. The Hermitian symmetric domain of type IV is defined as
	\[ \CD = \CD_M = \{ [Z] \in \p(M_{\BC}) | Z \cdot Z = 0, Z \cdot \overline{Z} > 0 \}^{+} \]
	where $+$ stands for one of the two connected components.
    Let $\Orthnoplus(M)$ be the orthogonal group of $M$ and let $\Orth(M) \le \Orthnoplus(M)$ be the subgroup that preserves the connected components.
	
	There are two Hodge bundles on $\CM$.
	Consider first the line bundle
	\[ \CL = \CO_{\p(M_{\BC})}(-1)|_{\CD} \]
	which is a subbundle of $M_{\BC} \otimes \CO_{\CD}$ by restriction of the tautological sequence.
	The action of $\Orth(M)$ on $M_{\BC} \otimes \CO_{\CD}$ preserves this subbundle,
	hence $\Orth(M)$ acts on $\CL$.
	
	The pairing on $M$ induces a symmetric pairing on $M_{\BC} \otimes \CO_{\CD}$.
	By definition of $\CD$, the subbundle $\CL$ is isotropic.
	The second Hodge bundle is the rank $n$ vector bundle defined by
	\[ \CE = \CL^{\perp}/\CL \]
	where $\CL^{\perp}$ can be defined as the kernel of the surjection $\beta : M_{\BC} \otimes \CO_{\CD} \to \CL^{\vee}$ obtained by dualizing the inclusion $\CL \subset M_{\BC} \otimes \CO_{\CD}$ and then using the isomorphism $M_{\BC}^{\vee} \cong M$ induced by the intersection pairing.
	The fiber of $\CE$ over the point $[Z] \in \CD$ is $$\CE_{[Z]} = Z^{\perp} / \BC Z \subset M_{\BC}/ \BC Z.$$
	
	The intersection pairing on $M_{\BC}$ then yields a canonical non-degenerate holomorphic pairing
	\[ \tr:\CE \otimes \CE \to \CO, \]
	which we call the trace morphism.
	In particular, $\CE$ is an orthogonal bundle, i.e. $\CE \cong \CE^{\vee}$. 

	The relationship between $\CE$ and the tangent bundle of $\CD$ is as follows:
	\begin{lemma}[{\cite[Example 2.3]{Ma}}] \label{lemma:tangent bundle}
		The tangent bundle of $\CD$ admits the canonical isomorphism
		\[ T_{\CD} \xrightarrow{\cong} \CE \otimes \CL^{\vee}. \]
	\end{lemma}
	\begin{proof}
		This is because
		\[ T_{\p(M_{\BC})} \cong \hom\Big( \CO(-1), M_{\BC} \otimes \CO / \CO(-1) \Big) \]
		and the normal bundle sequence of $\CD \subset \p(M_{\BC})$ is precisely
		\[ 0 \to T_{\CD} \to T_{\p(M_{\BC})}|_{\CD} \xrightarrow{\beta} \CL^{\vee}|_{\CD} \to 0. \qedhere \] 
	\end{proof}

\begin{rmk} \label{rmk:identification of sections of E with twisted differentials on D}
By Lemma~\ref{lemma:tangent bundle} and the self-duality $\CE \cong \CE^{\vee}$ we have a canonical isomorphism $$\CE \cong \Omega_{\CD} \otimes \CL^{\vee}.$$ We will therefore identify sections of $\CE$ with the corresponding sections of $\Omega_{\CD} \otimes \CL^{\vee}$, and more generally, sections of $\CE^{\otimes s} \otimes \CL^k$ with the corresponding sections of $\Omega_{\CD}^{\otimes s} \otimes \CL^{k-s}$.
\end{rmk}
    
	\begin{defn} \label{defn:modular forms}
Assume $n \geq 3$.
Let $\Gamma$ be a finite index subgroup of $\Orth(M)$.
A modular form of weight $k$ and rank $s$ for $\Gamma$
is a holomorphic section of $\CL^{\otimes k} \otimes \CE^{\otimes s}$ which is invariant under $\Gamma$.
We write $\Mod_{k,s}(\Gamma)$ for the vector space of modular forms of weight $k$ and rank $s$.
\end{defn}

\begin{rmk}
If $n \leq 2$ then the definition of a modular form is identical but requires additionally a condition at the cusps. We refer to \cite[Section 3.4]{Ma} for details.
\end{rmk}

\begin{rmk}
Ma more generally defines vector-valued modular forms as $\Gamma$-invariant sections of $\CE_{\lambda} \otimes \CL^{\otimes k}$,
where $\CE_{\lambda}$ is the image of $\CE$ under the Schur functor associated to a partition $\lambda = (\lambda_1 \geq \lambda_2 \geq \ldots \geq \lambda_n \geq 0$ such that $\lambda^t_1 + \lambda^t_2 \leq n$, where $\lambda^t$ is the transpose of $t$,
see \cite[Sec.3.2]{Ma}.
Since all the $\CE_{\lambda}$ are subbundles of $\CE^{\otimes s}$ for some $s$, the modular forms of Ma are subspaces of our spaces of modular forms.
In fact, we have a natural action of the symmetric group of $s$ elements on $\Mod_{k,s}(\Gamma)$,
and we can recover the sections of $\CE_{\lambda}$ by applying Schur funtors.
Hence no generalization is lost (at least for our purposes) by considering sections of tensor powers of $\CE$.
\end{rmk}
	
\subsection{Affine geometry}
It will be useful to relate the geometry of $\CD$ to the $\BC^{\ast}$-equivariant geometry of the affine cone
\[ \A(\CD) = \{ [Z] \in M_{\BC} | Z \cdot Z = 0, Z \cdot \overline{Z} > 0 \}^{+}. \]
We say a function $f : U \to \BC$ defined on a $\BC^{\ast}$-equivariant open $U \subset \CA(\CD)$ is of 
bidegree $(k,k')$ if $f(\lambda Z) = \lambda^k \overline{\lambda}^{k'} f(Z)$ for all $Z \in U$ and $\lambda \in \BC^{\ast}$.
We say it is of degree $k$ if it is of bidegree $(k,0)$.

Let $q:\A(\CD) \to \CD$ be the projection. Under descent, 
a smooth (or holomorphic) section $s$ of $\CL^k \otimes \overline{\CL}^{k'}$ over some open $U \subset \CD$ corresponds to a smooth (or holomorphic) function $f : q^{-1}(U) \to \BC$ which is of bigdegree $(-k,-k')$.\footnote{Indeed, $f$ is just given by the image of $q^{\ast}(s)$ under the natural trivialization $q^{\ast}(\CL^k \otimes \overline{\CL}^{k'}) \cong \CO \otimes \Ft^{-k} \otimes \overline{\Ft}^{-k'}$, where $\Ft$ is the standard representation of $\BC^{\ast}$.}
	In particular,
	a modular form of weight $k$ and rank $0$ corresponds to a degree $-k$ holomorphic function $f:\A(\CD) \to \BC$ such that
	\[ f(\gamma Z) = f(Z) \text{ for all } \gamma \in \Gamma. \]
In what follows we identify sections of $\CL^k \otimes \overline{\CL}^{k'}$
with the corresponding functions on the affine cone.
	
\begin{example}(A non-holomorphic modular form)
	For $Z \in M_{\BC}$ we write $|Z|^2 = \frac{1}{2} Z \cdot \overline{Z}$. Consider the function
	on the affine cone $\A(\CD)$ given by
	\[ f(Z) = |Z|^2. \]
		We have $f \in \CA(\CD, \CL^{\vee} \otimes \overline{\CL}^{\vee})$ and $f(\gamma Z) = f(Z)$ for all $\gamma \in \Orth(V)$.
\end{example}

Next we give a description of sections of $\Omega_{\CD}$.
For that we need to understand $q^{\ast} \Omega_{\CD}$.
	Consider the inclusion $\CD \to \p(M_{\BC})$. We have the associated short exact sequence
	\[ 0 \to I/I^2 \xrightarrow{d} \Omega_{\p(M_{\BC})}|_{\CD} \to \Omega_{\CD} \to 0 \]
	where $I$ is the ideal generated by $Z^2$ and the first map is the differential.
	The Euler sequence is
	\[
	0 \to \Omega_{\p(M_{\BC})} \to \CO(-1) \otimes M_{\BC}^{\vee} \xrightarrow{\rho} \CO_{\p(M_{\BC})} \to 0
	\]
	where $\rho$ is the dual of the tautological inclusion $\CO(-1) \subset M_{\BC} \otimes \CO_{\p(M_{\BC})}$. 
	Suppose we choose a basis $e_i$ of $M$.
	Under pullback by $q$, the sections of $\CO(-1) \otimes M_{\BC}^{\vee}$ correspond to differentials $\sum_i a_i(Z) dZ_i$ where $a_i(Z)$ is of degree $-1$. Moreover, the morphism $\rho$ is defined by $\rho(dZ_i) = Z_i$,
	or equivalently, it is given by contraction with the radial vector field 
    \[ v=\sum_i Z_i \frac{\partial}{\partial Z_i}. \]
	Hence sections of $\Omega_{\p(M_{\BC})}$ over an open set $U$ correspond to 
	differentials 
	$\omega = \sum_i a_i(Z) dZ_i$, defined on $q^{-1}(U)$ with $a_i(Z)$ of degree $-1$, for which the contraction with $v$ vanishes.
	Then sections of $\Omega_{\CD}$ corresponds to such forms
	where $Z \in \A(\CD)$ and the differential is taken modulo the 
	ideal generated by $$d(Z^2) = 2 \sum_{a,b} Z_a g_{ab} dZ_b,$$
	where $g_{ab} = \langle e_a, e_b \rangle$ is the pairing of the basis vectors.

Hence viewing sections of $\CE$ as sections of $\Omega_{\CD} \otimes \CL^{\vee}$ as in 
Remark~\ref{rmk:identification of sections of E with twisted differentials on D},
we find that
a section of $\CE$ corresponds to a 
differential form
$\omega = \sum_i a_i(Z) dZ_i$ on $A(\CD)$ with $a_i(Z)$ of degree $0$,
whose contraction with $v=\sum_i Z_i \partial/\partial Z_i$ vanishes
and which is taken modulo the ideal $dZ^2$.

We arrive at the following description for the sections of $\CE^{\otimes s} \otimes \CL^{k}$:
\begin{lemma} \label{lemma:identification E}
There is a canonical isomorphism between the space of smooth (resp. holomorphic) sections $s$ of $\CE^{\otimes s} \otimes \CL^k$ over some open $U \subset \CD$ and the space of multi-differentials
\[ F = \sum_{I = (i_1, \ldots, i_s)} F_I dZ_I, \quad dZ_I = dZ_{i_1} \otimes \cdots \otimes dZ_{i_s}, \]
defined on some open $q^{-1}(U) \subset \A(\CD)$ modulo the subspace $W$, where
\begin{itemize}
\item $F_I(Z)$ are smooth (resp. holomorphic) functions of degree $-k$,
\item the contraction $F(\ldots,v,\ldots)$ of $F$ with $v=\sum_k Z_k \partial/\partial Z_k$ in any of its entries vanishes:
\[ \forall r \in \{ 1,\ldots, s\}:\quad  \sum_{i_r} F_{(i_1,\ldots,i_s)} Z_{i_r} = 0, \]
\item $W$ is the subspace spanned by multi-differentials of the form $F_1 \otimes dZ^2 \otimes F_2$ for any $F_1, F_2$.
\end{itemize}
\end{lemma}
\begin{proof}
This follows from the discussion before: to a section $t$ of $\CE^{\otimes s} \otimes \CL^k \cong \Omega_{\CD}^{\otimes s} \otimes \CL^{k-s}$ we associate $F=q^{\ast}(t)$ where we use the canonical trivialization of $q^{\ast} \CL$. To $F$ we associate $t$ by descent.
\end{proof}
        
We can also describe the trace morphism
$\tr : \CE \otimes \CE \to \CO$ in the affine geometry:
	
	\begin{lemma} \label{lemms3sdfd}
		If we identify the sections $\omega_1, \omega_2$ of $\CE$ with the differentials
		$\omega_1 = \sum_i a_i(Z) dZ_i$ and $\omega_2 = \sum_j b_j(Z) dZ_j$
		as in Lemma~\ref{lemma:identification E}, then
		\begin{equation} \label{3423}
			\tr(\omega_1, \omega_2) = \sum_{i,j} g^{ij} a_i(Z) b_j(Z).
		\end{equation}
	\end{lemma}
	\begin{proof}
 We first check that the expression \eqref{3423} is well-defined:
		Indeed, if $$\omega_1 = dZ^2 = 2\sum_i Z_a g_{ab} dZ_b,$$
		then $\tr(\omega_1, \omega_2) = \omega_2(v) = 0$ since $\omega_2$ is a section of $\CE$.
		The claim now follows by construction of the pairing on $\CE$.
		Namely, under the canonical isomorphism $\CE \cong T_{\CD} \otimes \CL$,
		the pairing of the tangent vectors $\partial/\partial Z_i$ and $\partial/\partial Z_j$ is
		$g_{ij}$; however, we use the inverse of $g$ to identify $\CE \to \CE^{\vee}$,
		so by duality the pairing between $dZ_i$ and $dZ_j$ is $g^{ij}$.
	\end{proof}

\begin{rmk}(Warning) \label{rmk:Warning}
The trace morphism $\tr: \CE \to \CE \to \CO$ can be viewed as a section $$s_{\tr} : \CO \to \CE^{\vee} \otimes \CE^{\vee} \cong \CE \otimes \CE.$$
 Then $\tr(s_{\tr}) =n \in \Gamma(\CO_{\CD})$ as one can  check fiberwise (on $\CD$).

However, we can also consider the form on the affine cone corresponding to $s_{\tr}$, which is given by
\[ \overline{s}_{\tr} = \sum_{i,j} g_{ij} dZ_i \otimes dZ_j. \]
If we naively apply the formula of Lemma~\ref{lemms3sdfd}, i.e. substitute $dZ_i \otimes dZ_j \mapsto g^{ij}$, we get $$\sum_{i,j} g_{ij} g^{ij} = \dim(M_{\BC}) = n+2,$$ which is false. The problem is that $\overline{s}_{\tr}$ is not written in the form $\omega_1 \otimes \omega_2$ (or a linear combination of such forms) where $\omega_1, \omega_2$ are sections of $\CE$ satisfying $\omega_i(v)=0$. So Lemma~\ref{lemms3sdfd} does not apply.
\end{rmk}

Contraction of the $(i,j)$-th indices of tensors defines a morphism
	\[
	\tr_{ij} : \CE^{\otimes s} \to \CE^{\otimes (s-2)},
	\quad \quad \tr_{ij}(v_1 \otimes \cdots \otimes v_s) := 
	\tr(v_i, v_j) v_1 \otimes \cdots \hat{v}_i \otimes \cdots \otimes \hat{v}_j \otimes v_s,
	\]
where $\hat{\ }$ means the corresponding entry is omitted.
More generally, contracting indices $a_i,b_i$ pairwise for $i=1,\ldots, r$, with $a_1, b_1, \ldots, a_r, b_r$ all distinct,  yields the morphism 
\begin{gather*} \tr_{a_1b_1, \ldots, a_r b_r} : \CE^{\otimes s} \to \CE^{\otimes (s-2r)} \\
\tr_{a_1b_1, \ldots, a_r b_r}( v_1 \otimes \cdots \otimes v_s )
= \prod_{j=1}^{r} \tr( v_j \otimes w_j) \cdot \bigotimes_{i \notin \{ a_1, b_1, \ldots, a_r, b_r \}} v_i.\end{gather*}

We will also use the pairing
\begin{equation}
\label{eqn:defn of pairing}
\langle - , - \rangle : \CA(U,\CE^{\otimes (s+t)}) \times \CA(U, \CE^{\otimes t}) \to \CA(U, \CE^{\otimes s}), \quad (F,G) \mapsto \langle F,G \rangle
\end{equation}
which is given by contracting the final $s$ factors pairwise, i.e. 
$$\left\langle F, G \right\rangle =
\tr_{(s+1,s+t+1), \ldots, (s+t,s+2t)}(F \otimes G).$$

	The relationship between sections on $\CD$ and equivariant sections on the cone is summarized in Figure~\ref{figure:1}.

	\begin{figure}[h!]
		\begin{tabular}{lcl}
			\hline
			Geometry of $\CD$ & $\quad \longleftrightarrow \quad$& Equivariant geometry of $A(\CD)$ \\
			\hline 
			\vspace{5pt}
			Section $f$ of $\CL^k$ & $\quad \longleftrightarrow \quad$ & Function $f : A(\CD) \to \BC$ of degree $-k$, that is $f(\lambda Z) = \lambda^{-k} f(Z)$ \\[5pt]
			Section $\omega$ of $\Omega_{\p(M_{\BC})}$  & $\quad \longleftrightarrow \quad$ & Differential $\omega = \sum_i a_i(Z) dZ_i$ with $a_i: M_{\BC} \to \BC$ of degree $-1$\\
   && and $\iota_{v} \omega = 0$ where $v=\sum_i Z_i \frac{\partial}{\partial Z_i}$ \\[10pt]
			Section $\omega$ of $\Omega_{\CD}$ & $\quad \longleftrightarrow \quad$ & 
			\begin{minipage}{0.65\textwidth}
				Differential $\omega = \sum_i a_i(Z) dZ_i$ with $a_i: A(\CD) \to \BC$ of degree $-1$, $\iota_{v} \omega = 0$ and modulo the ideal generated by $d(Z^2)$
			\end{minipage} \\[10pt]
			Section $\omega$ of $\CE$ & $\quad \longleftrightarrow \quad$ & 
			\begin{minipage}{0.65\textwidth}
				Differential $\omega = \sum_i a_i(Z) dZ_i$ with $a_i: A(\CD) \to \BC$ of degree $0$, $\iota_{v} \omega = 0$ and modulo the ideal generated by $d(Z^2)$
			\end{minipage} \\[10pt]
			$\tr : \CE \otimes \CE \to \CO$ & $\quad \longleftrightarrow \quad$ & 
			$\tr(\sum_i a_i(Z) dZ_i, \sum_j b_j(Z) dZ_j) = \sum_{i,j} g^{ij} a_i(Z) b_j(Z)$
		\end{tabular}
  \caption{}
	\label{figure:1}
\end{figure}

\subsection{An isomorphism from $\CE$ to $\overline{\CE}$}
Our next goal will be to define a natural $O^{+}(M)$-invariant non-degenerate section
$H \in \CA( \CD, \Omega_{\CD} \otimes \overline{\Omega}_{\CD} )$.
We will first define this section on the affine cone $\A(\CD)$ and then descend it to $\CD$.
 Concretely, define
	\[
	H := \partial \overline{\partial} \log( |Z|^2 ), \quad Z \in M_{\BC}.
	\]
	If $e_i$ is a basis of $M$, and $Z_i$ are the corresponding coordinates and
	$g_{ab} = \langle e_a, e_b \rangle$, then
	\[
	H = \sum_{i,j} H_{ij} dZ_i \otimes d \overline{Z}_j
	\]
	with
	\begin{align*}
		H_{ij} 
		& = \frac{\partial}{\partial Z_i} \frac{\partial}{\partial \overline{Z}_j} \log(|Z|)^2 \\
		& = \frac{g_{ij}}{2 |Z|^2} - \frac{1}{4 |Z|^4} \sum_{a} g_{ia} \overline{Z}_a \sum_b g_{jb} Z_b.
	\end{align*}
	Observe that $\overline{H_{ij}} = H_{ji}$, so that
	$\overline{H} = H^{\mathrm{swap}}$, where $(-)^\mathrm{swap}$ stands for swapping the two factors.
	Moreover, $H(v, -) = 0$, hence $H(-, \overline{v})$ also. Also $H(\lambda Z) = |\lambda|^{-2} H(Z)$.
	\begin{prop}
		$H$ is a well-defined element of $\CA(\CD, \Omega_{\CD} \otimes \overline{\Omega}_{\CD})$ satisfying
		\begin{enumerate}
			\item[(a)] $H$ is $O^{+}(M_{\BR})$-invariant;
			\item[(b)] $H$ is non-degenerate; that is, it induces an isomorphism $T_{\CD} \to \overline{\Omega}_{\CD}$.
		\end{enumerate}
	\end{prop}
	\begin{proof}
		The first claim is clear from the definition.
		Since $\Orth(M_{\BR})$ acts transitively,
		it suffices to prove the statement for the restriction to one point.
		This can be done by a direct computation.\footnote{Concretely, consider a basis $e_a$ (over $\BR$) with $e_1^2 = 1$, $e_2^2=1$ and $e_a \cdot e_1 = e_a \cdot e_2$ for $a>2$. Over the point $Z=e_1 + i e_2$ we have $H = (dZ_1 - i dZ_2)\otimes \alpha + (d Z_1 + i dZ_2) \otimes \beta + \sum_{i,j \geq 3} g_{ij}/2 |Z|^2$, for some $(0,1)$ forms $\alpha, \beta$.
			Since $H(v,-) = 0$ we see that $\alpha = 0$ and since we work modulo $dZ^2$, we can assume $\beta=0$. The remaining terms are non-degenerate on the subspace spanned by $\partial/\partial Z_i$, $i=3,\ldots, n+2$.}
	\end{proof}

We use $|Z|^2 H \in \Gamma(\CD, \CE \otimes \overline{\CE})$ to define the morphism
\[ \psi := \tr_{12} \circ ( - \otimes 2|Z|^2 H) : \CE \to \overline{\CE}. \]
We also have its complex conjugate $\varphi := \overline{\psi} : \overline{\CE} \to \CE$.
\begin{lemma}
	$\varphi \circ \psi = \id_{\CE}$ and $\psi \circ \varphi = \id_{\overline{\CE}}$.
\end{lemma}
\begin{proof}
	Since $\overline{H} = H^{\mathrm{swap}}$, we have
	\[ \varphi = \tr_{13} \circ ( - \otimes 2|Z|^2 H ) : \overline{\CE} \to \CE. \]
	Define the element
	\[ R_{ij} := 2 |Z|^2 \sum_{\ell} g^{i \ell} H_{\ell j}. \]
	A short computation shows that
	\[ R_{ij} = \delta_{ij} - \frac{1}{2 |Z|^2} \overline{Z}_i \sum_{b} g_{jb} Z_b. \]
	We have
	\begin{align}
		\label{psi}
		\psi\left( \sum_k a_k(Z) dZ_k \right) & = 2 |Z|^2 \sum_{k,i,j} a_k g^{ki} H_{ij} d\overline{Z}_j =  \sum_{k,j} a_k(Z) R_{kj} d\overline{Z}_j \\
		\label{phi}
		\varphi\left( \sum_k b_k(Z) d\overline{Z}_k \right) & = 2 |Z|^2 \sum_{k,i,j} b_k g^{kj} H_{ij} dZ_i = \sum_{i,j} b_k(Z) \overline{R}_{kj} dZ_j.
	\end{align}
 Intuitively, $\psi$ and $\varphi$ can be interpreted as the replacement rules
 \[ dZ_{k} \mapsto \sum_j R_{kj} d \overline{Z}_j, \quad d \overline{Z}_k \mapsto \sum_j \overline{R}_{kj} dZ_j \]
	Since $Z^2 = 0$ we have
	\begin{equation} \label{RR} R_{ij} \overline{R}_{j \ell} = \delta_{i \ell}
		- \frac{1}{2|Z|^2} \left( \overline{Z}_i \sum_{a} g_{\ell a} Z_a 
		+ Z_i \sum_{b} g_{\ell b} \overline{Z}_b \right). \end{equation}
	If $\omega = \sum_i a_i(Z) dZ_i$ is a section of $\CE$, such that $\iota_{v} \omega = 0$, then
	\begin{align*}
		\varphi \psi(\omega) & = \sum_{i,j,\ell} a_i(Z) R_{ij} \overline{R}_{j \ell} dZ_{\ell} \\
		& = \omega - \frac{1}{2|Z|^2} \sum_{i} a_i(Z) \overline{Z}_i \underbrace{\sum_{a,\ell} g_{\ell a} Z_a dZ_{\ell}}_{= dZ^2 = 0}
		- \frac{1}{2|Z|^2} \underbrace{\sum_{i} a_i(Z) Z_i}_{\iota_{v} \omega = 0} \sum_{a} g_{\ell a} \overline{Z}_a \\ & = \omega.
	\end{align*}
	The second claim follows from the first.
\end{proof}

\subsection{Raising and lowering operators}
The raising and lowering operators are modular substitutes for holomorphic and anti-holomorphic derivatives of modular forms.

\begin{defn} The lowering and raising operators are defined by
	\begin{gather*}
		\LO_k := |Z|^2 (\id_{\CE^{\otimes s}} \otimes \varphi) \circ \overline{\partial} :
		\CA(U, \CL^k \otimes \CE^s) \to \CA(U, \CL^{k-1} \otimes \CE^{s+1}); \\
		\RO_k := |Z|^{-2k} (\varphi^{\otimes s} \otimes \id_{\CE}) \circ \partial \circ
		|Z|^{2k} \psi^{\otimes s} : 
		\CA(U, \CL^k \otimes \CE^s) \to \CA(U, \CL^{k+1} \otimes \CE^{s+1})
	\end{gather*}
	on any open $U \subset \CD$.
\end{defn}

If clear from the context, we usually drop the subscript $k$ in $\LO_k,\RO_k$.

We will give a more explicit formula for both $\LO$ and $\RO$. Define
\[ \nu = \partial \log( |Z|^2 ) = \frac{1}{2 |Z|^2} \sum_{a,b} \overline{Z}_b g_{ab} dZ_a \in \CA( \A(\CD), \Omega_{\A(\CD)} ). \]
Note that $\nu$ is not a section of $q^{\ast} \Omega_{\CD}$ because $\iota_{v} \nu = 1$ for $v=\sum_i Z_i \partial/\partial Z_i$.

Also, define
\begin{equation} g = \sum_{a,b} g_{ab} dZ_a \otimes dZ_b \in \Gamma(\CE \otimes \CE). \label{g on period domain} \end{equation}

\begin{prop} \label{prop:Properties of L and R}~ 
	\begin{enumerate}
		\item[(i)] The operators $\LO$ and $\RO$ are well-defined and derivations.
		\item[(ii)] For any $F \in \CA(U, \CL^k \otimes \CE^{\otimes s})$,
		\[ \RO(F) = |Z|^{-2k-2} \varphi^{\otimes (s+1)} \overline{\LO \Big( |Z|^{2k} \varphi^{\otimes s} \overline{F} \Big)} \]
		\item[(iii)] For any $F \in \CA(\CD, \CL^k \otimes \CE^{\otimes s})$ and $\gamma \in O^{+}(M_{\BR})$ we have
		\[ \gamma^{\ast} \RO(F) = \RO(\gamma^{\ast} F), \quad \gamma^{\ast} \LO(F) = \LO(\gamma^{\ast}F). \]
		\item[(iv)]
		If we write $F = \sum_{I = (i_1, \ldots, i_s)} F_I dZ_I$, where $dZ_I = dZ_{i_1} \otimes \cdots \otimes dZ_{i_s}$, then
\begin{gather}
\LO(F) = \sum_I \frac{\partial F_I}{\partial \overline{Z}_k} dZ_I \otimes \left( |Z|^2 dZ_k - Z_k \partial(|Z|^2) \right) \label{formulaL} \\
		\RO(F) = \partial(F) + k F \otimes \nu
		+ \sum_{r=1}^{s} \sigma_{r, s+1}(F \otimes \nu)
		- \sum_{r=1}^{s} \tr_{s+2,s+3} \sigma_{r, s+2}( F \otimes g \otimes \nu ) \label{formulaR}
\end{gather}
		where $\sigma_{i,j}$ stands for the transposition of a set $\{ 1, \ldots, \ell \}$ interchanging $i,j$.
	\end{enumerate}
\end{prop}
\vspace{5pt}

An explicit formula for $\RO$ in terms of coordinates is given in the proof.

\begin{example}
	For example, if $\omega$ is a section of $\CE \otimes \CL^k$, then
	\[ \RO(\omega) = \partial(\omega) + k \omega \otimes \nu
	+ \nu \otimes \omega - \tr(\omega \otimes \nu) g. \]
	If $\omega_1, \omega_2$ are sections of $\CE$ then
	\begin{align*}
		\RO(\omega_1 \otimes \omega_2)
		& =
		\partial(\omega_1 \otimes \omega_2)
		+ \nu \otimes \omega_2 \otimes \omega_1 + \omega_1 \otimes \nu \otimes \omega_2 \\
  & - \tr(\omega_1 \otimes \nu) \sigma_{23}( g \otimes \omega_2 ) - \tr(\omega_2 \otimes \nu) \omega_1 \otimes g
	\end{align*}
	More generally, if $F = \omega_1 \otimes \ldots \otimes \omega_s$ is a section of $\CL^k \otimes \CE^s$, then
	\begin{align*}
		\RO(\omega_1 \otimes \cdots \otimes \omega_s) &= \partial(\omega_1 \otimes \cdots \otimes \omega_s) + k \omega_1 \otimes \cdots \otimes \omega_s \otimes \nu \\
		& + \sum_{r=1}^{s}  \omega_1 \otimes \ldots \otimes \omega_{r-1} \otimes \nu \otimes \omega_{r+1} \otimes \cdots \otimes \omega_{s} \otimes \omega_r \\
  		& - \sum_{r=1}^{s} \sum_{a,b} g_{ab} \tr(\omega_r \otimes \nu) \omega_1 \otimes \cdots \otimes \omega_{r-1} \otimes dZ_a \otimes \omega_{r+1} \otimes \cdots \otimes \omega_{s} \otimes dZ_{b}.
	\end{align*}
\end{example}

\begin{example}
	A local section $f$ of $\CL^k$ corresponds to a degree $-k$ function $f:A(\CD) \to \BC$.
	The derivative of $f$ should be a section of $\CL^{k} \otimes \Omega_{\CD}$, which corresponds to a
	differential $\omega = \sum_i a_i(Z)$ with $a_i$ of degree $-(k+1)$
	such that $\iota_{v} \omega = 0$, where $v=\sum_i Z_i \frac{\partial}{\partial Z_i}$.
	The naive derivative
	$\partial(f) = \sum_{i} \frac{\partial f}{\partial Z_i} dZ_i$ satisfies the first condition,
	but by Euler's equation we have
	\[ \iota_{v} \partial(f) = \sum_{i} Z_i \frac{\partial f}{\partial Z_i} = -k f(Z). \]
	The raising operator corrects for the difference:
	\[ \RO(f) = \partial(f) + k \partial( \log(|Z|^2) ) f \]
	and
	\[ \iota_{v} \partial( \log  |Z|^2 ) = 1. \]
\end{example}
\begin{example}
	Consider a local section of $\CE$ corresponding to a differential
	$\omega = \sum_{k} a_k(Z) dZ_k$. One could consider the naive derivative
	\[ 
	\partial(\omega) = \sum_{k,\ell} \frac{\partial a_k}{\partial Z_{\ell}} dZ_k \otimes dZ_{\ell}.
	\]
	However, $\partial(\omega)$ does not respect the ideal $(dZ^2)$. Indeed, we have
	\[
	\partial(dZ^2) = \sum_{a,b} g_{ab} dZ_a \otimes dZ_{b} \]
	which is non-zero modulo $dZ^2$. Moreover,
	\[ \partial(\omega)(v, - ) = \sum_i \partial(a_i) Z_i = \partial( \sum_i a_i Z_i ) - 
	\sum_i a_i dZ_i = - \omega \]
	is also non-zero. The raising operator corrects for both of these defects:
	it satisfies $\RO(dZ^2) = 0$ modulo $dZ^2$ and $\RO(\omega)(v, - )=0$,
	so that we obtain a section of $\CE^{\otimes 2} \otimes \CL$. \qed
\end{example}

\begin{proof}
	(i) The lowering and raising operators are the compositions
	\begin{multline*}
		\CA(U, \CL^k \otimes \CE^s) \xrightarrow{\overline{\partial}}
		\CA(U, \CL^k \otimes \CE^s \otimes \overline{\Omega})
		\cong \CA(U, \CL^k \otimes \CE^s \otimes \overline{\CE} \otimes \overline{\CL} ) \\
		\xrightarrow{\id_{\CE^{\otimes s}} \otimes \varphi}
		\CA(U, \CL^k \otimes \CE^{s+1} \otimes \overline{\CL} )
		\xrightarrow{|Z|^2} 
		\CA(U, \CL^{k-1} \otimes \CE^{s+1} ),
	\end{multline*}
	and
	\[
	\CA(U, \CL^k \otimes \CE^s) \xrightarrow{|Z|^{2k} \psi^{\otimes s}}
	\CA(U, \overline{\CL}^{\vee k} \otimes \overline{\CE}^s)
	\xrightarrow{\partial} 
	\CA(U, \overline{\CL}^{\vee k} \otimes \overline{\CE}^s \otimes \CE \otimes \CL)
	\xrightarrow{ |Z|^{-2k} \phi^{s} \otimes \id_{\CE}}
	\CA(U, \CL^{k+1} \otimes \CE^{s+1}).
	\]
	Hence they are well-defined. Since $\partial$ and $\overline{\partial}$ are derivations, so are $\LO$ and $\RO$. \\
	(ii) This is a straightforward computation using  $\overline{\psi} = \varphi$. \\
	(iii) Since $H$ is $\Orth(M_{\BR})$-invariant, the maps $\varphi, \psi$ both commute with pullback by $g$. This shows that $\LO$ and $\RO$ are $\Orth(M_{\BR})$-invariant. \\
	(iv) By \eqref{phi} the lowering operator is
\[ 				\LO(F) 
				= |Z|^2 \sum_{I, k, i}
				\frac{\partial F_{I}}{\partial \overline{Z}_k} \overline{R}_{k i} \, dZ_I \otimes dZ_i  \]
    which implies the claim if we use the definition of $R_{ki}$.    
	For the raising operator, if $f$ is a section of $\CL^k$ then
	\[
	\RO(f) = |Z|^{-2k} \partial( |Z|^{2k} f ) = \partial(f) + k f \partial(\log |Z|^2 ) = \partial(f) + k f \nu.
	\]
	If $\omega = \sum_i a_i(Z) dZ_i$ is a section of $\CE$, then
	\[ \RO(\omega) = \varphi \partial \psi(\omega)
	= \sum_{i,j,m} R_{ij} \overline{R}_{jm} dZ_m \otimes \partial(a_i)
	+ \sum_{i,j,m} a_i dZ_m \otimes \partial(R_{ij}) \overline{R}_{jm}. \]
	In the first summand we insert \eqref{RR} and use 
\[ \sum_i \partial(a_i) Z_i = \partial( \sum_i a_i(Z) Z_i ) - \sum_i a_i dZ_i = - \sum_i a_i dZ_i = -\omega \quad (\text{since } \iota_{v}(\omega)=0) \]
which gives
\[ \sum_{i,j,m} R_{ij} \overline{R}_{jm} dZ_m \otimes \partial(a_i) = \sum_{i} dZ_i \otimes \partial(a_i) + \nu \otimes \sum_k a_k dZ_k = \partial(\omega) + \nu \otimes \omega. \]
 For the second term we use the identity (valid for $Z^2=0$ and modulo $dZ^2$)
	\[
	\sum_{j} \partial(R_{ij}) \overline{R}_{jm} = 
 -\frac{1}{2 |Z|^2} \overline{Z}_i \sum_{b} g_{mb} dZ_b
 + \frac{1}{2 |Z|^2} \overline{Z}_i \sum_{b} g_{mb} Z_b \, \nu
	\]
This gives
\[
\sum_{i,j,m} a_i dZ_m \otimes \partial(R_{ij}) \overline{R}_{jm}
=
-\frac{1}{2|Z|^2} \sum_{i} a_i \overline{Z}_i \sum_{m,b} g_{mb} dZ_b \otimes dZ_m = - \tr(\nu \otimes \omega) g. \]
Thus we get $\RO(\omega) = \delta(\omega) + \nu \otimes \omega - \tr(\nu \otimes \omega) g$ as desired.

The general case is essentially a combination of these two cases and straightforward to check. For $F=\sum_{I} F_I dZ_I$ an explicit formula in coefficients for $\RO(F)$ is given by
\begin{multline} \label{formula for R}
	\RO(F)  = \partial(F)
		\ +\  k F \otimes \nu 
  + \sum_{r=1}^{s} \sum_{I} F_I dZ_{i_1} \otimes \cdots \otimes dZ_{i_{r-1}} \otimes \nu \otimes dZ_{i_{r+1}} \otimes \cdots \otimes dZ_{i_s} \otimes dZ_{i_r} \\
  - \sum_{r=1}^{s} \sum_{I} \sum_{a,b} \frac{F_I \overline{Z}_{i_{r}}}{2 |Z|^2} g_{ab} 
  dZ_{i_1} \otimes \cdots \otimes dZ_{i_{r-1}} \otimes dZ_a \otimes dZ_{i_{r+1}} \otimes \cdots \otimes dZ_{i_s} \otimes dZ_b
  \qedhere
\end{multline}
\end{proof}

The commutator of $\LO$ and $\RO$ is defined as
\[ [\LO, \RO] := \LO \circ \RO - \sigma_{s+1, s+2} \circ \RO \circ \LO : \CA(U, \CE^{\otimes s} \otimes \CL^k)
\to \CA(U, \CE^{\otimes s+2} \otimes \CL^k), \]
where $\sigma_{ij}$ again stands for the transposition.
\begin{prop} \label{prop:commutation relation}
For $F \in \CA(U, \CE^{\otimes s} \otimes \CL^k)$ we have the commutator formula
\[ {[} \LO, \RO ] F = \frac{k}{2} F \otimes g
+ \frac{1}{2} \sum_{r=1}^{s} \sigma_{r, s+1}( F \otimes g) - \sigma_{r, s+2}(F \otimes g). \]
\end{prop}
\begin{proof}
We define ``raising" and ``lowering" operators on differentials on $\A(\CD)$ by the formulas 
\eqref{formula for R} and \eqref{formulaL},
and compute the commutator $[\LO,\RO]F$ directly on $\A(\CD)$.
This is lengthy but straightforward and omitted here.
In the computation one uses Lemma~\ref{lemma:basic evaluation}(ii) and the fact that for any $F = \sum_I F_I dZ_I \in \CA(U, \CE^{\otimes s} \otimes \CL^k)$ the function $F_I$ on the affine cone is homogeneous of bidegree $(-k,0)$, and therefore satisfies Euler's relation
$\sum_{\ell} \frac{\partial F_I}{\partial \overline{Z}_{\ell}} \overline{Z}_{\ell} = 0$.
\end{proof}
\begin{lemma} \label{lemma:basic evaluation} In $\CA(\Omega_{\A(\CD)}^{\otimes 2})$ we have
$\LO( |Z|^2 ) = 0$ and $\LO(\nu) = \frac{1}{2} g$.
\end{lemma}
\begin{proof}
For the first note that
\[ \LO( |Z|^2 ) = \frac{1}{4} |Z|^2 d(Z^2) - \frac{1}{2} (Z \cdot Z) \partial(|Z|^2) \]
which is zero since $Z \cdot Z=0$ and $dZ^2=0$ in $\Omega_{\A(\CD)}$.
Hence we get
\[ \LO(\nu) = \LO\left( \frac{\partial(|Z|^2)}{|Z|^2} \right) = \frac{|Z|^2} \LO(\partial(|Z|^2)) = \frac{g}{2} - \frac{1}{4 |Z|^2} d(Z^2) \otimes \partial(|Z|^2) = \frac{g}{2}. \qedhere \]
\end{proof}

\begin{example}
If $f$ is a section of $\CL^k$, then
\[ [\LO,\RO] f = \frac{k}{2} f g. \]
If $\omega$ is a section of $\CE$ then we have
\[ [\LO,\RO](\omega) = \frac{1}{2} \sigma_{23}(g \otimes \omega) - \frac{1}{2} g \otimes \omega. \]
\end{example}

\begin{lemma} \label{lemma:L^2 and R^2}
Let $F = \sum_{I = (i_1, \ldots, i_s)} F_I dZ_I$, where $dZ_I = dZ_{i_1} \otimes \cdots \otimes dZ_{i_s}$. Then we have
\begin{align*}
\LO^2(F) = & \sum_{I, k,\ell} \frac{\partial^2 F_I}{\partial \overline{Z}_k \partial \overline{Z}_{\ell}}
dZ_I \otimes \left( |Z|^2 dZ_k - Z_k \partial(|Z|^2) \right)
\otimes \left( |Z|^2 dZ_{\ell} - Z_{\ell} \partial(|Z|^2) \right) \\
& - \frac{1}{2} |Z|^2 \sum_{I,k} \frac{\partial F_I}{\partial \overline{Z}_k} Z_k dZ_I \otimes g.
\end{align*}
In particular, $\LO^d(F)$ and $\RO^d(F)$ are symmetric in the last $d$ factors.
\end{lemma}
\begin{proof}
The first claim follows by a straightforward calculation from \eqref{formulaL} using the relations $dZ^2=0$ and $Z^2=0$.
By inspection $\LO^2(F)$ is hence symmetric, so $\LO^d(F)$ is symmetric. Moreover, by Proposition~\ref{prop:Properties of L and R}(ii) and an induction argument we have
\[ \RO^d(F) = |Z|^{-2k-2d} \varphi^{\otimes (s+d)} \overline{\LO^d\Big( |Z|^{2k} \varphi^{\otimes s} \overline{F}\Big)} \]
which shows that $\RO^d(F)$ is also symmetric.
\end{proof}

Consider the scalar-valued raising and lowering operators
\begin{gather}
\CR := \tr_{12} \circ \RO \circ \RO : \CA( \CD, \CL^k ) \to \CA( \CD, \CL^{k+2}), \notag \\
\CL := \tr_{12} \circ \LO \circ \LO : \CA( \CD, \CL^k ) \to \CA( \CD, \CL^{k-2}). \label{Zemel L}
\end{gather}

The following formulas show that $\CR, \CL$ are the raising and lowering operators that were
defined by Zemel \cite{Zemel}.
\begin{cor}
For $f \in \CA(\CD, \CL^k)$ we have
\[ \CR(f) = \sum_{i,j} \frac{\partial^2 f}{\partial Z_i \partial Z_j} g^{ij}
+ (2k+2-n) \frac{1}{2|Z|^2} \sum_{a} \overline{Z}_k \frac{\partial f}{\partial Z_k}. \]
\[ \CL(f)  = |Z|^4 \sum_{i,j} g^{ij} \frac{\partial^2 f}{\partial \overline{Z}_i \partial \overline{Z}_j} 
+ (2-n) \frac{1}{2} |Z|^2 \sum_{k} \frac{\partial f}{\partial \overline{Z}_k} Z_k.
\]
\end{cor}

\begin{proof}
We have $\RO(f) = \partial(f) + k \nu f \in \CA(\CL^{k+1} \otimes \CE)$ and therefore
\begin{equation} \label{eqn:abcdef}
	\begin{aligned}
		\RO^2(f) & = \partial^2(f) + k \partial(\nu) f + k \nu \otimes \partial(f) \\
		& + (k+1) \partial(f) \otimes \nu + k (k+1) f \nu \otimes \nu \\
		& - \tr(\partial(f) \otimes \nu) g - k f \tr(\nu, \nu) g \\
		& + \nu \otimes \partial(f) + k f \nu \otimes \nu.
	\end{aligned}
\end{equation}
Observe that
\begin{align*}
	\tr(\nu \otimes \nu) & = 0 \quad (\text{since } \overline{Z}\cdot \overline{Z} = 0) \\
	\tr( \partial(\nu)) & = 0 \quad (\text{since } \partial(\nu) = - \nu \otimes \nu) \\
	\tr( \nu \otimes \partial(f)) & = \frac{1}{2 |Z|^2} \sum_{a} \overline{Z}_a \frac{\partial f}{\partial Z_a} \quad (\text{direct computation})\\
	\tr(g) & = n \quad (\text{Remark~\ref{rmk:Warning}}).
\end{align*}
We conclude that
\[ \tr(\RO^2(f)) = \sum_{i,j} \frac{\partial^2 f}{\partial Z_i \partial Z_j} g^{ij}
+ (2k+2-n) \tr( \partial(f), \nu ). \]

For the lowering operator, taking the trace in Lemma~\ref{lemma:L^2 and R^2} yields
\begin{align*}
\CL(f)  & = |Z|^4 \sum_{i,j} g^{ij} \frac{\partial^2 f}{\partial \overline{Z}_i \partial \overline{Z}_j} 
- \frac{n}{2} |Z|^2 \sum_{k} \frac{\partial f}{\partial \overline{Z}_k} Z_k
- |Z|^2 \frac{1}{2} \sum_{i,j} \frac{\partial^2 f}{\partial \overline{Z}_i \partial \overline{Z}_j} (\overline{Z}_i Z_j + Z_i \overline{Z}_j).
\end{align*}
To simplify the last term, note that the function $f$ viewed on the affine cone $\A(\CD)$ is of bi-degree $(k,0)$, so $f_k:=\partial f/\partial \overline{Z}_k$ is of bidegree $(k,-1)$, hence satisfies the Euler equation
\[ \sum_{\ell} \frac{\partial f_k}{\partial \overline{Z}_{\ell}} \overline{Z}_{\ell} = - f_k. \]
So $$\sum_{i,j} \frac{\partial^2 f}{\partial \overline{Z}_i \partial \overline{Z}_j} (\overline{Z}_i Z_j + Z_i \overline{Z}_j) = -2 \sum_{j} Z_J \frac{\partial f}{\partial \overline{Z}_j}.$$
%
\end{proof}

\subsection{Almost-holomorphic modular forms}
Let $U \subset \CD$ be an open.
A section 
\[ F \in \CA(U, \CE^{\otimes s} \otimes \CL^k) \]
is called {\em almost-holomorphic of depth $d$} if
$\LO^{d}(F) \neq 0$ and $\LO^{d+1}(F) = 0$.
We say that $F$ is almost-holomorphic if it is so of some depth $d$.

\begin{defn}
Let $\Gamma$ be a finite index subgroup of $\Orth(M)$
and assume that $n \geq 3$.
An almost-holomorphic modular form of weight $k$ and rank $s$ for $\Gamma$ is an almost-holomorphic $\Gamma$-invariant section of $\CL^k \otimes \CE^{\otimes s}$.
We write $\AHMod_{k,s}(\Gamma)$ for the space of almost-holomorphic modular forms of weight $k$ and rank $s$.
\end{defn}

Since $\LO$ is a derivation, $\AHMod_{*,*}(\Gamma)$ is a bigraded ring: if $F$ and $G$ are almost-holomorphic of depths $d_F$ and $d_G$, then $F \otimes G$ is almost-holomorphic of depth $d_F+d_G$.

We refer to Section~\ref{subsec:general cusp} for the additional cusp condition that needs to be required in case $n \leq 2$.

\begin{prop}
The raising and lowering operators act on the space of almost-holomorphic modular forms for $\Gamma$:
\begin{gather*}
	\RO : \AHMod_{k,s}(\Gamma) \to \AHMod_{k+1,s+1}(\Gamma) \\
	\LO : \AHMod_{k,s}(\Gamma) \to \AHMod_{k-1,s+1}(\Gamma).
\end{gather*}
\end{prop}
\begin{proof}
By Proposition~\ref{prop:commutation relation}, if a section $F$ of $\CE^{\otimes s} \otimes \CL^k$ is almost-holomorphic, then $\RO(F)$ is again almost-holomorphic.
Moreover, by Proposition~\ref{prop:Properties of L and R}(ii) we have that if 
$F$ is $\Gamma$-invariant, then so is $\RO(F)$ and $\LO(F)$.
\end{proof}

\subsection{Logarithmic almost-holomorphic modular forms}
\label{subsec:log ahol modular forms}
Consider a (reduced)
divisor $\CH \subset \CD$ which is the vanishing locus of a holomorphic modular form $f \in \Mod_{\ell,0}(\Gamma)$ for some group $\Gamma$ of some positive weight $\ell>0$:
\[ \CH = \{ Z \in \CD \, |\, f(Z) = 0 \}. \]
We do not require $f$ to vanish along $\CH$ to first order.

\begin{defn}
A {\em logarithmic modular form} on $(\CD,\CH)$ of weight $k$ and rank $s$ for $\Gamma$ is a meromorphic $\Gamma$-invariant section of $\CL^{\otimes k} \otimes \CE^{\otimes s}$ which has a pole of order $\leq k$ along $\CH$ and no other poles.
We let $\Mod_{k,s}^{\log}(\Gamma)$ be the vector space of these forms.
\end{defn}

In other words, logarithmic modular forms behave roughly like derivatives of the logarithm of a function on $\mathcal{D}$ that vanishes on $\mathcal{H}$. (This was the motivation for the name.)

\begin{example}(\cite[Sec.5.2]{GKMW})
A basic construction of logarithmic modular forms on $(\CD,\CH)$ is
\[
\epsilon(f) := \frac{\RO^2(f)}{f} - \frac{\ell+1}{\ell} \frac{\RO(f) \otimes \RO(f)}{f^2} + \frac{1}{2 \ell} \frac{\tr(\RO(f) \otimes \RO(f))}{f^2} g.
\]
where $f$ is a (scalar-valued) modular form of weight $\ell$.
This defines a section of $\CL^{2} \otimes \CE^{\otimes 2}$ invariant under $\Gamma$. By expanding the raising operator terms using \eqref{eqn:abcdef}, we see this is given by
\[ \epsilon(f) = \frac{\partial^2(f)}{f} - \frac{\ell+1}{\ell} \frac{\partial(f) \otimes \partial(f)}{f^2} + \frac{1}{2 \ell} \frac{\tr( \partial(f) \otimes \partial(f))}{f^2} g \]
so this is holomorphic except for poles of order $\leq 2$ along $\CH$. Hence $\epsilon \in \Mod^{\log}_{2,2}(\Gamma)$.
\end{example}

We want to extend this notion to almost-holomorphic modular forms.
For this we first define the pole order of an almost-holomorphic function along $\CH$:

\begin{defn} An almost-holomorphic function $F$ defined on $\CD\setminus \CH$
has {\em pole order $\leq d$ along $\CH$} if for any $x \in \CH$, neighborhood $U_x$ of $x$, and holomorphic function $h_x : U_x \to \BC$ which vanishes on $\CD \cap U_x$ to first order and nowhere else (i.e. a defining equation for $\CD \cap U_x$), 
the function $|F \cdot h_x^d|$ defined on $U_x \setminus (U_x \cap \CH)$ extends to a smooth function on $U_x$.
\end{defn}

Logarithmic almost-holomorphic modular forms are then defined as follows:

\begin{defn}\label{defn:log}
A logarithmic almost-holomorphic modular form on $(\CD, \CH)$ of weight $k$ and rank $s$ for $\Gamma$ is a section $F \in \CA(\CD \setminus \CH, \CL^k \otimes \CE^{\otimes s})$ such that:
\begin{enumerate}
	\item[(i)] $F$ is almost-holomorphic;
	\item[(ii)] $F$ is $\Gamma$-invariant: $\gamma^{\ast}(F) = F$ for all $\gamma \in \Gamma$;
	\item[(iii)] $\LO^d(F)$ has a pole of order $\leq k-d$ along $\CH$ for all $d \geq 0$.
\end{enumerate}
\end{defn}

Let $\AHMod_{k,s}^{\log}(\Gamma)$ be the vector space
of weight $k$ rank $s$ logarithmic almost-holomorphic modular forms for $\Gamma$.
Its immediate to check that the raising and lowering operators act on the spaces of almost-holomorphic logarithmic modular forms.

\begin{example}
The simplest example of a logarithmic almost-holomorphic modular form
is the ``logarithmic derivative" of $f$,
\[ G := \frac{1}{\ell} \frac{\RO(f)}{f} \in \AHMod_{1,1}^{\log}(\Gamma), \]
where $\ell$ is the weight of $f$. The lowering operator gives
$\LO(G) = \frac{1}{2} g$.
\end{example}

We prove the following simple structure theorem,
relating logarithmic almost-holomorphic modular forms
to vector-valued modular forms of larger rank:

\begin{prop} \label{structure prop log mod forms} We have a natural isomorphism
\[ \AHMod^{\log}_{k,s}(\Gamma) \cong \bigoplus_{d=0}^{k} \Mod_{k-d,s+d}^{\log}(\Gamma)^{S_d} \]
where $\Mod_{k-d,s+d}^{\log}(\Gamma)^{S_d} \subset \Mod_{k-d,s+d}^{\log}(\Gamma)$ is the subspace of rank $s+d$ forms invariant under the $S_d$ action permuting the last $d$ factors.
\end{prop}
\begin{proof}
There is a natural map
\[ \alpha_d : \Mod_{k-d,s+d}^{\log}(\Gamma)^{\mathrm{sym}} \to \AHMod^{\log}_{k,s}(\Gamma) \]
given by
$F_d \mapsto \left\langle F_d, G^{\otimes d} \right\rangle$,
where we use the notation introduced in \eqref{eqn:defn of pairing}.

Since $F_d \in \Mod_{k-d,s+d}^{\log}(\Gamma)^{\mathrm{sym}}$ is holomorphic and $S_d$-invariant in the last $d$ factors we have
\[
\LO^{d}\left( \left\langle F_d, G^{\otimes d} \right\rangle \right) = 
\frac{d!}{2^d} \tr_{(s+1, s+d+1), (s+2, s+d+3), \ldots ,(s+d,s+d+2d-1)}( F_d \otimes g^{\otimes d})
= \frac{d!}{2^d} F_d
\]

Consider the sum of these maps over $d$,
\[ \alpha=\sum_{d} \alpha_d : \bigoplus_{d=0}^{k} \Mod_{k-d,s+d}^{\log}(\Gamma)^{\mathrm{sym}} \to \AHMod^{\log}_{k,s}(\Gamma). \]

We show that $\alpha$ is injective: If $\alpha(F_0,\ldots, F_k) = 0$, then if it exists, let $k_0$ be the largest number such that $F_{k_0} \neq 0$. Then $\LO^{k_0}( \alpha(F_0,\ldots, F_k) )$ is a non-zero multiple of $F_{k_0}$, so $F_{k_0} = 0$, which is a contradiction. Hence $\alpha(F_0,\ldots, F_k)=0$.
To show that $\alpha$ is surjective, let $F \in \AHMod^{\log}_{k,s}(\Gamma)$, let $d$ be its depth and let $F_d = \LO^d(F)$. Then $F_d$ is meromorphic (with poles of order at most $\leq k-d$ since $F$ is logarithmic), and by Lemma~\ref{lemma:L^2 and R^2} invariant under the permutation of the last $S_d$-factors, so lies in $\Mod_{k-d,s+d}^{\log}(\Gamma)^{\mathrm{sym}}$.
Since there are no holomorphic vector-valued modular forms of negative weight by \cite[Prop.4.4]{Ma},
we must have moreover $d \leq k$.
Then by the above
\[ F' := F - \frac{2^d}{d!} \left\langle F_d, G^d \right\rangle \]
has depth $d-1$, so the claim follows by induction on the depth.
\end{proof}

\begin{rmk}
The proof shows that a logarithmic almost-holomorphic modular form of weight $k$ has depth at most $k$.
\end{rmk}

\begin{example}(Serre derivative, see also \cite[Sec.5.2]{KMPS})
The derivative of a classical modular forms is no longer modular, but transforms with a correction term.
By adding a multiple of the second Eisenstein series, one can construct a derivative operator on modular forms, called the Serre derivative \cite{123}.
We can introduce here a Serre derivative on logarithmic orthogonal modular forms.

As before denote the (almost-holomorphic) ``logarithmic derivative" of $f$ with
\[ G := \frac{1}{\ell} \frac{\RO(f)}{f} = \frac{1}{\ell} \frac{\partial(f)}{f} + \nu, \]
where $\ell$ is the weight of $f$, and the holomorphic logarithmic derivative with
\[
G_{\mathrm{hol}} = \frac{1}{\ell} \frac{\partial(f)}{f}.
\]
We define the Serre derivative
\[
\partial_S: \Mod^{\log}_{k,s}(\Gamma) \to \Mod^{\log}_{k+1,s+1}
\]
by setting for $F \in \Mod^{\log}_{k,s}(\Gamma)$:
\begin{align*}
	\partial_S(F) & = \RO(F) - k F \otimes G + \sum_{r=1}^{s} \tr_{s+2,s+3} \sigma_{r, s+2}( F \otimes g \otimes G ) - \sum_{r=1}^{s} \sigma_{r, s+1}(F \otimes G) \\
	& = \partial(F) - k F \otimes G_{\mathrm{hol}}
	+ \sum_{r=1}^{s} \tr_{s+2,s+3} \sigma_{r, s+2}( F \otimes g \otimes G_{\mathrm{hol}} ) - \sum_{r=1}^{s} \sigma_{r, s+1}(F \otimes G_{\mathrm{hol}}).
\end{align*}
The first line above shows that $\partial_S$ sends section of $\CL^k \otimes \CE^s$ to sections of $\CL^k \otimes \CE^{s+1}$ and preserves $\Gamma$-invariance. The second shows that it preserves holomorphicity (away from the divisor $\CH$) and increases the pole order at most by $1$.
Hence $\partial_S$ is well-defined.

This gives another way to construct new logarithmic modular forms out of old ones. For example, for $k \ge 2$ we have the forms
$$\epsilon_k(f) = \partial_S^{k-2}(\epsilon(f)) \in \Mod^{\log}_{k,k}(\Gamma).$$ (Note that $\partial_S(f) = 0$ by construction.)
\end{example}

\section{Theta lifts and almost-holomorphic modular forms}
\label{sec:theta lifts}

In this section we first review the construction of the regularized theta lift as introduced by Borcherds \cite{Borcherds1998}.
We then show that 
the theta kernel behaves well with respect to the raising and lowering operators on both $\SL_2(\BR)$ and $\mathrm{O}(2, n)$, which implies that theta lifts are also well-behaved (in a sense to be made precise later).
This leads to the characterizations of those almost-holomorphic modular forms whose lift is again almost-holomorphic (Theorem~\ref{thm:list of almost-holomorphic modular form}).

\subsection{Metaplectic group}
The metaplectic group $\mathrm{Mp}_2(\mathbb{R})$ is the unique connected double cover of $\SL_2(\BR)$. It is usually constructed as the group of pairs $(M, \phi(\tau))$, where $M = \begin{pmatrix} a & b \\c & d \end{pmatrix} \in \mathrm{SL}_2(\RR)$ and $\phi(\tau)$ is a holomorphic square root of $\tau \mapsto c \tau + d$ on the upper half-plane.
The group structure on $\mathrm{Mp}_2(\mathbb{R})$ is defined by
$$(M_1,\phi_1(\tau)) \cdot (M_2, \phi_2(\tau)) = (M_1 M_2, \phi_1(A_2 \tau) \phi_2(\tau).$$

The integral metaplectic group $\mathrm{Mp}_2(\mathbb{Z})$ is 
the preimage of $\SL_2(\BZ)$ under the projection $\mathrm{Mp}_2(\mathbb{R}) \to \SL_2(\BR)$. It is generated by the elements
\begin{align*}
S &= \left( \begin{pmatrix} 0 & -1 \\ 1 & 0 \end{pmatrix}, \sqrt{\tau} \right), \\
T & = \left( \begin{pmatrix} 1 & 1 \\ 0 & 1 \end{pmatrix}, +1 \right),
\end{align*}
where the square root $\sqrt{\tau}$ has positive imaginary part, modulo the relations 
\[ S^8 = (ST)^6 = \id. \]

\subsection{The Weil representation and theta functions}

Let $M$ be an even integral lattice of rank $n$ and signature $\sigma$.
We consider two representations assoiated to the lattice $M$, the finite Weil representation and the unitary Weil representation.

\subsubsection{Finite Weil representation}
The finite Weil representation $\rho$ associated to $M$ is the representation of $\mathrm{Mp}_2(\mathbb{Z})$ on the group ring $\mathbb{C}[M'/M] = \mathbb{C}[e_x: \, x \in M'/M]$, defined 
by 
\begin{align*} 
\rho(S) e_x & = \frac{e^{-\pi i \sigma / 4}}{\sqrt{|M'/M|}} \sum_{y \in M'/M} e^{-2\pi i \langle x, y \rangle} e_y \\ \rho(T) e_x & = e^{\pi i \langle x, x \rangle} e_x.
\end{align*}

If we define the discrete Fourier transform 
\[ \CF(e_{\gamma}) = \frac{1}{\sqrt{|M'/M|}} \sum_{\beta} e^{- 2 \pi i \langle \gamma, \beta\rangle} e_{\beta}, \] which satisfies $\CF^2 e_{\gamma} = e_{-\gamma}$, then
$\rho(S) = e^{- \pi i \sigma/4} \CF$. Hence $\rho(S^2) = \rho( -I, i) e_{\gamma} = e^{- \pi\sigma/2} e_{\gamma}$
and thus $\rho(S^4) = \rho(I,-1)$ acts by multiplication by $e^{-\pi i \sigma} = (-1)^{\rk M}$. In particular, the finite Weil representation factors through $\SL_2(\BZ)$ for $M$ of even rank.

\subsubsection{Unitary Weil representation}
The unitary Weil representation is the unitary representation $\omega$ of $\mathrm{Mp}_2(\mathbb{R})$ on $L^2(M_{\RR})$ 
defined on the usual generators of $\mathrm{Mp}_2(\mathbb{R})$ by \begin{align*} \omega \Big( \begin{pmatrix} t & 0 \\ 0 & t^{-1} \end{pmatrix}, \pm \sqrt{t} \Big) f(x) &= (\pm \sqrt{t})^n f(tx); \\ \omega \Big( \begin{pmatrix} 1 & b \\ 0 & 1 \end{pmatrix}, 1 \Big) f(x) &= e^{\pi i b \langle x, x \rangle} f(x);\end{align*} 
and $\omega(\begin{pmatrix} 0 & -1 \\ 1 & 0 \end{pmatrix}, \sqrt{\tau})$ is essentially the Fourier transform: if $f \in L^1 \cap L^2$ then 
$$\omega \Big( \begin{pmatrix} 0 & -1 \\ 1 & 0 \end{pmatrix}, \sqrt{\tau} \Big) f(x) = e^{\pi i \sigma / 4} \hat f(-x) = \frac{e^{\pi i \sigma / 4}}{\sqrt{|M'/M|}} \int_{M_{\RR}} f(y) e^{-2\pi i \langle x, y \rangle} \, \mathrm{d}y.$$

The representation $\omega$ descends to a representation of $\mathrm{SL}_2(\mathbb{R})$ if and only if $n$ is even.

In any coordinates $(x_i)$ on $M \otimes \mathbb{R}$, if $g_{ij}$ are the entries of the Gram matrix and $g^{ij}$ its inverse, the induced representation of the Lie algebra $\mathfrak{sl}_2(\RR)$ is determined by \begin{align*} \omega(H) f(x) &= \Big(\frac{n}{2} + \sum_i x_i \partial_i\Big) f(x); \\ \omega(X) f(x) &= \pi i \langle x, x \rangle \cdot f(x); \\ \omega(Y) f(x) &= -\frac{1}{4\pi i} \sum_{i,j} g^{ij} \frac{\partial^2 f}{\partial x_i \partial x_j},\end{align*} where $X = \begin{pmatrix} 0 & 1 \\ 0 & 0 \end{pmatrix}$, $Y = \begin{pmatrix} 0 & 0 \\ 1 & 0 \end{pmatrix}$, $H = \begin{pmatrix} 1 & 0 \\ 0 & -1 \end{pmatrix}$. These expressions are basis-independent.

\subsubsection{Theta functions}
Let $f : M_{\RR} \rightarrow \mathbb{C}$ be a Schwartz function. Define 
the vector-valued theta function attached to $M$ by $$\Theta(f; \tau) = y^{n/4 - k/2} \sum_{\lambda \in M'} f(\lambda \sqrt{y}) e^{\pi i x \langle \lambda, \lambda \rangle} \mathfrak{e}_{\lambda + M}.$$
The following result of Vign\'eras \cite{Vigneras1977} describes its behavior under the finite Weil representation:

\begin{prop}[Vign\'eras]
If $f$ satisfies the differential equation $\omega(Y - X) f = -ik \cdot f$ for some $k \in \frac{1}{2}\mathbb{Z}$, then 
for every $\gamma \in \mathrm{Mp}_2(\mathbb{Z})$ we have $$\Theta(f; \gamma \cdot \tau) = \rho(\gamma) j(\gamma; \tau)^k \Theta(f; \tau).$$
\end{prop}
\begin{rmk}
Explicitly, the equation $\omega(Y - X)f = -ik \cdot f$ is equivalent to
$$\Delta f(x) = -2\pi k f(x) + 2\pi^2 \langle x, x \rangle f(x).$$ Here $\Delta = \frac{1}{2} \sum_{i,j} g^{ij} \partial_i \partial_j$ is the holomorphic Laplace operator.
\end{rmk}

\begin{proof}
We give a proof for the reader's convenience.
Taking exponentials shows that the differential equation is equivalent to $\omega(K) f(x) = e^{-ik \theta} f(x)$ for any $K = \begin{pmatrix} \cos(\theta) & -\sin(\theta) \\ \sin(\theta) & \cos(\theta) \end{pmatrix}$ in $\mathrm{SO}_2(\RR)$. (In the odd-rank case, one has to work in the metaplectic group and $k$ is necessarily half-integral, but we ignore that here by abuse of notation.)
We can then write $$\Theta(f; \tau) = y^{-k/2} \sum_{\lambda \in M'} \omega(\tau) f(\lambda) \mathfrak{e}_{\lambda + M}$$ where $\tau =x+iy \in \mathbb{H}$ also denotes the associated matrix $\begin{pmatrix} 1 & x \\ 0 & 1 \end{pmatrix} \cdot \begin{pmatrix} y^{1/2} & 0 \\ 0 & y^{-1/2} \end{pmatrix}$ in $\mathrm{SL}_2(\RR)$. Note that if $\tau = |\tau| e^{i \theta}$ then $-\tau S \tau = \begin{pmatrix} \cos(\theta) & -\sin(\theta) \\ \sin(\theta) & \cos(\theta) \end{pmatrix} \in \mathrm{SO}_2(\mathbb{R})$.

The Fourier transform of $\omega(\tau) f$ is therefore $$e^{-\pi i \sigma / 4} \omega(S) \omega(\tau) f(-x) = e^{-ik \theta - \pi i \sigma / 4} \omega(-1/\tau) f(-x).$$ By Poisson summation, we then obtain
\[\Theta(f; -1/\tau) = y^{-k/2} |\tau|^k e^{ik \theta} \rho(S)\Theta(f; \tau) = \tau^k \rho(S) \Theta(f; \tau). \qedhere \]
\end{proof}

\subsubsection{{Maa\ss\, operators}}
The Maa\ss\, raising and lowering operators, defined for vector-valued functions and half-integral weight $k$ by the usual formulas \eqref{def:Maas raisin and lowering operators},
are compatible with the (finite) Weil representation in the sense that 
\begin{align*}
\Big( \RM_k f \Big) \Big|_{k+2, \rho} \gamma & = \RM_k \Big( f \Big|_{k, \rho} \gamma \Big) \\
\Big(\LM_k f \Big) \Big|_{k - 2, \rho} \gamma & = \LM_k \Big( f \Big|_{k, \rho} \gamma \Big)
\end{align*}
for any $\gamma \in \mathrm{Mp}_2(\mathbb{Z})$, where $f |_{k, \rho} \gamma = (c \tau + d)^{-k} \rho(\gamma)^{-1} f(\tau).$

On theta functions, the action of the Maa{\ss} operators is 
\begin{equation}\LM_k \Theta(f; \tau) = \Theta(\omega(\mathfrak{L}) f; \tau), \quad \RM_k \Theta(f; \tau) = \Theta(\omega(\mathfrak{R}) f; \tau), \label{LR on Theta} \end{equation}
where 
\[ \mathfrak{L} = \begin{pmatrix} 1/2 & -i/2 \\ -i/2 & -1/2 \end{pmatrix}, \quad \quad \mathfrak{R} = \begin{pmatrix} 1/2 & i/2 \\ i/2 & -1/2 \end{pmatrix} \in \mathfrak{sl}_2(\CC). \]
These matrices satisfy $[Y-X, \mathfrak{L}] = 2i \mathfrak{L}$ and $[Y-X, \mathfrak{R}] = -2i \mathfrak{R}$, so if $f$ satisfies the Vign\'eras equation with weight $k$ then $\omega(\mathfrak{L})f$ and $\omega(\mathfrak{R})f$ satisfy the equation with weights $k-2$ and $k+2$, respectively.

\subsection{The theta kernel}
Suppose now that $M$ has signature $(2, n)$ and rank $n+2$ and let $\A(\CD)$ be the affine cone over the period domain $\CD$. For any $Z \in \A(\CD)$ and $x \in M_{\RR}$, we write $x_Z$ for the orthogonal projection of $x$ to the positive-definite space in $M_{\RR}$ spanned by $\mathrm{re}(Z)$ and $\mathrm{im}(Z)$, and $x_Z^{\perp} = x - x_Z$ for the projection of $x$ to its orthogonal complement. Hence $$x_Z = \frac{\langle x, Z \rangle \overline{Z} + \langle x, \overline{Z} \rangle Z}{\langle Z, \overline{Z} \rangle}$$ and $$\langle x_Z, x_Z \rangle = 2 \frac{\langle x, Z \rangle \langle x, \overline{Z} \rangle}{\langle Z, \overline{Z} \rangle} = \frac{|\langle x, Z \rangle|^2}{|Z|^2}.$$

For any $Z \in \A(\CD)$, define $\Theta_k(\tau, Z)$ to be the theta function constructed from the Schwartz function $$f_k = f_{k, Z} : M_{\RR} \rightarrow \mathbb{C} \quad f_k(x) = \langle x, Z \rangle^k e^{\pi \langle x_Z^{\perp}, x_Z^{\perp} \rangle - \pi \langle x_Z, x_Z \rangle}.$$ So \begin{equation} \label{thetea function}\Theta_k(\tau, Z) = y^{n/2} \sum_{\lambda \in M'} \langle \lambda , Z \rangle^k e^{\pi i \tau \langle \lambda_Z, \lambda_Z \rangle + \pi i \overline{\tau} \langle \lambda_Z^{\perp}, \lambda_Z^{\perp} \rangle} e_{\lambda + M}.\end{equation}

The action of the infinitesimal Weil representation on $f_k$ is by 
\begin{align*} \omega(H) f_k &= \Big( \frac{n}{2} + 1 + k + 2\pi \langle x, x \rangle - 8\pi \frac{\langle x, Z \rangle \langle x, \overline{Z} \rangle}{\langle Z, \overline{Z} \rangle} \Big) \cdot f_k; \\ \omega(X) f_k &= \pi i \langle x, x \rangle \cdot f_k; \\ \omega(Y) f_k &= i \Big(\pi \langle x, x \rangle - \kappa \Big) \cdot f_k,\end{align*} where $\kappa = k + 1 - n/2$. Therefore, $f_k$ satisfies Vign\'eras' equation $$\omega(Y - X) f_k = -i \kappa f_k$$ and
 $\Theta_k$ transforms as a modular form of weight $\kappa$ in the variable $\tau$ with respect to the finite Weil representation $\rho_M$. With respect to the orthogonal group of $M$, it satisfies $$\Theta_k(\tau, \sigma Z) = y^{n/2} \sum_{\lambda \in M'} \langle \lambda, Z \rangle^k e^{\pi i \tau \langle \lambda_Z, \lambda_Z \rangle + \pi i \overline{\tau} \langle \lambda_Z^{\perp}, \lambda_Z^{\perp} \rangle} e_{\sigma \lambda + M} = \sigma \Theta_k(\tau, Z),$$ where $\mathrm{O}(M)$ acts on $\CC[M'/M]$ by $\sigma e_x = e_{ \sigma x}$. Moreover, $\Theta_k$ is of degree $k$ in the variable $Z$. Therefore, $\Theta_k$ transforms like a modular form of weight $(-k)$ in the variable $Z$ under the discriminant kernel.

The infinitesimal lowering and raising operators $\omega(\mathfrak{L})$, $\omega(\mathfrak{R})$ act on the function $f_k$ in $\Theta_k$ by $$\omega(\mathfrak{L}) f_k = \Big( \frac{n}{2} + 2\pi \langle x,x \rangle - 4\pi \frac{\langle x, Z \rangle \langle x, \overline{Z} \rangle}{\langle Z, \overline{Z} \rangle} \Big) \cdot f_k;$$ $$\omega(\mathfrak{R}) f_k = \Big( k + 1 - 4\pi \frac{\langle x, Z \rangle \langle x, \overline{Z} \rangle}{\langle Z, \overline{Z} \rangle} \Big) \cdot f_k.$$

\begin{prop}\label{prop:lr_maass} The orthogonal and Maa{\ss} raising and lowering operators on $\Theta_k$ satisfy $$\LO_{-k} \Big[ \RM_{\kappa} \Theta_k \Big] = -2\pi \cdot \RO_{-(k+2)}[\Theta_{k+2}]$$
\end{prop}
\begin{proof}
Using \eqref{LR on Theta} it suffices to prove
\[ \LO_{-k}[ \Theta( \omega(\mathfrak{R}) f_k ; \tau) ]
= -2 \pi \RO_{-(k+2)}[ \Theta( f_k ; \tau)] \]
Since the theta function is defined by a lattice sum,
this differential equation follows from the following differential equation for the input function $f_k$:
\[ \LO_{-k}[\omega(\mathfrak{R}) f_k] = -2\pi \RO_{-(k+2)}[f_{k+2}]. \]

To prove the latter equation, we apply the formulas \eqref{formulaL} and \eqref{formulaR} in coordinates (which are valid for arbitrary functions on the affine cone) to compute 
\begin{gather*}
\LO( \langle x, \overline{Z} \rangle) = \frac{1}{2}\Big( \langle Z, \overline{Z} \rangle \langle x, \mathrm{d}Z \rangle - \langle x, Z \rangle \langle \overline{Z}, \mathrm{d}Z \rangle \Big) \\
\partial \Big( \langle x, Z \rangle \Big) =  \langle x, \mathrm{d}Z \rangle, \quad 
\LO(\langle Z, \overline{Z} \rangle) = 0, 
\quad 
\partial \langle Z, \overline{Z} \rangle  = \langle \overline{Z}, \mathrm{d}Z \rangle. \end{gather*}
If we rewrite the function $f_{k}(x) = f_{k,Z}(x)$ as
$$f_k(x) = \langle x, Z \rangle^k \exp \Big( \pi \langle x, x \rangle -4\pi \frac{\langle x, Z \rangle \langle x, \overline{Z} \rangle}{\langle Z, \overline{Z} \rangle} \Big)$$ 
it follows that 
\begin{gather*} \LO_{-k}[f_k] = -\frac{2\pi \langle x, Z \rangle}{\langle Z, \overline{Z} \rangle} \Big( \langle Z, \overline{Z} \rangle \langle x, \mathrm{d}Z \rangle - \langle x, Z \rangle \langle \overline{Z}, \mathrm{d}Z \rangle \Big) \cdot f_k \\
\RO_{-k}[f_k] = \Big( \frac{k \langle Z, \overline{Z} \rangle - 4\pi \langle x, Z \rangle \langle x, \overline{Z} \rangle}{\langle Z, \overline{Z} \rangle^2 \langle x, Z \rangle} \Big) \Big( \langle Z, \overline{Z} \rangle \langle x, \mathrm{d}Z \rangle - \langle x, Z \rangle \langle \overline{Z}, \mathrm{d}Z \rangle \Big) \cdot f_k.
\end{gather*}
Hence we conclude that \begin{align*} \LO_{-k}[\omega(\mathfrak{R}) f_k] &= (k+1)  \LO_{-k}[f_k] - 4\pi \frac{\langle x, Z \rangle}{\langle Z, \overline{Z} \rangle} \LO_{-k}[\langle x, \overline{Z} \rangle] f_k - 4\pi \frac{\langle x, Z \rangle \langle x, \overline{Z} \rangle}{\langle Z, \overline{Z} \rangle} \LO_{-k}[f_k] \\ &= \Big( k + 2 - 4\pi \frac{\langle x, Z \rangle \langle x, \overline{Z} \rangle}{\langle Z, \overline{Z} \rangle} \Big) \Big(-2\pi \frac{\langle x, Z \rangle}{\langle Z, \overline{Z} \rangle} \Big) \Big( \langle Z, \overline{Z} \rangle \langle x, \mathrm{d}Z \rangle - \langle x, Z \rangle \langle \overline{Z}, \mathrm{d}Z \rangle \Big) \cdot f_k \\ &= -2\pi \RO_{-(k+2)}[f_{k+2}]. \qedhere \end{align*}
\end{proof}

This implies the following identity which was first proved by Zemel:
\begin{lem}[{Zemel \cite{Zemel}}]\label{lem:zemel_kernel}
    $\mathcal{L}_{-k}[\Theta_k] = 2\pi \cdot \LM \Theta_{k+2}.$
\end{lem}
\begin{proof}
Since $\LO(\langle \overline{Z}, \mathrm{d}Z \rangle) = \frac{1}{2} \langle Z, \overline{Z} \rangle \cdot g,$ it follows that $$\LO_{-k-1} \LO_{-k}[f_k] = \Big[ \pi \langle x, Z \rangle^2 g + \frac{4\pi^2 \langle x, Z \rangle^2}{\langle Z, \overline{Z} \rangle^2} \Big( \langle Z, \overline{Z} \rangle \langle x, \mathrm{d}Z \rangle - \langle x, Z \rangle \langle \overline{Z}, \mathrm{d}Z \rangle \Big)^{\otimes 2} \Big] \cdot f_k.$$ Taking the trace, it follows that the image of $f_k$ under Zemel's lowering operator is $$\mathcal{L}[f_k] = \Big[ n \pi \langle x, Z \rangle^2 + \frac{4\pi^2 \langle x, Z \rangle^2}{\langle Z, \overline{Z} \rangle} \Big( \langle Z, \overline{Z} \rangle \langle x, x \rangle - 2 \langle x, Z \rangle \langle x, \overline{Z} \rangle \Big) \Big] \cdot f_k.$$
Comparing this with the action of $\omega(\mathfrak{L})$ using \eqref{LR on Theta}, we obtain the claim.
\end{proof}

\subsection{The regularized theta lift}

$M$ continues to be an even integral lattice of signature $(2, n)$.

For a real-analytic $\mathrm{SL}_2(\ZZ)$-invariant function $f : \mathbb{H} \rightarrow \CC$ with at most exponential growth at $\infty$, the regularized integral $\int^{\mathrm{reg}} f(\tau) \, \frac{\mathrm{d}x \, \mathrm{d}y}{y^2}$ is defined as the residue at $s=0$ of the limit $$\lim_{h \rightarrow \infty} \int_{\mathcal{F}_h} f(\tau)\mathbf{E}(\tau; s) \, \frac{\mathrm{d}x \, \mathrm{d}y}{y^2},$$ where $\mathcal{F}_h$ is $\{\tau = x+iy: \, |\tau| > 1,\; |x| < 1/2, \; y < h\}$, and where $$\mathbf{E}(\tau; s) = -\pi^{-s} \Gamma(s) \zeta(2s) \sum_{\substack{c, d \in \mathbb{Z} \\ \mathrm{gcd}(c, d) = 1}} \frac{y^s}{|c \tau + d|^{2s}}$$  is the completed Eisenstein series of weight zero, normalized to have residue $1$ at $s=0$.
Roughly speaking, the integral over the fundamental domain is understood as an iterated integral, in which the integral in the horizontal direction is applied first. For background see \cite[Sec.6]{Borcherds1998}.

For vector-valued modular forms $f$ and $g$ of the same weight $\kappa$, it will be convenient to use the notation $$[f, g] := y^{\kappa} \sum_{\lambda \in M'/M} f_{\lambda}(\tau) \overline{g_{\lambda}(\tau)}.$$ This is then an $\mathrm{SL}_2(\mathbb{Z})$-invariant function.

\begin{defn} 
Let $k \geq 0$ be a non-negative integer and let
\[ F = \sum_{\lambda \in M'/M} F_{\lambda} \mathfrak{e}_{\lambda} \]
be an almost-holomorphic modular form for the Weil representation attached to $M$ of weight $\kappa = k + 1 - n/2$
with no poles in the interior but with possibly a pole at $\infty$ with 
at worst exponential growth.
The theta lift of $F$ is the following normalization of the regularized integral $$\mathrm{Lift}(F)(Z) = \frac{1}{2} \Big( \frac{i}{2} \langle Z, \overline{Z} \rangle \Big)^{-k} \int^{\mathrm{reg}} [ F, \Theta_k ] \, \frac{\mathrm{d}x \, \mathrm{d}y}{y^2},$$
where $\Theta_k(\tau, Z) = \sum_{\lambda \in M'/M} \Theta_{k, \lambda}(\tau, Z) \mathfrak{e}_{\lambda}$ is the theta function defined in \eqref{thetea function}. \qed
\end{defn}

Since $\Theta_k$ transforms like an orthogonal modular form of weight $(-k)$, it follows that $\langle Z, \overline{Z} \rangle^{-k} \overline{\Theta_k}$ and therefore $\Lift(F)$ transforms like a modular form of weight $k$ under the discriminant kernel
\[ \widetilde{O}^{+}(M) = \mathrm{Ker}( O^{+}(M) \to O( M^{\vee}/M ) ). \]
More generally, for all $\gamma \in \mathrm{O}^+(M)$ we have 
\begin{equation} \label{tr property}\Lift(F) | \gamma = \Lift(F | \gamma) \end{equation} where $\gamma$ acts on vector-valued modular forms through the formula $$\Big( \sum_{\beta \in M'/M} f_{\beta} e_{\beta} \Big) \Big| \gamma = \sum_{\beta \in M'/M} f_{\beta} e_{\gamma \beta}.$$

The singularities of the theta lift are given as follows.
Consider the expansion
$$F(\tau) = \sum_{t=0}^d F^{(t)} (2\pi y)^{-t} = \sum_{t=0}^d \sum_{\lambda \in M'/M} \sum_{n \in \mathbb{Q}} c^{(t)}(n, \lambda) q^n (2\pi y)^{-t} \mathfrak{e}_{\lambda}.$$

\begin{prop}\label{prop:singularity}
The theta lift of $F$ defines a real-analytic function with singularities on hyperplanes $\lambda^{\perp}$ for vectors $\lambda \in M'$ of positive norm. Its singularity there is of type
\begin{align*} (2\pi i)^{-k} \langle \lambda, Z \rangle^{-k} \times \Big[ &\sum_{t=0}^{k-1} (k-t-1)! c^{(t)}\Big(-\frac{\langle \lambda, \lambda \rangle}{2}, \lambda + M \Big) \Big( \frac{|\langle \lambda, Z \rangle|^2}{|Z|^2} \Big)^t \\ & \quad + \sum_{t=k}^d \frac{(-1)^{1+k-t}}{(t - k)!} c^{(t)}\Big(-\frac{\langle \lambda, \lambda \rangle}{2}, \lambda + M \Big) \Big( \frac{|\langle \lambda, Z \rangle|^2}{|Z|^2} \Big)^t \ln\Big( 2\pi \frac{|\langle \lambda, Z \rangle|^2}{|Z|^2} \Big)\Big].
\end{align*}
\end{prop}
(A function $f$ is said to have a singularity of type $g$ on a hyperplane $\lambda^{\perp}$ if $f - g$ extends to a real-analytic function on a neighborhood of $\lambda^{\perp}$.)
\begin{proof} This is a special case of Theorem 6.2 of \cite{Borcherds1998}.
\end{proof}

\begin{rmk}
When $d \geq k$, the theta lift contains terms of the form
\[ \Big( \frac{|\langle \lambda, Z \rangle|^2}{|Z|^2} \Big)^t \ln\Big( 2\pi \frac{|\langle \lambda, Z \rangle|^2}{|Z|^2} \Big). \]
Although we have logarithmic terms, the argument inside the $\log$ is real and positive (away from $\lambda^{\perp}$) so it is well-defined without branch cuts. In other words, the lift of $F$ is a well-defined function on the period domain away from $\lambda^{\perp}$.
\end{rmk}

If $k>0$ and $F$ is holomorphic (with a pole at $\infty$), then $\Lift(F)$ is meromorphic \cite{Borcherds1998}, and hence by the transformation property \eqref{tr property} above a meromorphic modular form for the subgroup of $O^{+}(M)$ under which $F$ is invariant (which at least contains the discriminant kernel).

If $k = 0$ and $F = \sum_{\lambda, n} c(n, \lambda) q^n \mathfrak{e}_{\lambda}$ is holomorphic (with a pole only at $\infty$), the singularity of $\Lift(F)$ along $\lambda^{\perp}$ is of type $$-c\Big(-\frac{\langle \lambda, \lambda \rangle}{2}, \lambda + M \Big) \cdot \ln \Big( 2\pi \frac{|\langle \lambda, Z \rangle|^2}{|Z|^2} \Big).$$ The main result of Borcherds \cite{Borcherds1998} is that if $F$ has integral Fourier coeffcients then there exists a meromorphic modular form $\Psi_F$, called Borcherds product, with the property $$\log |\Psi_F| = -\frac{1}{4} \Big( \Lift(F) + c(0,0) \cdot \log |Z|^2 + (\text{constant}) \Big).$$ In particular, the weight of $\Psi_F$ is $\frac{c(0, 0)}{2}$, and its zeros and poles lie on hyperplanes $\lambda^{\perp}$, with multiplicities determined by the Fourier coefficients $c(-\langle \lambda, \lambda \rangle / 2, \lambda + M)$. 
Moreover, $\Psi_F$ is modular for the subgroup of $O^+(M)$ under which $F$ is invariant, but usually with a character.

\subsection{Raising and lowering theta lifts}
We will now discuss the image of $\mathrm{Lift}(F)$ under the orthogonal raising and lowering operators.

The following well-known result will be used repeatedly:

\begin{lem}\label{lem:greens_formula} Suppose $F$ is an almost-holomorphic modular form of weight $\kappa = k+1-n/2$ for the Weil representation attached to $M$ with at worst exponential growth at $\infty$. Then $$\int^{\mathrm{reg}} [\RM_{\kappa} f, \Theta_{k+2}]\, \frac{\mathrm{d}x \, \mathrm{d}y}{y^2} = -\int^{\mathrm{reg}} [f, \LM_{\kappa + 2} \Theta_{k+2}] \frac{\mathrm{d}x\, \mathrm{d}y}{y^2}.$$
\end{lem}
\begin{proof}
    The point is that $$\Big([\RM_k f, \Theta_{k+2}] + [f, \LM_{k+2} \Theta_{k+2}]\Big) \frac{\mathrm{d}x \wedge \mathrm{d}y}{y^2} = -\mathrm{d} \Big( y^k \sum_{\lambda \in M'/M} f_{\lambda} \overline{\Theta_{k+2, \lambda}} \, \mathrm{d}\overline{\tau} \Big)$$ is an exact differential. So for any $h > 1$, letting $\mathcal{F}_h = \{x+iy \in \mathcal{F}: \; y < h\}$ be the fundamental domain with a cutoff at height $h$, we use Green's theorem to write
    \begin{align*} \int_{\mathcal{F}_h} \Big([\RM_k f, \Theta_{k+2}] + [f, \LM_{k+2} \Theta_{k+2}]\Big) \frac{\mathrm{d}x \wedge \mathrm{d}y}{y^2} &= \int_{-1/2+ih}^{1/2+ih} y^k \sum_{\lambda \in M'/M} f_{\lambda}(\tau) \overline{\Theta_{k+2, \lambda}(\tau, Z)} \, \mathrm{d}\overline{\tau} \\ &= h^k \sum_{\lambda \in M'} \overline{\langle \lambda, Z \rangle^{k+2}} c \Big( \frac{1}{2}\langle \lambda, \lambda \rangle \lambda + M \Big) \cdot e^{-2\pi h \langle \lambda_Z, \lambda_Z \rangle},
    \end{align*} where we write $F(\tau) = \sum_{\gamma \in M'/M} \sum_{n \in \mathbb{Q}} c(n, \gamma)$ with $c(n, \gamma) \in \mathbb{C}[y^{-1}]$.
    In the limit $h \rightarrow \infty$, the right-hand side tends exponentially quickly to zero, because any $\lambda \in M'$ for which $\langle \lambda_Z, \lambda_Z \rangle = 0$ also satisfies $\langle \lambda, Z \rangle = 0$, hence does not appear in the sum.
    By definition of the regularized integral, this means \[ \int^{\mathrm{reg}} \Big([\RM_k f, \Theta_{k+2}] + [f, \LM_{k+2} \Theta_{k+2}]\Big) \frac{\mathrm{d}x \wedge \mathrm{d}y}{y^2} = 0. \qedhere \]
\end{proof}

\begin{thm}
\label{thm:L on theta lift}
Suppose $F$ is an almost-holomorphic modular form of weight $\kappa = k + 1 - n/2$ with $k \ge 2$. The image of $\mathrm{Lift}(F)$ under the orthogonal lowering operator is $$\LO_k[\mathrm{Lift}(F)] = -\frac{1}{2\pi} \cdot \RO_{k-2}[\mathrm{Lift}(\LM F)].$$
\end{thm}
\begin{proof}
The vector-valued raising operator is independent of $\tau$ and can be swapped with the regularized integral.
Using Proposition \ref{prop:lr_maass}, $$\LO_{-k} \Big[ \RM_{\kappa} \Theta_k \Big] = -2\pi \cdot \RO_{-(k+2)}[\Theta_{k+2}],$$ we obtain $$\RO_{k-2} \Big[ \Big(\frac{\langle Z, \overline{Z} \rangle}{2}\Big)^{2-k} \cdot \overline{\RM \Theta_{k-2}} \Big] = -2\pi \cdot \overline{\LO_k \Big[ \Big( \frac{\langle Z, \overline{Z} \rangle}{2} \Big)^{-k} \overline{\Theta_k} \Big]}.$$

Therefore \begin{align*} \RO_{k-2}[\mathrm{Lift}(\LM F)] &= \frac{1}{2} (i/2)^{2-k} \RO_{k-2} \Big[ \langle Z, \overline{Z} \rangle^{2-k} \int^{\mathrm{reg}} [\LM F, \Theta_{k-2}] \frac{\mathrm{d}x\, \mathrm{d}y}{y^2} \Big] \\ &= -\frac{1}{2}(i/2)^{2-k} \int^{\mathrm{reg}} \Big[ F, \overline{\RO_{k-2}[\langle Z, \overline{Z} \rangle^{2-k} \overline{\RM \Theta_{k-2}}]}\Big] \frac{\mathrm{d}x\, \mathrm{d}y}{y^2} \\ &= -\pi (i/2)^{-k} \LO_k\Big[ \langle Z, \overline{Z} \rangle^{-k} \int^{\mathrm{reg}} [F, \Theta_k] \frac{\mathrm{d}x\, \mathrm{d}y}{y^2}\Big] \\ &= -2\pi \cdot \LO[\mathrm{Lift}(F)]. \qedhere \end{align*} 
\end{proof}

Lemma \ref{lem:zemel_kernel} for the theta kernel yields the action of Zemel's operators on theta lifts:

\begin{prop}\label{prop:liftzemel}
Let $F$ be almost-holomorphic of weight $\kappa = k+1-n/2$. \\
(i) The image of $\Lift(F)$ under the Zemel raising operator is $$\CR\Lift(F) = 2\pi \cdot \Lift(\RM F).$$
(ii) Suppose $k \ge 4$. The image of $\Lift(F)$ under the Zemel lowering operator is $$\CL \Lift(F) = \frac{1}{4\pi} \Lift \Big[\Big( \Delta_{k - 1 - n/2} - (k-2)n\Big) \LM F \Big],$$ where 
 $$\Delta_k = 4y^2 \frac{\partial^2}{\partial \tau \partial \overline{\tau}} - 2iky \frac{\partial}{\partial \overline{\tau}}$$ is the weight $k$ hyperbolic Laplacian on $\mathbb{H}$.
\end{prop}
\begin{proof}

(i) We can move $\mathcal{R}$ under the integration sign and therefore write \begin{align*}
	\mathcal{R}_k(\mathrm{Lift} \, F) &= 2^{k-1} i^{-k} \int_{\Gamma \backslash \mathbb{H}}^{\mathrm{reg}} \sum_{\lambda \in M'/M} F_{\lambda}(\tau) \mathcal{R}_k\Big( \langle Z, \overline{Z} \rangle^{-k} \overline{\Theta_{k, \lambda}(\tau, Z)} \Big) \frac{\mathrm{d}x \, \mathrm{d}y}{y^2}.\end{align*}

Applying (ii) of Proposition \ref{prop:Properties of L and R} twice and taking traces yields $$\mathcal{R}_k\Big( \langle Z, \overline{Z} \rangle^{-k} \overline{\Theta_{k, \lambda}(\tau, Z)} \Big) = 4 \cdot \langle Z, \overline{Z} \rangle^{-k-2} \overline{\mathcal{L}_{-k} \Theta_{k, \lambda}(\tau, Z)},$$ which by Lemma \ref{lem:zemel_kernel} equals $$8\pi \cdot \langle Z, \overline{Z} \rangle^{-k-2} \overline{\LM_{\kappa+2}\Theta_{k+2}}.$$ Now we use Lemma \ref{lem:greens_formula} to obtain \begin{align*} \mathcal{R}_k(\mathrm{Lift}\, F) &= 8\pi \cdot 2^{k-1} i^{-k} \langle Z, \overline{Z} \rangle^{-k-2} \int_{\Gamma \backslash \mathbb{H}}^{\mathrm{reg}} [F, \LM_{\kappa+2} \Theta_{k+2}] \, \frac{\mathrm{d}x \, \mathrm{d}y}{y^2} \\ &= -8\pi \cdot 2^{k-1} i^{-k} \langle Z, \overline{Z} \rangle^{-k-2} \int_{\Gamma \backslash \mathbb{H}}^{\mathrm{reg}} [\RM_{\kappa} F, \Theta_{k+2}] \, \frac{\mathrm{d}x \, \mathrm{d}y}{y^2} \\ &= 2\pi \cdot \mathrm{Lift}\, \RM_{\kappa}(F).\end{align*}

(ii) Using the commutator formula for $\LO$ and $\RO$ and the above results, we have \begin{align*} \LO_{k-1} \LO_k [\Lift(F)] &= -\frac{1}{2\pi} \LO_{k-1} \RO_{k-2} \Big[ \Lift(\LM F)\Big] \\ &= -\frac{k-2}{4\pi} \Lift(\LM F) \cdot g - \frac{1}{4\pi} \sigma_{1,2} \RO_{k-3} \LO_{k-2} \Lift( \LM F) \\ &= -\frac{k-2}{4\pi} \Lift(\LM F) \cdot g + \frac{1}{8\pi^2} \sigma_{1,2} \RO_{k-3} \RO_{k-4} \Lift(\LM \LM F).
\end{align*}
Taking traces and bearing in mind that $\tr(g)=n$ and part (i) yields \begin{align*} \CL \Lift(F) &= -\frac{(k-2)n}{4\pi} \Lift(\LM F) + \frac{1}{8\pi^2} \CR \Lift(\LM \LM F) \\ &= -\frac{(k-2)n}{4\pi} \Lift(\LM F) + \frac{1}{4\pi} \Lift(\RM \LM \LM F) \\ &= \frac{1}{4\pi} \Lift \Big[\Big( \Delta_{k - 1 - n/2} - (k-2)n\Big) \LM F \Big].
\end{align*}
Here we used the identity $\RM \LM = \Delta$.
\end{proof}

\begin{prop} \label{prop:R on Borcherds lift}
Let $F$ be holomorphic modular form (with a pole in $\infty$) of weight $1-n/2$ (so $k=0$). Then $\RO(\Lift(F))$ is an almost-holomorphic modular form of depth $1$. If $F$ has integral Fourier coefficients and $\Psi_F$ is the Borcherds product constructed from $F$, then
\begin{gather*}
 \RO( \Lift(F) ) = \partial(\Lift(F)) = -2 \frac{\RO(\Psi_F)}{\Psi_F}
 \end{gather*}
\end{prop}
\begin{proof}
To show that $\RO(\Phi_F)$ is almost-holomorphic, we can assume without loss of generality that $F$ has integral Fourier coefficients: by a theorem of McGraw \cite{McGraw}, such forms span the space of modular forms of weight $1-n/2$ with a pole at $\infty$. For simplicity, we write $\Phi=\Lift(F)$ and $\Psi=\Psi_F$.
Since $$\frac{1}{\Psi} \partial \Psi = \partial \log \Psi = \partial \Big( 2 \log |\Psi| - \log \overline{\Psi} \Big) = -\frac{1}{2} \partial \Phi_F - \frac{c(0, 0)}{2} \partial (\log |Z|^2),$$ and $\Psi$ has weight $\frac{c(0,0)}{2},$ it follows that \[\RO(\Psi) = \partial \Psi + \frac{c(0,0)}{2} \partial(\log |Z|^2) \Psi = -\frac{1}{2}\Psi \partial \Phi_F. \]
To see that $\RO(\Phi_F)$ is almost-holomorphic (of depth 1),  note that
\[
\LO( \RO(\Phi)) = \LO\left( \frac{\RO(\Psi)}{\Psi} \right) = \LO\left( \partial(\Psi)/\Psi + \frac{c(0,0)}{2} \nu \right)
= \frac{c(0,0)}{4} g
\]
where we use $\LO(\nu)=g/2$. Hence $\LO^2(\RO(\Phi))=0$. 
\end{proof}

\subsection{Lift of almost-holomorphic modular forms}
We are finally able to say when the lift of an almost-holomorphic modular form is almost-holomorphic:
\begin{thm} \label{thm:list of almost-holomorphic modular form}
Let $F$ be an almost-holomorphic modular form of weight $\kappa=k+1-\frac{n}{2}$ and depth $d>0$.
If $k \geq 2d$, then $\Lift(F)$ is a logarithmic almost-holomorphic modular form of weight $k$ and depth $2d$.
If moreover $M = U \oplus L$, then $\Lift(F)$ is almost-holomorphic if and only if $k \geq 2d$.
\end{thm}

\begin{proof}
Assume first that $k \geq 2d$ and $d \geq 1$.
We argue by induction on the depth $d$. Since $k \geq 2$
by Theorem~\ref{thm:L on theta lift} we have
\[
\LO(\Lift(F)) = -\frac{1}{2 \pi} \cdot \RO(\Lift( \LM F )).
\]

If $d=1$, then $\LM F$ is holomorphic of depth $0$ and weight $\geq 1-n/2$, so $\RO \Lift (\LM F)$ is almost-holomorphic of depth $1$ (if the weight is $=1-n/2$ use Proposition~\ref{prop:R on Borcherds lift}). Hence $\Lift(F)$ is almost-holomorphic of depth $2$.

If $d>1$, then $\Lift(\LM F)$ is almost-holomorphic of depth $2d-2$ by induction, so $\RO(\Lift(\LM F))$ is almost-holomorphic of depth $2d-1$.

Proposition \ref{prop:singularity} shows that if $k \ge 2d$ then $F$ has only poles of order $k$, and applying the lowering operator shows that $\LO^r F$ has poles of order $k - r$ for any $r \le 2d$. Hence $F$ is logarithmic in the sense of Definition \ref{defn:log}.

If $k<2d$, then $\LO^{d}(F)$ is of negative weight. Therefore, $F$ has nonzero Fourier coefficients $c^{(d)}(n, \gamma)$ with $n < 0$. If $M$ splits a hyperbolic plane then we can find vectors $\lambda \in \gamma + L$ of norm $\frac{\langle \lambda, \lambda \rangle}{2} = n$. By Proposition \ref{prop:singularity} these contribute logarithmic singularities to $\Lift(F)$ which are not annihilated under any power of the lowering operator, so $\Lift(F)$ is not almost-holomorphic.  
\end{proof}

\begin{rmk} The converse ``only-if" direction is not true in general if $M$ does not split a hyperbolic plane because the theta lift is not necessarily injective, even on holomorphic modular forms.
\end{rmk}

\section{The tube domain and quasimodular forms}
\label{sec:tube domain and quasimodular forms}

We consider the almost-holomorphic modular forms in the tube domain and define quasimodular forms as their constant term (or holomorphic part). We then give two proofs to show that the constant-term morphism, is an isomorphism. We also discuss the Fourier expansion of quasimodular forms and end with a series of examples by looking at the lattices $U\oplus (2m)$, $U \oplus U$ and $U \oplus U(m)$.

\subsection{Tube domain}
Let $M$ be a lattice of signature $(2,n)$ as before. In this section we will assume that $M$ splits a hyperbolic plane, i.e. $$M = L \overset{\perp}{\oplus} U$$ where $U = \mathbb{Z}e \oplus \mathbb{Z}f$ such that $e^2 = f^2 = 0$ and $e \cdot f = 1$.
Hence $L$ is a lattice of signature $(1,n-1)$.
We write vectors $Z=z_0 e + z + z_{n+1} f\in M_{\BC}$ with $z \in L_{\BC}$ as $Z=(z_0, z, z_{n+1})$.

The tube domain is
\[
\CC = \{ z \in L_{\BC} | \Im(z)^2 > 0 \}^{+}
\]
where `$+$' stands for one of the two connected components.
We have the biholomorphism
\begin{equation} \label{tube domain iso} \varphi : \CC \xrightarrow{\cong} \CD, \quad \varphi(z) = \left( -\frac{1}{2} z^2, z , 1 \right). \end{equation}

The map $\varphi$ induces an action 
of $O^{+}(M_{\BR})$ on $\CC$, where $\gamma \in O^{+}(M_{\BR})$ acts by
\[ \gamma \cdot z = \varphi^{-1}( \gamma \cdot \varphi(z))\quad \Leftrightarrow\quad \varphi(\gamma \cdot z) = \gamma \cdot \varphi(z). \]

Consider also the morphism $\widetilde{\varphi}: \CD \to \A(\CD)$ defined by
$\widetilde{\varphi}(z)=(-\frac{1}{2}z^2, z, 1)$.
Then $q \circ \widetilde{\varphi} = \varphi$, where $q : \A(\CD) \to \CD$ is the projection.
Hence for any $z \in \CC$, there exists $J(\gamma,z) \in \BC^{\ast}$ such that
\[ \gamma \cdot \widetilde{\varphi}(z) = J(\gamma,z) \widetilde{\varphi}(\gamma z). \]
The factor $J(\gamma,z)$ is holomorphic in $z$ and has the cocycle property
\[ J( \tilde{\gamma} \gamma, z ) = J(\tilde{\gamma}, \gamma z) J(\gamma, z). \]

The following is straightforward.
\begin{lemma} \label{lemma:norm Z}
Let $Z=\widetilde{\varphi}(z) = (-\frac{1}{2} z^2, z , 1)$. Then 
\[ \Im(z)^2 = |Z|^2 = \frac{1}{2} \langle Z,  \overline{Z} \rangle. \]
\end{lemma}

\subsection{Multidifferentials}
Recall that on $\CD$ we had a canonical isomorphism $\CE^{\otimes s} \otimes \CL^k \cong \Omega_{\CD}^{\otimes s} \otimes \CL^{k-s}$ which allowed us to view sections $\omega \in \CA(\CD, \CE^{\otimes s} \otimes \CL^k)$ as multidifferentials on $\CD$ (up to a twist by $\CL$).
Pulling back via $\varphi : \CC \to \CD$ we obtain a multidifferential form $\omega' = \varphi^{\ast}(\omega) \in \CA(\CC, \Omega_{\CC}^{\otimes s})$.


\begin{lemma}
Let $\Gamma \subset O^{+}(M)$ be a finite-index subgroup.
Then $\omega \mapsto \varphi^{\ast}(\omega)$ defines an isomorphism between
the space $\CA(\CD, \CE^{\otimes s} \otimes \CL^k)^{\Gamma}$ of $\Gamma$-invariant smooth sections of $\CE^{\otimes s} \otimes \CL^k$ and the space of smooth multidifferential forms $\omega'\in \CA(\CC, \Omega_{\CC}^{\otimes s})$ satisfying
\[ \forall \gamma \in \Gamma: \quad \gamma^{\ast}(\omega') = J(\gamma,z)^{k-s} \omega'. \]
\end{lemma}

In particular, we will identify sections of $\CE^{\otimes s} \otimes \CL^k$ with the corresponding multidifferentials on $\CC$.

\begin{proof}
By Lemma~\ref{lemma:identification E}, the section $\omega$ corresponds to the multidifferential $\widetilde{\omega} = q^{\ast}(\omega)$ on $\A(\CD)$ (satisfying some conditions), and we have $$\omega' = \varphi^{\ast} \omega = \widetilde{\varphi}^{\ast} q^{\ast} \omega = 
\widetilde{\varphi}^{\ast} \tilde{\omega}.$$ The coefficients of $\widetilde{\omega}$ are of degree $-k$, so $\widetilde{\omega}$ is of total degree $-k+s$, by which we mean $\lambda^{\ast}(\widetilde{\omega}) = \lambda^{-k+s} \cdot \widetilde{\omega}$ for all $\lambda \in \BC^{\ast}$. Moreover, $\widetilde{\omega}$ is invariant under $\Gamma$.
For $\gamma \in \Gamma$ we have $\gamma \circ \widetilde{\varphi} = J \cdot \widetilde{\varphi} \circ \gamma$, such that
$\gamma^{\ast} \tilde{\omega} = \tilde{\omega}$ translates to
\[
\omega' = \widetilde{\varphi}^{\ast} \tilde{\omega} = 
\widetilde{\varphi}^{\ast} \gamma^{\ast} \tilde{\omega} = 
(\gamma \circ \widetilde{\varphi})^{\ast} \tilde{\omega} 
= (J \cdot \widetilde{\varphi} \circ \gamma)^{\ast} \tilde{\omega} =
\gamma^{\ast} \widetilde{\varphi}^{\ast} J^{\ast}\tilde{\omega}
= J^{-k+s} \gamma^{\ast} \omega',
\]
and we get
$\gamma^{\ast}( \omega' ) = J^{k-s} \omega'$.
Conversely, any multidifferential $\omega$' on $\CC$ with this property can be shown to extend to such a multidifferential on $\A(\CD)$ invariant under $\Gamma$.
\end{proof}

Let $e_1, \ldots, e_n$ be a basis of $L$. Extend this to a basis of $M$ by $e_0:=e$ and $e_{n+1}:=f$.
Let $g_{ij} = \langle e_i, e_j \rangle$ for $i,j=0,\ldots, n+1$ be the Gram matrix and let $g^{ij}$ be its inverse.
Since $e,f$ are orthogonal to $L$, the submatrix $(g^{ij})_{i,j=1}^{n}$ is then the inverse of the Gram matrix $(g_{ij})_{i,j=1}^{n}$.
Let $z=\sum_i z_i e_i=(z_1,\ldots, z_n)$ be the coordinates of $z \in L_{\BC}$ in the basis $e_1,\ldots, e_n$. 
A multidifferential $F$ on $\CC$ can be expanded uniquely as
\begin{equation} F=\sum_{\substack{I=(i_1,\ldots,i_s) \\ 1 \leq i_k \leq n}} F_I dz_{I}, \quad dz_I = dz_{i_1} \otimes \cdots \otimes dz_{i_s}. \label{repn of F as multidifferential} \end{equation}
We define the contraction of the $(a,b)$-th indices of $F$ by
\[
\tr_{ab}(F) = \sum_{I=(i_1,\ldots,i_s)} F_I g^{i_a i_b} \cdot dz_{i_1} \otimes \cdots \otimes \widehat{dz_{i_a}} \otimes \cdots \otimes \widehat{dz_{i_b}} \otimes \cdots \otimes dz_{i_s}.
\]

\begin{lemma} \label{lemma:trace under pullback by phitilde}
For any section $\omega$ of $\CE^{\otimes s}$,
we have $\varphi^{\ast}(\tr_{ab}(\omega)) = \tr_{ab}( \varphi^{\ast} \omega )$.
\end{lemma}
\begin{proof}
Both trace maps are defined pointwise, so its enough to check this at a point. Moreover, we can assume without loss of generality that $s=2$ and $\omega = \omega_1 \otimes \omega_2$.
Then $\tr_{12}(\omega_1 \otimes \omega_2)$ was computed in Lemma~\ref{lemms3sdfd} and the claim follows by a direct computation.

More precisely, let $\widetilde{\omega}_1 = \sum_i a_i dZ_i$ and $\widetilde{\omega}_2 = \sum_i b_i dZ_i$ be the differential forms on the affine cone corresponding to $\omega_1, \omega_2$
written in the variables $Z=(Z_0,\ldots, Z_{n+1})$ given by the basis $e_0,\ldots, e_{n+1}$.
Then $\tr(\omega_1 \otimes \omega_2)$ is given by $\tr(\widetilde{\omega}_1 \otimes \widetilde{\omega}_2) = \sum_{i,j=0}^{n+1} g^{ij} a_i(Z) b_j(Z)$ by Lemma~\ref{lemms3sdfd},
and therefore
$\varphi^{\ast} \tr(\omega_1 \otimes \omega_2) = \widetilde{\varphi}^{\ast} \tr( \widetilde{\omega}_1 \otimes \widetilde{\omega}_2 )$.
We hence need to check that
\begin{equation} \label{checkable equation} \tr(\widetilde{\varphi}^{\ast}(\widetilde{\omega}_1) \otimes \widetilde{\varphi}^{\ast}(\widetilde{\omega}_2)) = \widetilde{\varphi}^{\ast}(\tr(\widetilde{\omega}_1 \otimes \widetilde{\omega}_2)). \end{equation}
Note that
\begin{equation} \label{3fdsfd}
\begin{aligned}
\widetilde{\varphi}^{\ast}(dZ_0) & = -\sum_{a,b=1}^{n} z_a g_{ab} dz_{b} \\
\widetilde{\varphi}^{\ast}(dZ_i) & = dz_i \quad \text{for } i=1,\ldots,n \\
\widetilde{\varphi}^{\ast}(dZ_{n+1}) & = 0.
\end{aligned}
\end{equation}
This yields $\varphi^{\ast}(\widetilde{\omega}_1) = \sum_{i=1}^{n} (a_i - \sum_{a=1}^{n} z_a g_{ai} a_0) dz_i$ and similarly for $\varphi^{\ast}(\widetilde{\omega}_2)$.
Moreover, for $v=\sum_{i} Z_i \frac{\partial}{\partial Z_i}$ we have $\omega_i(v)=0$ by Lemma~\ref{lemma:identification E}, so $\sum_{i} a_i Z_i=0$, which gives us
$\widetilde{\varphi}^{\ast}(a_{n+1}) = \frac{1}{2} z^2 \widetilde{\varphi}^{\ast}a_0 - \sum_i \widetilde{\varphi}^{\ast}(a_i) z_i$. Putting this together allows us to check \eqref{checkable equation} directly. 

However, one can also the following shortcut. Both sides of \eqref{checkable equation} only depend on $\widetilde{\omega}_i$ up to adding a multiple of $d(Z^2) = 2 Z_0 dZ_{n+1} + 2Z_{n+1} dZ_0 + 2 \sum_{i,j} Z_i g_{ij} dZ_{j}$. So after adding the appropriate multiple, we may assume
$\widetilde{\varphi}^{\ast}(a_0(Z)) = 0$. Hence
\[
\widetilde{\varphi}^{\ast}(\tr(\widetilde{\omega}_1 \otimes \widetilde{\omega}_2))
=
\sum_{i,j=0}^{n+1} \widetilde{\varphi}^{\ast}(a_i) \widetilde{\varphi}^{\ast}(b_j) g^{ij}
=
\sum_{i,j=1}^{n} \widetilde{\varphi}^{\ast}(a_i) \widetilde{\varphi}^{\ast}(b_j) g^{ij}
=
\tr(\widetilde{\varphi}^{\ast}(\widetilde{\omega}_1) \otimes \widetilde{\varphi}^{\ast}(\widetilde{\omega}_2)). \qedhere
\]
\end{proof}

Recall that we defined
\[ \nu = \partial \log(|Z|^2)  \in \CA( \A(\CD), \Omega_{\A(\CD)}). \]
Pulling back by $\widetilde{\varphi}$ and using Lemma~\ref{lemma:norm Z} gives
\[ \widetilde{\varphi}^{\ast}(\nu)=\partial( \widetilde{\varphi}^{\ast} \log |Z|^2 ) = \partial( \log( \Im(z)^2 )). \]
By a slight abuse of notation, we also write $\nu$ for the pullback $\widetilde{\varphi}^{\ast}(\nu)$; that is,
\[ \nu = \partial( \log ( \Im(z)^2 )) = \frac{1}{i \Im(z)^2} \sum_{a,b} \Im(z_a) g_{ab} dz_b \in \CA( \CC, \Omega_{\CC}). \]

\begin{lemma} \label{lemma:trace under pullback with nu}
For any section $\omega$ of $\CE^{\otimes s}$ on $\CD$, let $\widetilde{\omega}$ be the corresponding multidifferential on the affine cone (as in Lemma~\ref{lemma:identification E}). Then 
$\widetilde{\varphi}^{\ast}(\tr_{s,s+1}(\widetilde{\omega} \otimes \nu)) = \tr_{s,s+1}( \widetilde{\varphi}^{\ast}(\widetilde{\omega}) \otimes \nu )$.
\end{lemma}
\begin{proof}
This is proven parallel to Lemma~\ref{lemma:trace under pullback by phitilde}.
\end{proof}

Recall that we defined the element
\[ g = \sum_{a,b=0}^{n+1} g_{ab} dZ_a \otimes dZ_b \in \Gamma(\CD,\CE \otimes \CE). \]
We will also write $g$ for the pullback $\varphi^{\ast}(g)$, so
\[ g = \sum_{a,b=1}^{n} g_{ab} dz_a \otimes dz_b. \]

We are finally ready to describe the raising and lowering operators on the tube domain.



\begin{prop} \label{prop:R and L in the tube domain I}
Let $F= \sum_{I} F_I dz_I$ be a section of $\CE^{\otimes s} \otimes \CL^k$ (which we identify with a multidifferential on the tube domain). Then we have
\[
\RO(F) = \partial(F) + k F \otimes \nu 
		+ \sum_{r=1}^{s} \sigma_{r, s+1}(F \otimes \nu) \notag
		- \sum_{r=1}^{s} \tr_{s+2,s+3} \sigma_{r, s+2}( F \otimes g \otimes \nu )
\]
where $\sigma_{i,j}$ stands for interchanging the $i$-th and $j$-th factors in the tensor, and
\[
\LO(F) = \sum_{I,\ell} \frac{\partial F_I}{\partial \overline{z}_\ell} dz_I \otimes \Big( \Im(z)^2 dz_\ell - 2 i \cdot \Im(z_\ell) \partial( \Im(z)^2 ) \Big)
\]
\end{prop}
\begin{proof}
The formula for $\RO$ follows immediately from \eqref{formulaR} in Proposition~\ref{prop:Properties of L and R} by pulling back via $\widetilde{\varphi}$ and using Lemma~\ref{lemma:trace under pullback with nu}.
Assume first that $F$ is a section of $\CL^k$. We identify $F=F(Z)$ with a function on the affine cone of degree $-k$. By \eqref{formulaL} we have
\[ \LO(F) = \sum_{\ell} \frac{\partial F}{\partial \overline{Z}_\ell} (|Z|^2 dZ_\ell -  Z_\ell \partial(|Z|^2)). \]
We want to pull this back via $\widetilde{\varphi}$.
We use \eqref{3fdsfd} for the pullback of the $dZ_i$.
For the derivatives note first that since $F$ is of bidegree $(-k,0)$, Euler's equation gives
$\sum_{\ell} \frac{\partial F}{\partial \overline{Z}_\ell} = 0$, so after pulling back to $\CC$, we get
\[
\frac{\partial F}{\partial \overline{Z}_{n+1}}(\widetilde{\varphi}(z))
= \frac{\partial F}{\partial \overline{Z}_0}(\widetilde{\varphi}(z)) \frac{1}{2} \overline{z}^2 -
\sum_{\ell=1}^{n} \frac{\partial F}{\partial \overline{Z}_\ell}(\widetilde{\varphi}(z)) \overline{z}_\ell.
\]
Moreover, by the chain rule for all $\ell=1,\ldots,n$ we have
\[
\frac{\partial F}{\partial \overline{Z}_\ell}(\widetilde{\varphi}(z))
=
\frac{\partial F \circ \widetilde{\varphi}}{\partial \overline{z}_\ell}(z) 
+ \frac{\partial F}{\partial \overline{Z}_0}(\widetilde{\varphi}(z)) \sum_{a} \overline{Z}_a g_{a\ell}.
\]
Inserting everything we can write $\widetilde{\varphi}^{\ast}\LO(F)$ in terms of 
$\frac{\partial F}{\partial \overline{Z}_0}(\widetilde{\varphi}(z))$
and $\frac{\partial F \circ \widetilde{\varphi}}{\partial \overline{z}_\ell}(z)$.
One finds that the dependence on the former drops out, and the expression in terms of the latter is precisely the claim.

In general, if $F$ is given by a multidifferential $\sum_{I=(i_0,\ldots,i_{n+1})} F'_I dZ_I$ on the affine cone, then by adding appropriate multiples of $dZ^2$ in the factors (as in the proof of Lemma~\ref{lemma:trace under pullback by phitilde}, this does not change $F$ as a section of $\CE^{\otimes s} \otimes \CL^k$; see Lemma~\ref{lemma:identification E}),
we may assume that $\widetilde{\varphi}^{\ast} F'_{(i_0,\ldots, i_{n+1})} = 0$ whenever $i_a=0$ for some $a$.
Then the computation of $\widetilde{\varphi}^{\ast}\LO(F)$ immediately reduces to the scalar-valued case.
\end{proof}

\subsection{Multidifferentials as multilinear maps}
\label{subsec:multidiff as multilinear maps}
The commutator identity for $[\LO,\RO]$ can be neatly formulated as follows.
Since $\CC$ is an open subspace of $L_{\BC}$, we can identify tangent vectors in $\CC$ with elements on $L_{\BC}$. Hence a multidifferential $F \in \CA(\CC, \Omega_{\CC}^{\otimes s}$ defines a linear map
$F : L_{\BC}^{\otimes s} \to \BC$ which may be viewed as a multi-linear map
\[ F : L_{\BC}^{\times s} \to \BC. \]
Concretely, if $F$ is represented as \eqref{repn of F as multidifferential}, then $F(e_{i_1} \otimes \cdots \otimes e_{i_s}) = F(e_{i_1}, \ldots, e_{i_s}) = F_{i_1, \ldots, i_s}$.

Given $\alpha \in L_{\BC}$, denote by $\LO_{\alpha}$ the operator on multi-differentials that applies $\LO$ and then inserts $\alpha$ in the last entry, and similar for $\RO$.
In other words, given $\gamma_1, \ldots, \gamma_s, \alpha, \beta \in L_{\BC}$ we write
\[
\LO_{\alpha} F(\gamma_1, \ldots, \gamma_s) := \LO(F)( \gamma_1 \otimes \cdots \otimes \gamma_s \otimes \alpha)
\]
and for the compositions we write
\begin{align*} 
\LO_{\alpha} \RO_{\beta} F(\gamma_1, \ldots, \gamma_s) 
&:= (\LO \RO(F))( \gamma_1 \otimes \cdots \otimes \gamma_s \otimes \beta \otimes \alpha) \\
\RO_{\beta} \LO_{\alpha} F(\gamma_1 , \ldots,  \gamma_s) 
&:= (\RO \LO(F))( \gamma_1 \otimes \cdots \otimes \gamma_s \otimes \alpha \otimes \beta).
\end{align*}

\begin{lemma} \label{lemma:commutation relation}
Let $F$ be a section of $\CE^{\otimes s} \otimes \CL^k$, identified with a multidifferential on the tube domain. Let $\gamma_1, \ldots, \gamma_s, \alpha, \beta \in L_{\BC}$.
Then
\[
[\LO_{\alpha}, \RO_{\beta}] F(\gamma_1, \ldots, \gamma_s)
=
\frac{k}{2} \langle \alpha, \beta \rangle F(\gamma_1, \ldots, \gamma_s)
+ \frac{1}{2} \sum_{m=1}^{s} F( \ldots, (\beta \wedge \alpha)(\gamma_i), \ldots )
\]
where $(\alpha \wedge \beta)(\gamma) = \langle \beta, \gamma \rangle \alpha - \langle \alpha, \gamma \rangle \beta$ is the Eichler transvection (Section~\ref{subsec:Eichler transvection})
\end{lemma}
\begin{proof}
By Proposition~\ref{prop:commutation relation} (and pulling it back via $\varphi$) we have the commutation relation
\[ {[} \LO, \RO ] F = \frac{k}{2} F \otimes g
+ \frac{1}{2} \sum_{r=1}^{s} \sigma_{r, s+1}( F \otimes g) - \sigma_{r, s+2}(F \otimes g). \]
Applying this to $\gamma_1 \otimes \gamma_s \otimes \beta \otimes \alpha$ yields the right hand side.
For example,
\begin{align*}
\sigma_{r, s+1}( F \otimes g)(\gamma_1 \otimes \gamma_s \otimes \beta \otimes \alpha)
& = (F \otimes g)(\gamma_1 \otimes \cdots \otimes \beta \otimes \cdots \otimes \gamma_s \otimes \gamma_r \otimes \alpha)  \\
& = F(\gamma_1, \ldots, \gamma_{r-1}, \beta, \gamma_{r+1}, \ldots, \gamma_s) \langle \alpha, \gamma_r \rangle. \qedhere
\end{align*}
\end{proof}

\subsection{Tensor representation}
The tube domain $\CC$ is an open subset in the vector space $L_{\BC}$, which allows us to identify $\Omega_{L_{\BC}} \cong L_{\BC}^{\vee} \otimes \CO$. After choosing a basis of $L$, we may therefore view a multidifferential $F \in \CA(\CC,\Omega_{\CC}^{\otimes s})$ as a tensor-valued function $F : \CC \to (L_{\BC}^{\vee})^{\otimes s} \cong 
(\BC^{n})^{\otimes s}$. Concretely, given a basis $e_1, \dots, e_n$ with associated coordinates $z=(z_1,\ldots,z_n)$,
we identify a multidifferential
\[ F=\sum_{\substack{I=(i_1,\ldots,i_s) \\ 1 \leq i_k \leq n}} F_I dz_{I}, \quad dz_I = dz_{i_1} \otimes \cdots \otimes dz_{i_s}. \]
with the tensor-valued function
\[ F = ( F_{i_1,\ldots, i_s} )_{i_1,\ldots,i_s=1}^{n} : \CC \to (\BC^{n})^{\otimes s}. \]
We often omit the brackets, simply writing $F = F_{i_1,\ldots, F_{i_s}}$ and call $F$ a tensor. \\

We will express the transformation law and the raising and lowering operators in this language.

For $\gamma \in O^{+}(M_{\BR})$, define $\lambda(\gamma,z)$ to be the $n \times n$-matrix
with entries
\[ \lambda(\gamma,z)_{ij} = J(\gamma,z) \cdot \frac{ \partial (\gamma \cdot z)_i }{\partial z_j}. \]
The chain rule implies the cocycle property $\lambda( \tilde{\gamma} \gamma, z) = \lambda( \tilde{\gamma}, \gamma z ) \lambda( \gamma, z)$
for all $\gamma, \tilde{\gamma} \in O^{+}(M_{\BR})$.

\begin{example} \label{example:Automorphy factors}
We have the following automorphy factors
for the generators of the orthogonal group $O^+(M_{\BR})$
(compare \cite[Sec.3]{Zemel}):
\begin{itemize}
	\item[(a)] The elements $\exp( \xi \wedge e)$ for $\xi \in L_{\BR}$ (defined in Section~\ref{subsec:Eichler transvection}) which act by
	\[ \exp( \xi \wedge e) \cdot z = z + \xi\] have automorphy factors \[ \quad J(\exp( \xi \wedge e) , z ) = 1,
	\quad \lambda(\exp( \xi \wedge e),z) = \id. \]
	\item[(b)] For $A \in O(L_{\BR})$ and $a \in \BR$ where
	\[ \begin{cases} a > 0 & \text{ if } A \in O^+(L_{\BR}) \\
		a < 0 & \text{ if } A \notin O^+(L_{\BR})
	\end{cases} \]
	we have the element
	\[
	k_{a,A} = \begin{pmatrix} a & & \\ & A & \\ & & 1/a \end{pmatrix} \in O^+(M_{\BR}) \]
	where the matrix is with respect to the basis $e, e_1, \ldots, e_n, f$.
	It acts by
	\[ k_{a,A} \cdot z = a A z\] and has automorphy factors \[J(k_{a,A},z) = 1/a, \quad \lambda(k_{a,A},z) = A. \]
	\item[(c)] The element $S \in O^+(M_{\BR})$ defined by
	$e \to f, f \mapsto e, S|_{L} = \id$.
	It acts by
	\[ S \cdot z = -2 \frac{z}{\langle z, z \rangle}\] and has automorphy factors \[J(S, z) = -\frac{1}{2} \langle z, z \rangle,
	\quad \lambda(S,z)_{ij} = \delta_{ij} \langle z, z \rangle - 2 z_i \langle e_j, z \rangle.
	\]
\end{itemize}
\end{example}

\begin{defn}
For $k \in \BZ$, $s \geq 0$ and a tensor $F : \CC \to (\BC^{\otimes n})^{s}$, the slash action is defined by
\[
(F|_{k,s} \gamma)(z) :=
J(\gamma,z)^{-k} F(\gamma z) \cdot \lambda(\gamma,z)^{\otimes s}, \quad \gamma \in O^+(M_{\BR}).
\]
\end{defn}

Written out in coefficients, this is
\[ (F|_{k,s} \gamma)(z)_{j_1,\ldots,j_s} =
J(\gamma,z)^{-k} \sum_{i_1,\ldots,i_s=1}^{n} F_{i_1, \ldots, i_s}(\gamma z)
\lambda_{i_1 j_1}(\gamma, z) \cdots \lambda_{i_s j_s}(\gamma, z) \]

If a multi-differential $F \in \CA(\CC, \Omega_{\CD}^{\otimes s})$ corresponds to a tensor $F = (F_{i_1,\ldots, i_s})$, then
$J(\gamma,z)^{-k+s} \gamma^{\ast}(F)$ corresponds to the tensor $(F_{i_1,\ldots, i_s})|_{k,s} \gamma$. Hence we get:
\begin{lemma}
Let $\Gamma \subset O^{+}(M)$ be a subgroup.
There are natural (depending on a basis of $L$) isomorphisms between:
\begin{enumerate}
\item[(i)] the space $\CA(\CD, \CE^{\otimes s} \otimes \CL^k)^{\Gamma}$ of $\Gamma$-invariant smooth sections of $\CE^{\otimes s} \otimes \CL^k$,
\item[(ii)] the space of smooth multidifferential forms $\omega'\in \CA(\CC, \Omega_{\CC}^{\otimes s})$ satisfying
\[ \forall \gamma \in \Gamma: \quad \gamma^{\ast}(\omega') = J(\gamma,z)^{k-s} \omega', \]
\item[(iii)] the space of smooth functions $F:\CC \to (\BC^{\otimes n})^{\otimes s}$ satisfying
\[ \forall \gamma \in \Gamma: F|_{k,s} \gamma = F. \]
\end{enumerate}
\end{lemma}
\begin{cor}
Let $\Gamma \subset O^{+}(M)$ be a finite index subgroup.
There is a natural isomorphism between $\Mod_{k,s}(\Gamma)$ and the space of holomorphic tensors
$F : \CC \to (\BC^{\otimes n})^{\otimes s}$ 
satisfying $F|_{k,s} \gamma = F$ for all $\gamma \in \Gamma$.
\end{cor}

Now we describe the raising and lowering operators on tensors. Expand
\[ \nu = \partial \log(\Im(z)^2) = \sum_j \nu_j(z) dz_j, \quad \nu_j(z) = \frac{\partial}{\partial z_j} \log( \Im(z)^2 ) 
=
\sum_{a=1}^{n}
\frac{1}{i} \frac{\Im(z_a) g_{aj}}{\Im(z)^2}. \]

\begin{prop} \label{prop:R L in tube domain}
(i) 
If $F = (F_{i_1,\ldots, i_s})$ is a section of $\CE^{\otimes s} \otimes \CL^k$ written as a tensor, then:
\begin{align*}
	\RO(F)_{i_1, \ldots, i_s,a} 
	& = \frac{ \partial F_{i_1, \ldots, i_s} }{\partial z_{a}} + k \nu_{a} F_{i_1, \ldots, i_s}
	+ 
	\sum_{m=1}^{s} \nu_{i_m} F_{i_1, \ldots, i_{m-1}, a, i_{m+1}, \ldots, i_s} \\
	& \quad - \sum_{m=1}^{s} \sum_{b,c} g_{a i_m} \nu_{b} g^{b c} F_{i_1, \ldots, i_{m-1}, c, i_{m+1}, \ldots, i_s}. \\
	\LO( F )_{i_1, \ldots, i_s,a} & 
    = \frac{\partial F_I}{\partial \overline{Z}_a} \Im(z)^2 - 2 \sum_{b=1}^{n} \frac{\partial F_I}{\partial \overline{Z}_{b}} \Im(z_b) \sum_{c} g_{ac} \Im(z_c)
\end{align*}
(ii) We have the commutation relation
\begin{multline*}
	\LO \RO(F)_{i_1,\ldots,i_s,a,b} - \RO \LO(F)_{i_1,\ldots,i_s,b,a}
	=
	\frac{k}{2} g_{ab} F_{i_1, \ldots, i_s} \\
	- \frac{1}{2} \sum_{m=1}^{s} g_{i_m b} F_{i_1, \ldots, i_{m-1}, a, i_{m+1}, \ldots, i_s}
	+ \frac{1}{2} \sum_{m=1}^{s} g_{i_m a} F_{i_1, \ldots, i_{m-1}, b, i_{m+1}, \ldots, i_s}
\end{multline*}
\end{prop}
\begin{proof}
Part (i) is a reformulation of Proposition~\ref{prop:R and L in the tube domain I}.
Part (ii) follows from Proposition~\ref{prop:commutation relation} by pullback via $\varphi$
or by a direct computation from (i).
\end{proof}

Recall that a section $F$ of $\CE^{\otimes s} \otimes \CL^k$ is almost-holomorphic of depth $d$ if $\LO^d(F) \neq 0$ but $\LO^{e}(F) = 0$ for all $e>d$.
By Proposition~\ref{prop:R L in tube domain} it is easy to describe almost-holomorphic functions:
\begin{lemma} The function $F=(F_{i_1,\ldots,i_s}) : \CC \to (\BC^{\otimes n})^{s}$ is almost-holomorphic of depth $d$ if and only if every $F_{i_1,\ldots,i_s}$ is a polynomial in the $\nu_1,\ldots,\nu_n$ of degree at most $d$, and at least one $F_{i_1,\ldots,i_s}$ is of degree exactly $d$.
\end{lemma}
\begin{proof}
If we view $z_1,\ldots, z_n,\nu_1,\ldots, \nu_n$ as independent variables on $L_{\BC}$ (instead of $z_1,\ldots,z_n,\overline{z}_1, \ldots, \overline{z}_n$) we find the following expression for $\LO$:
\begin{equation} \label{L on nu} \LO( F )_{i_1, \ldots, i_s,a} = \frac{1}{2} \sum_{j} g_{aj} \frac{\partial F_{i_1, \ldots, i_s}}{\partial \nu_j}. \end{equation}
For example, we have  $\LO(\nu) = \frac{1}{2} g$.
Thus
\[ \sum_{a_1,\ldots, a_d} \LO^d(F)_{i_1,\ldots,i_s,a_1,\ldots,a_d} g^{a_1 b_1} \cdots g^{a_d b_d} = 
\frac{1}{2^d} \frac{\partial^d F_{i_1,\ldots,i_s}}{\partial \nu_{b_1} \cdots \partial \nu_{b_d}}. \]
This gives the claim.
\end{proof}

\begin{cor}
Let $\Gamma \subset O^{+}(M)$ be a finite index subgroup.
There is a natural isomorphism between the space of almost-holomorphic modular forms $\AHMod_{k,s}(\Gamma)$ and the space of smooth tensors
$F=(F_{i_1,\ldots, i_s}) : \CC \to (\BC^{\otimes n})^{\otimes s}$ that satisfy
\begin{itemize}
\item $F|_{k,s} \gamma = F$ for all $\gamma \in \Gamma$, and
\item every $F_{i_1,\ldots,i_s}$ is a polynomial in the $\nu_1,\ldots, \nu_s$.
\end{itemize}
\end{cor}

\subsection{Quasimodular forms}
In the last section we have seen that a tensor $F : \CC \to (\BC^{n})^{\otimes s}$ is almost-holomorphic if it can be written as a polynomial in the $\nu_i$.
We define the holomorphic limit of $F$ to be the holomorphic function $f : \CC \to (\BC^{n})^{\otimes s}$ given by
\[ f := F|_{\nu_1 = \ldots = \nu_n = 0} \]
Intuitively, this corresponds to taking a limit $\Im(z)^2 \to \infty$,
so that $\nu_1,\ldots,\nu_s \to 0$.

We define quasimodular forms as the holomorphic limit of almost-holomorphic modular forms:

\begin{defn}
Let $\Gamma \subset O^{+}(M_{\BR})$ be a finite index subgroup and $n \geq 3$.
A function $f = ( f_{i_1, \ldots, i_s} ) : \CC \to (\BC^{\otimes n})^{s}$ is a quasimodular form of weight $k$ and rank $s$ for $\Gamma$ if there exists a almost-holomorphic modular form $F = F_{i_1, \ldots, i_s}$ of weight $k$ and rank $s$ for $\Gamma$ such that
$$F_{i_1, \ldots, i_s}|_{\nu_1 = \ldots = \nu_n = 0} = f.$$
\end{defn}
\begin{rmk}
We refer to Section~\ref{subsec:general cusp} for the $n \leq 2$ case.
\end{rmk}

\begin{rmk}
The lowering operator $\LO$ was defined on the period domain $\CD$. The notion of an almost-holomorphic modular form is hence independent of the choice of tube domain.
On the other hand, the holomorphic limit $\Im(z)^2 \to \infty$ does depend on the choice of tube domain, hence the notion of quasimodular form does as well.
\end{rmk}

Let $\QMod_{k,s}(\Gamma)$ 
be the vector space of quasimodular forms for group $\Gamma$ of weight $k$ and rank $s$.
By definition there is a morphism
\[ \ct : \AHMod_{k,s}(\Gamma) \to \QMod_{k,s}(\Gamma), \quad F \mapsto F|_{\nu_1 = \ldots, \nu_n=0} \]
which we call the {\em constant-term map}. 

The fundamental result of this section is the following:

\begin{thm} \label{thm:constant term}
The constant-term map $\ct$ is an isomorphism.
\end{thm}

We will give two proofs of this result below.
The first proof uses the transformation property of the quasimodular form (which may be of independent interest), and will
require some preparation.
Consider the element
\[ \nu(z) = \partial \log( \Im(z)^2 ) = \sum_{i=1}^{n} \nu_i(z) dz_i. \]

\begin{lemma} \label{lemma:nu transformation} For any $\gamma \in O^+(V)$, we have
\[ (\nu|_{1,1} \gamma)(z) = \nu(z) - \partial \log J(\gamma,z) . \]
\end{lemma}
\begin{proof}
Let $Z=\widetilde{\varphi}(z) = (-\frac{1}{2} z^2, z , 1)$. 
Recall that $Z \cdot \overline{Z} = 2 \Im(Z)^2$. Hence
for any $\gamma \in O^{+}(M_{\BR})$ we have
\begin{align*}
	\Im( \gamma z)^2 
	& = \frac{1}{2} \langle \widetilde{\varphi}(\gamma Z), \overline{ \widetilde{\varphi}(\gamma z)} \rangle \\
	& = \frac{1}{2}  \left\langle J(\gamma,z)^{-1} \gamma \cdot \widetilde{\varphi(z)}, \overline{J(\gamma,z)}^{-1} \gamma \cdot \overline{\widetilde{\varphi}(z)} \right\rangle \\
	& = \frac{1}{2} |J(\gamma,z)|^{-2} \langle \widetilde{\varphi}(z), \overline{\widetilde{\varphi(z)}} \rangle \\
	& = |J(\gamma,z)|^{-2} \Im(z)^2. 
\end{align*}
By taking the logarithmic derivative one obtains
\[ \sum_{i} \nu_i(\gamma z) \frac{\partial(\gamma \cdot z)_i}{\partial z_j} = - \frac{\partial}{\partial z_j}( \log J(\gamma,z) ) + \nu_j(z). \]
This proves the claim.
\end{proof}

Recall the Gram matrix 
$g = (g_{ij})_{i,j=1}^{n}$, where $g_{ij} = \langle e_i, e_j \rangle$.
We have the basic invariance property:

\begin{lemma} \label{lemma:sigma and g}
For all $\gamma \in O^+(M_{\BR})$ we have
\begin{enumerate}
	\item $\lambda(\gamma,z)^t g \lambda(\gamma,z) = g$.
	\item $\lambda(\gamma,z) g^{-1} \lambda(\gamma,z)^t = g^{-1}$
\end{enumerate}
\end{lemma}
\begin{proof}
This can be proved in at least two ways.
Consider the element
$$g_{\CD} = \sum_{a,b} g_{ab} dZ_a \otimes dZ_b \in \Gamma(\CD, \CE \otimes \CE^{\vee})$$ that was defined in \eqref{g on period domain}.
Then $g_{\CD}$ is clearly $O^{+}(M_{\BR})$-invariant,
so its pullback $\varphi^{\ast}(g_{\CD})$ is as well, but this means exactly that $\lambda(\gamma,z)^t g \lambda(\gamma,z) = g$ for all $\gamma \in O^{+}(M_{\BR})$.
The second claim follows from the first by inverting and using the cocycle property of $\lambda$ which gives $\lambda(\gamma,z)^{-1} = \lambda( \gamma^{-1}, \gamma z)$.

Alternatively, we can also argue by direct computation.
We need to prove that
\[ \langle \lambda(\gamma,z) v, \lambda(\gamma,z) w \rangle = \langle v, w \rangle \]
If the statement holds for any $\gamma, \gamma'$ then it also holds for $\gamma' \gamma$
by the cocycle properties of $\lambda$ and $J$.
Thus it suffices to check this on generators, where it follows immediately
from Example~\ref{example:Automorphy factors}:
\begin{enumerate}
	\item[(a)] For $\gamma = \exp(\xi \wedge e)$ we have
	$\lambda(\gamma,z) = \id$ so the claim holds.
	\item[(b)] For $\gamma = k_{a,A}$ we have $\lambda(\gamma,z) = A$ where $A$ is orthogonal, so the claim holds.
	\item[(c)] For $\gamma = S$ we have
	$\lambda(S,z) = \langle z, z \rangle \cdot I_n - 2 z (gz)^t$,
    which implies the claim by a straightforward computation (using $z^t g z=\langle z, z \rangle$). \qedhere
\end{enumerate}
\end{proof}

Let $F = (F_{i_1,\ldots, i_s}) : \CC \to (\BC^n)^{\otimes s}$ be a tensor corresponding to a section of $\CE^{\otimes s} \otimes \CL^k$. 
We have the following invariance property of the trace.
\begin{lemma}
$\tr_{ab}(F)|_{k,s-2} \gamma = \tr_{ab}( F|_{k,s} \gamma )$ for all $\gamma \in O^{+}(M_{\BR})$.
\end{lemma}
\begin{proof}
This follows from a computation using Lemma~\ref{lemma:sigma and g}(2), i.. that we have $\sum_{j,k} \lambda_{ij} g^{jk} \lambda_{\ell k} = g^{i \ell}$.
Alternatively, it can be proved by observing that the trace operation as defined on $\CD$ is $O^{+}(M)$-invariant, so the claim follows after pulling back to the tube domain.
\end{proof}


Let $F = (F_{i_1,\ldots,i_s})$ be an almost-holomorphic function of depth $d$.
Since each component of $F$ is a polynomial in the $\nu_1,\ldots,\nu_s$
we can write
\[ F = \sum_{t=0}^{d} \langle F^{(t)}, \nu^{\otimes t} \rangle, \]
where the holomorphic functions
\[ F^{(t)} = \left( F_{i_1,\ldots,i_s,i_{s+1}, \ldots, i_{s+t}} \right)_{i_{1},\ldots,i_{s+t} = 1}^{n} : \CC \to (\BC^{n})^{\otimes (s+t)}, \quad t=0,\ldots,d \]
are symmetric in the last $t$-entries, i.e.
\[ F_{i_1,\ldots,i_s,\sigma(i_{s+1}, \ldots, i_{s+t})} = F_{i_1,\ldots,i_s,i_{s+1}, \ldots, i_{s+t}} \text{ for all } \sigma \in S_t. \]
Here we used the pairing \eqref{eqn:defn of pairing}, i.e.
$\left\langle F, G \right\rangle =
\tr_{(s+1,s+t+1), \ldots, (s+t,s+2t)}(F \otimes G)$ for $F$ an $(s+t)$-tensor and $G$ and $t$-tensor.


\begin{lemma} If $F = \sum_{t=0}^{d} \langle F^{(t)}, \nu^{\otimes t} \rangle$ is an almost-holomorphic modular form, then
\[ F^{(t)} = 2^t t! \LO^t(F)|_{\nu=0}. \]
\end{lemma}
\begin{proof}
Straightforward computation.
\end{proof}

\begin{prop} \label{prop:transformation property}
Let $F = \sum_{t=0}^{d} \langle F^{(t)}, \nu^{\otimes t} \rangle$ be an almost-holomorphic modular form of weight $k$ and rank $s$ for $\Gamma$, and let $f := F|_{\nu=0}$ be the associated quasimodular form.
Then we have the transformation property
\[
(f|_{k,s} \gamma)(z) = \sum_{t=0}^{d} \langle F^{(t)}, (\partial \log J(\gamma,z))^{\otimes t} \rangle
\]
for all $\gamma \in \Gamma$. 
\end{prop}
\begin{proof}
Consider the equality
\begin{equation} \label{F transformation condition}
	\begin{aligned}
		F = F|_{k,s} \gamma 
		& = \sum_{t=0}^{d} \langle F^{(t)}|_{k-t,s+t} \gamma, \nu^{\otimes t}|_{t,t} \gamma \rangle \\
		& = \sum_{t=0}^{d} \langle F^{(t)}|_{k-t,s+t} \gamma, (\nu|_{1,1}\gamma)^{\otimes t} \rangle \\
		& = \sum_{t=0}^{d} \langle F^{(t)}|_{k-t,s+t} \gamma, (\nu(z) - \partial \log J(\gamma,z))^{\otimes t} \rangle
	\end{aligned}
\end{equation}
where we repeatedly used $(F \otimes G)|_{k,s+t} = F|_{k_1,s} \otimes G|_{k_2,t}$ for $k=k_1+k_2$ and Lemma~\ref{lemma:nu transformation}.
By considering the terms of degree $d$ in $\nu_j$, and then of degree $d-1$, and so on, one sees that $F^{(t)}|_{k-t,s+t}$ is uniquely determined by this equation.
We claim that
\[
F^{(t)}|_{k-t,s+t} \gamma = \sum_{t'=t}^{d} \binom{t'}{t} \left\langle F^{(t')}, ( \partial \log J(\gamma,z))^{\otimes (t'-t)} \right\rangle.
\]
Indeed, inserting this into \eqref{F transformation condition} shows that this gives the unique solution.
\end{proof}

\begin{proof}[First Proof of Theorem~\ref{thm:constant term}]
The constant term map is well-defined and surjective by definition, so
we need to show it is injective. Assume $F=\sum_{t=0}^{t_1} \langle F^{(t)}, \nu^{\otimes t} \rangle$ is an almost-holomorphic modular form of depth $t_1$ with $F^{(0)}=0$.
We need to show that $F = 0$. Assume that $F$ is non-zero and let $t_0$ be the smallest number such that $F^{(t_0)} \neq 0$.
By Proposition~\ref{prop:transformation property}, for any $\gamma \in \Gamma$, we have
\[ \sum_{t=t_0}^{t_1} \langle F^{(t)}, ( \partial \log J(\gamma,z))^{\otimes t} \rangle = 0. \]
Multiplying with $J(\gamma,z)^{t_1}$ we get
\begin{equation} \sum_{t=t_0}^{t_1} J(\gamma,z)^{t_1-t} \langle F^{(t)}, ( \partial J(\gamma,z))^{\otimes t} \rangle = 0. \label{3rsdefs} \end{equation}

For an element
\[ \gamma = \begin{pmatrix}
	a & v_e^{t} & b \\
	w_{e} & B & w_f \\
	c & v_f^{t} & d
\end{pmatrix}.
\]
in the basis $e,b_1,\ldots,b_n,f$ of $M$,
with $b_i$ a basis of $L$, we have
\[ J(\gamma,z) = -\frac{1}{2} z^2 c + (z, v_f)_L + d,
\quad \partial J(\gamma,z) = -c (z,dz)_L + (v_f,dz)_L. \]
Inserting we find the equation
\begin{equation} \label{30isdfi3}
\begin{aligned}
0 = & \langle F^{(t_0)}(z), (-c (z,dz)_L + (v_f,dz)_L)^{\otimes t_0} \rangle
\left(-\frac{1}{2} z^2 c + (z, v_f)_L + d \right)^{t_1-t_0} + \\
& \langle F^{(t_0+1)}(z), (-c (z,dz)_L + (v_f,dz)_L)^{\otimes (t_0+1)} \rangle
\left(-\frac{1}{2} z^2 c + (z, v_f)_L + d \right)^{t_1-t_0-1} + \ldots \,.
\end{aligned}
\end{equation}
For fixed $z$, the above is a polynomial in $c,v_f,d$.
By the Borel density theorem \cite[Chapter 5]{Raghunathan}, the subgroup $\Gamma \subset O(M_{\BC})$ is Zariski-dense.
We hence may take $\gamma$ to be an arbitrary element in $O(M_{\BC})$. If $S$ is the Gram matrix
of $L$ in the basis $b_i$, we can thus specialize to the element
\[ \gamma=
\begin{pmatrix}
\frac{1}{d} & & \\
-\frac{S^{-1} v_f}{d} & I_n & \\
-\frac{1}{2} \frac{v_f^{t}S^{-1} v_f}{d} & v_f^{\mathrm{t}} & d 
\end{pmatrix}
\]
for arbitrary $v_f \in L_{\BC}$ and $d \in \BC^{\ast}$.
In particular, we have $c=-\frac{1}{2} \frac{v_f^{t}S^{-1} v_f}{d}$ in the expression of $J(\gamma,z)$.
Thus treating $v_f$ and $d$ as variables, we can take the $d^{t_1-t_1}$ coefficient in \eqref{30isdfi3}
and obtain the equation
\[
0 = \langle F^{(t_0)}(z), (v_f,dz)_L^{\otimes t_0} \rangle
\]
Since $v_f$ was arbitrary and $F^{(t_0)}(z)$ symmetric,
we conclude $F^{(t_0)}(z)=0$, which is a contradiction.
\end{proof}

\begin{rmk}
The above proof is parallel to the proof in the case of classical modular forms, see \cite[Prop.3.4]{BO}.
\end{rmk}

The second proof uses reduction to the $1$-dimensional case
and is much shorter.
\begin{proof}[Seond Proof of Theorem~\ref{thm:constant term}]
Let $L' \subset L$ be a sublattice of signature $(1,n')$, giving rise to a sublattice $M'=L' \oplus U \subset M$ of signature $(2,n'+1)$.
There is a natural embedding of the corresponding period domains
$\CD' \subset \CD$
and of the groups
 $O^{+}(M') \subset O^{+}(M)$. For the $U$-cusp, we also have an embedding of the tube domains $\iota : \CC' \hookrightarrow \CC$.
Let $\Gamma'= \Gamma \cap O^{+}(M')$.
Then as explained in \cite[Chapter 4]{Ma} vector-valued modular forms for $\Gamma$ restrict to vector-valued modular forms for $\Gamma'$.
Since (viewed in the tube domain) almost-holomorphic functions are polynomials in the coefficients of $\nu = \partial \log \Im(z)^2$ and $\iota^{\ast} \nu = \nu$, we see that almost-holomorphic modular forms restrict to almost-holomorphic modular forms.
Moreover, the holomorphic limits restrict to holomorphic limits, and so quasimodular forms restrict to quasimodular forms.

Suppose now that $F$ is an almost-holomorphic modular form with vanishing constant term. Then for any positive-definite $1$-dimensional sublattice $L' = (2m) \subset L$, giving the sublattice $M'=U \oplus (2m) \subset M$, the restriction of $F$ to $\CD'$ is an almost-holomorphic modular form with vanishing constant term.
But then $F|_{\CD'}=0$ by the discussion in Section \ref{subsec:example classical case} below and the classical case proven in \cite[Prop.3.4]{BO} and \cite{KZ}.
Since the rays through positive rational vectors are dense in the positive cone, it follows that $F=0$.
\end{proof}

\subsection{Lowering and raising operators on quasimodular forms}
Since $\ct : \AHMod_{k,s}(\Gamma) \to \QMod_{k,s}(\Gamma)$ is an isomorphism, we can define
raising and lowering operators on quasimodular forms:
\begin{gather*}
\RO := \ct \circ \RO \circ \ct^{-1} : \QMod_{k,s}(\Gamma) \to \QMod_{k+1,s+1}(\Gamma), \\
\LO := \ct \circ L \circ \ct^{-1} : \QMod_{k,s}(\Gamma) \to \QMod_{k-1,s+1}(\Gamma).
\end{gather*}
Clearly $\RO = \partial$, so we find that if $F \in \QMod_{k,s}(\Gamma)$, then $\partial(F) \in \QMod_{k+1,s+1}(\Gamma)$. 

\subsection{Fourier expansion of quasimodular forms} \label{subsec:Fourier expansion of quasimodular forms}
Consider the natural inclusion
\[ L_{\BQ} \hookrightarrow O^{+}(M), \quad \xi \mapsto \exp(\xi \wedge e). \]
For our finite index subgroup $\Gamma \subset O^{+}(M)$,
define 
\[ U(\Gamma) = \Gamma \cap L_{\BQ}. \]
Then $U(\Gamma) \subset L_{\BQ}$ is a full-rank sublattice \cite[Sec.3.3.2]{Ma}.

Let $F$ be an almost-holomorphic quasimodular form of weight $k$ and rank $s$ for $\Gamma$ 
with associated quasimodular form $f=\ct(F)$.
Then for all $\xi \in U(\Gamma)$, by Example~\ref{example:Automorphy factors} we have $F(z + \xi) = F(z)$.
Since  $\nu|_{1,1} \exp(\xi \wedge e) = \nu$ (by Lemma~\ref{lemma:nu transformation}), the holomorphic part $f$ then also satisfies 
$f(z + \xi ) = f(z)$.
Hence $f$ has the Fourier expansion:
\[ f(z) = \sum_{\alpha \in U(L)^{\vee}} c_{\alpha} e^{2 \pi i \alpha(z)} \]
for some $c_{\alpha} \in (L_{\BC}^{\vee})^{\otimes s}$.
If $n \geq 3$, by the Koecher principle \cite[Sec.3.4]{Ma}, we have a non-zero Fourier-coefficient $c_{\alpha}$ only if $\alpha \in U(\Gamma)^{\vee}$ lies in the connected component of the closure of the positive cone $\{ x \in L_{\BQ} : x^2 \geq 0\}$ corresponding to the tube domain, see \cite[Sec.3.4]{Ma}.
For $n \leq 2$, this is an extra condition that we need to require in the definition of quasimodular forms,
see the next section.

\subsection{The general cusp and quasimodular forms for $n \geq 2$}
\label{subsec:general cusp}
The zero-dimensional cusps of $\CD$ correspond to primitive isotropic rank $1$ sublattices $I$ in $M$. Let $I$ be such a sublattice and let $e$ be a generator. If there exists an isotropic $f \in M$ with $e \cdot f=1$, then $M$ decomposes as $M=U\oplus L$, where $U$ is the hyperbolic plane spanned by $e,f$, and we are in the situation of this section.
In general, such an $f$ may not exist. Instead we pick some potentially non-isotropic $f \in M$ with $e \cdot f=d$, and consider the splitting
\[ M = \widetilde{U} \oplus L \]
where $\widetilde{U}$ is the saturated sublattice generated by $e,f$ and $L$ is its orthogonal complement. The tube domain then can be considered as before \cite{Ma}. The results of this sections then go through then with minor modifications that we leave to the reader. 
In particular, as before, the constant term $f(z)$ of a $\Gamma$-invariant almost-holomorphic modular form $F(z)$ will have a Fourier expansion
\[ f(z) = \sum_{\alpha \in U(L)^{\vee}} c_{\alpha} e^{2 \pi i \alpha(z)} \]
for some $c_{\alpha} \in (L_{\BC}^{\vee})^{\otimes s}$.

With this preparation we can then extend the definition of almost-holomorphic modular form and quasimodular form to the case $n \leq 2$, by declaring $F$ to be almost-holomorphic modular form and $f=\ct(F)$ to be quasimodular if additionally we require that for every $0$-dimensional cusp $I$ and for every $d \geq 0$, the Fourier-coefficients $c_{\alpha}$ of $f_d = \ct(\LO^d F)$ will be non-zero only if $\alpha \in U(\Gamma)^{\vee}$ lies in the connected component of the closure of the positive cone $\{ x \in L_{\BQ} : x^2 \geq 0\}$ corresponding to the tube domain.
(If $n \geq 3$ this condition is automatically satisfied by the K\"ocher principle.)

\subsection{Example: The classical case} \label{subsec:example classical case}
Consider the lattice $M=U \oplus (2m)$ for some $m \geq 1$,
which by the special isomorphism $\mathfrak{so}_{\BR}(2,1) \cong \mathfrak{sl}_2(\BR)$ corresponds to classical modular forms for a congruence group $\Gamma \le \SL_2(\mathbb{Z})$ (cf. \cite{Ma}). 
The bundle $\CE$ is a line bundle which is self-dual, so $\CE \cong \CO_{\CD}$ contributes no automorphy factor and so $\AHMod_{k,s}(\Gamma) = \AHMod_{k}(\Gamma)$.
Moreover the tube domain $\CC=\BH$ is the upper half plane, and if for $z \in \BH$ we write $z=x+i y$ with $x,y$ real then $\Im(z)^2 = 2m y$, which gives
\[ \nu = \partial \log(\Im(z)^2) = \frac{1}{iy} dz. \]
The lowering and raising operators are
\begin{align*}
\LO & = -2m y^2 \frac{\partial}{\partial \overline{z}} = \frac{m}{i} \LM_{2k}, \\
\RO_k & = \frac{1}{2i} \RM_{2k},
\end{align*}
where we remark that weight $k$ for orthogonal modular forms have weight $2k$ when viewed as modular forms for $\SL_2$.
Hence almost-holomorphic modular forms and quasimodular forms in our sense corresponds precisely to the classical notion.

\subsection{Example: $U \oplus U$} \label{subsec:example U+U}
We consider as an example the lattice $M=U \oplus U$,
which can be identified as a lattice with the set $M_2(\BZ)$ of integral $2 \times 2$-matrices
with quadratic form\footnote{If $(V, \langle -, - \rangle)$ is an inner product space, the quadratic form is $Q(v) = \frac{1}{2} \langle v,v \rangle$.} given by $-\det$, via the  isomorphism
\[ U \oplus U \xrightarrow{\cong} M_2(\BZ), \quad a_1 e + a_2 f + a_3 e_1 + a_4 f_1 \to \begin{pmatrix} -a_1 & a_3 \\ a_4 & a_2 \end{pmatrix}. \]
Elements in $\A(\CD)$ are thus given by complex matrices of determinant zero. The tube domain is
\[ \CC = \{ z=(z_1,z_2) | \Im(z_1),\Im(z_2) > 0 \} = \BH \times \BH \]
with the embedding
\[ \widetilde{\varphi}(z_1,z_2) = \begin{pmatrix} z_1 z_2 & z_1 \\ z_2 & 1 \end{pmatrix} \]

The group $\SL_2(\BZ) \times \SL_2(\BZ)$ acts on $M_2(\BZ)$ by $(\gamma_1, \gamma_2) \cdot A = \gamma_1 A \gamma_2^{t}$.
The group $\Gamma := O^{+}(U \oplus U)$ is generated by the image of $\SL_2(\BZ)^2$ together with the involution $S$ 
which acts as the transpose on matrices.
In particular, we have the short exact sequence
\[ 1 \to (\SL_2(\BZ) \times \SL_2(\BZ))/\langle (-I,-I) \rangle \to O^{+}(U \times U) \xrightarrow{\det} \BZ_2 \to 1. \]
The action on the tube domain is $S(z_1, z_2)=(z_2,z_1)$ and
\[ (\gamma_1, \gamma_2) \cdot (z_1, z_2) = (\gamma_1 \cdot z_1, \gamma_2 \cdot z) \]
where $\gamma_i \cdot z_i$ stands for the standard action of $\SL_2(\BZ)$ on the upper-half plane.

For any $\gamma_i=\binom{a_i\ b_i}{c_i\ d_i} \in \SL_2(\BZ)$ with $i=1,2$ we have the automorphy factors
\[
J( (\gamma_1, \gamma_2), z) = (c_1 z_1 + d_1) (c_2 z_2 + d_2), \quad J(S,z) = 1.
\]
On differentials we have
\[ (\gamma_1, \gamma_2)^{\ast}(dz_i) = \frac{1}{(c_i z_i+d_i)^2} dz_i, \quad S^{\ast}(dz_1) = dz_2. \]

A function $f(z_1,z_2)$ is called modular of bi-weight $(k_1,k_2)$ if
\[ \forall \gamma_1, \gamma_2 \in \SL_2(\BZ): \quad f\left( \gamma_1 \cdot z, \gamma_2 \cdot z \right) = (c_1 z_1 + d_1)^{k_1} (c_2 z_2 + d_2)^{k_2} f(z_1,z_2). \]
Then a modular form $f \in \Mod_k(\Gamma)$ corresponds to a function $f(z_1, z_2)$ which is modular of bi-weight $(k,k)$ and symmetric: $f(z_1,z_2)=f(z_2,z_1)$.
More generally, a modular form of rank $s$
\[
F = \sum_{i_1,\ldots, i_s=1}^{2} F_{i_1, \ldots, i_s} dz_{i_1} \otimes \cdots \otimes dz_{i_s} \quad \in \Mod_{k,s}(\Gamma)
\]
will have coefficients $F_{i_1, \ldots, i_s}$ of bi-weight $(k+\delta_1(i_1,\dots,i_s), k+\delta_2(i_1,\ldots, i_s))$, where
\[ \delta_{\ell}(i_1,\ldots,i_s) = | \{ j | i_j = \ell \}| - | \{ j | i_j \neq \ell \} |, \]
and will satisfy $F(z_2,z_1) = F(z_1,z_2)$.
For example, if
$f=f_1 dz_1 + f_2 dz_2 \in \Mod_{k,1}(\Gamma)$, then $f_1$ is of bi-weight $(k+1,k-1)$ and $f_2(z_1,z_2) = f_1(z_2,z_1)$ is of bi-weight $(k-1,k+1)$.

For the lowering and raising operators, write $y_j = \mathrm{Im}(z_j)$. Then we have
\[ |Z|^2 = \Im((z_1,z_2)^2) = 2 y_1 y_2 \]
and
\[ \partial(|Z|^2) = \frac{1}{i} ( y_2 dz_1 + y_1 dz_2 ). \]
This leads to the formula
\[
\RO(F) = \frac{1}{2i} \sum_{i_1,\ldots, i_s, \ell=1}^{2} \RM_{z_{\ell},k+\delta_{\ell}(i_1,\ldots,i_s)}(F_{i_1, \ldots, i_s}) dz_{i_1} \otimes \cdots \otimes dz_{i_s} \otimes dz_{\ell}
\]
and
\[
\LO(F) = \frac{1}{i} \sum_{i_1,\ldots, i_s, \ell=1}^{2} \LM_{z_{\ell}}(F_{i_1, \ldots, i_s}) dz_{i_1} \otimes \cdots \otimes dz_{i_s} \otimes dz_{3-\ell}
\]
where $\RM_{z,k}, \LM_{z}$ stands for the Maa{\ss} operators in the variable $z$ of weight $k$. 
In particular, if $f$ is a scalar-valued modular form of weight $k$, then
\[ \RO(f) = \frac{1}{2i} (\RM_{z_1,k}(f) dz_1 + \RM_{z_2,k}(f)  dz_2), \quad \LO(f) = \frac{1}{i} ( \LM_{z_2}(f) dz_1 + \LM_{z_1}(f) dz_2 ).\]
If $F = f_1 dz_1 + f_2 dz_2$ of weight $k$, so corresponds to a section of $\CL^{\otimes k} \otimes \CE$, then
\begin{align*}
\RO(F) & = \frac{1}{2i} \left( \RM_{z_1,k+1}(f_1) dz_1 \otimes dz_1 + \RM_{z_1,k-1}(f_2) dz_2 \otimes dz_1 +
\RM_{z_2,k-1}(f_1) dz_2 \otimes dz_1 + \RM_{z_2, k+1}(f_2) dz_2 \otimes dz_2 \right) \\
\LO(F) & = \frac{1}{i} \left( \LM_{z_2}(f_1) dz_1 \otimes dz_1 + \LM_{z_2}(f_2) dz_2 \otimes dz_1 + \LM_{z_1}(f_1) dz_1 \otimes dz_2 +\LM_{z_2}(f_2) dz_2 \otimes dz_2 \right)
\end{align*}
We find that $F$ is almost-holomorphic if every coefficient $F_{i_1,\ldots,i_s}$ is a polynomial in $1/y_1, 1/y_2$ with coefficients holomorphic functions.
The holomorphic part of $F$ is the constant term in this polynomial.
For scalar-valued quasimodular form this translates to the following concrete description:

\begin{prop} We have
\[ \QMod_{k}(\Gamma) = \Big( \QMod_{k}(\SL_2(\BZ)) \times \QMod_k(\SL_2(\BZ)) \Big)^{\BZ_2}, \]
where $\BZ_2$ acts by $f(z_1,z_2) \mapsto f(z_2,z_1)$, and similarly for (almost-holomorphic) modular forms.
\end{prop}

\subsection{Example: $U\oplus U(r)$} 
Let $M=U \oplus U(r)$, let $e,f$ be a symplectic basis of $U$, and let $e_1,f_1$ be a basis of $U(r)$ with $e_1^2=f_1^2=0$ and $e_1 \cdot f_1 = r$. We identify $M$ with the sublattice of $M_{2}(\BZ)$ with intersection form $-\det$ by
\[ U \oplus U(r) \hookrightarrow M_2(\BZ), \quad a_1 e + a_2 f + a_3 e_1 + a_4 f_1 \mapsto \begin{pmatrix} -a_1 & a_3 \\ r a_4 & a_2 \end{pmatrix}. \]
The tube domain is $\CC = \BH \times \BH$ with the embedding
\[ \tilde{\varphi}(z_1, z_2) = \begin{pmatrix} r z_1 z_2 & z_1 \\ r z_2 & 1 \end{pmatrix}. \]
The group $\Gamma_0(r) \times \Gamma_0(r)$ is embedded into $O^{+}(U \oplus U(r))$ by letting it act on $U \oplus U(r)$ by
\[ (\gamma_1, \gamma_2) \cdot M = \gamma_1 M \cdot \left( \begin{pmatrix} r & \\ & 1 \end{pmatrix} \gamma_2 \begin{pmatrix} 1/r & \\ & 1 \end{pmatrix} \right)^{t}. \]
In other words, if $\gamma = \binom{a\ b}{c\ d}$, then
\[ (1,\gamma) \cdot M = M \cdot \begin{pmatrix} a & rb \\ c/r & d \end{pmatrix}. \]
Moreover, one finds that $(\gamma_1, \gamma_2) \cdot (z_1, z_2) = (\gamma_1 \cdot z_1, \gamma_2 \cdot z_2)$ and
$J( (\gamma_1, \gamma_2),z) = J(\gamma_1, z_1) J(\gamma_2,z)$, where $J(\binom{a\ b}{c\ d}, \tau) = c \tau + d$ is the standard $\SL_2(\BZ)$ automorphy factor.

We also have the element $S \in O^{+}(U \oplus U(r))$ which fixes $e,f$ and interchanges $e_1,f_1$.
It acts by $S \cdot (z_1,z_2) = (z_2,z_1)$ with $J(S,z) = 1$.
Hence expansion in $(z_1,z_2)$-yields the inclusion
\[
\QMod_{k}(\Gamma) \subset \Big( \QMod_{k}(\Gamma_0(r)) \otimes \QMod_k(\Gamma_0(r)) \Big)^{\BZ_2}.
\]
The inclusion is not an equality, since $O^+(U \oplus U(r))$ is strictly larger for $r > 1$.\footnote{For example, $O^{+}(U \oplus U(2))$ contains the element $(F,F)$ where
$F = \frac{1}{\sqrt{2}} \begin{pmatrix} 0 & -1 \\ 2 & 0 \end{pmatrix}$
is the Fricke matrix.}

\section{Fourier expansion of the theta lift}
We determine the expansion of the theta lift of any almost-holomorphic modular form $F$, whether the lift is almost-holomorphic or not (Theorem~\ref{thm:expansion of theta lift}).
In the case of an almost-holomorphic lift, the statement simplifies significantly.
We discuss various examples, in particular the lift of almost-holomoorphic modular forms for $\Gamma_0(p)$ (Proposition~\ref{prop:lift of Gamma_0(p) ahm})

\subsection{Statement of results}
Suppose our signature $(2,n)$ lattice $M$ decomposes as
\[ M = U \oplus U_1(N) \oplus K(-1) \] where $K$ is positive-definite. Let $L = U_1(N) \oplus K(-1)$ with tube domain $\mathcal{C}$, let $e,f \in U_1(N)$ be a basis with $e \cdot f = N$ and $e^2=f^2=0$, and write vectors $z \in \mathcal{C}$ as tuples 
\[ z = (z_e, \mathfrak{z}, z_f) = u+iv \quad \text{ with } \quad u = (u_e, \mathfrak{u}, u_f), \quad v = (v_e, \mathfrak{v}, v_f). \]
The associated point $Z = (-z^2 / 2, z, 1) \in \mathcal{D}_M$ satisfies $\langle Z, \overline{Z} \rangle = 2 \langle v, v \rangle.$

Our main result is the following:

\begin{thm} \label{thm:expansion of theta lift}
Let $F(\tau) = \sum_{t=0}^d F^{(t)}(\tau) (2\pi y)^{-t}$ be an almost-holomorphic modular form of weight $\kappa = k + 1 - \mathrm{rank}(L)/2$ with $k \ge 0$, and write \begin{align*} F^{(t)}(\tau) &= \sum_{\mu \in L'/L} \sum_{n \in \mathbb{Q}} c^{(t)}(\mu; n) q^n \mathfrak{e}_{\mu} \\ &= \sum_{a, b \in \mathbb{Z}/N} \sum_{\gamma \in K'/K} \sum_{n \in \mathbb{Q}} c^{(t)}(a/N, \gamma, b/N; n) q^n \mathfrak{e}_{(a/N,\gamma,b/N)}.\end{align*}
There is an open region of the tube domain on which the theta lift $\mathrm{Lift}(F)$ is represented by the following series:
\begin{align*} \mathrm{Lift}(F) 
&= \sum_{\substack{\mu \in L^{\vee} \\ \mu > 0}} \sum_{t=0}^d c^{(t)}(\mu, \mu^2 / 2) \sum_{r=0}^{t} 4^{-r} r! \binom{t}{r} \binom{t+r-k}{r}  \pi^{-r-t} \sum_{\delta=1}^{\infty} \delta^{k-t-r-1} e^{2\pi i \langle \delta \mu, z \rangle} \langle v, v \rangle^{-t} \langle \mu, v \rangle^{t-r}  \\ 
&+ \sum_{\substack{\mu \in L^{\vee} \\ \mu < 0}} \sum_{t=k}^d c^{(t)}(\mu, \mu^2 / 2) \sum_{r=0}^t 4^{-r} r! \binom{t}{r} \binom{t+r-k}{r-k} \pi^{-r-t}  \sum_{\delta=1}^{\infty} \delta^{k-t-r-1} e^{2\pi i \langle \delta \mu, \overline{z} \rangle} \langle v,v\rangle^{-t} \langle -\mu, v \rangle^{t-r} \\ &+ \sum_{\substack{t \ge k \\ (t, k) \ne (0, 0)}} \frac{(2t-k)!}{(t-k)!} \zeta(2t+1-k) (2\pi)^{-2t} c^{(t)}(0, 0) \langle v, v \rangle^{-t}  \\ &+ \frac{1}{2} \sum_{t < k/2} \frac{t!}{(2\pi)^{2t}} \binom{2t-k}{t} \zeta(2t+1-k) c^{(t)}(0, 0)\langle v, v \rangle^{-t} \\ &+ \delta_{\substack{k \in 2\mathbb{Z} \\ k > 0}} 2^{-k} \pi^{-k-1/2} c^{(k/2)}(0, 0) \Big( \sum_{h=0}^{k/2 - 1} \frac{k! \Gamma(1/2+k/2-h)}{(-4)^h h! (k-2h)!} \sum_{j=1}^{k/2 - h} \frac{1}{2j-1} \Big) \langle v, v \rangle^{-k/2} \\ 
&+ \delta_{k=0} \frac{1}{2} c^{(0)}(0, 0) \Big(2\gamma - \ln(2\pi) - \ln(\langle v, v \rangle) \Big) \\ 
&+ \frac{(-1)^k}{2^{2k-1} \pi^{k-1}} \sum_{t \ge k/2 - 1} \frac{(-4)^t t!}{(2+2t-k)!} 
\sum_{\substack{\beta \in \mathbb{Z}/N\BZ \\ \lambda \in K'}} c^{(t)}(0, \lambda, \beta/N; \lambda^2 / 2) \\ 
&\quad\quad\quad\quad\quad\quad \times B_{2+2t-k} \Big( \{\beta/N + \langle \lambda, \mathfrak{v} \rangle / (Nv_e) \} \Big)  (N v_e)^{1+2t-k} \langle v, v \rangle^{-t}  \\ 
&+ \delta_{k=0} \Phi_0,
\end{align*}
with the following notation:
\begin{itemize}
\item $\mu > 0$ means that $\langle \mu, v \rangle > 0$ for all vectors $v \in L_{\mathbb{R}}$ that are sufficiently close to the boundary point $(0, 1) \oplus 0 \in L = U_1(N) \oplus K$, and $\mu < 0$ is defined similarly. The series converges on an open neighborhood of that point,
\item $\gamma$ is the Euler constant,
\item $B_k(x)$ are the Bernoulli polynomials defined by
$t e^{xt}/(e^t-1)=\sum_{n \geq 0}B_n(x) t^n/n!$,
\item $\Phi_0$ is defined as follows: let 
$\vartheta_{K,\gamma} = \sum_{\lambda \in K', \lambda \equiv \gamma \in K^{\vee}/K} e^{\pi i \tau \langle \lambda, \lambda \rangle_{K}}$ be the theta function of $K$ for $\gamma \in K^{\vee}/K$, and consider the expansion of the $\SL_2(\BZ)$ almost-holomorphic modular form
\[ \sum_{\gamma, \beta} F_{\gamma,\beta/N} \vartheta_{K,\gamma} = \sum_{t=0}^{d} f_t(\tau) E_2^{\ast}(\tau)^t, \]
where the modular forms $f_t(\tau)$ are holomorphic and $E_2^{\ast}(\tau)$ is the completion of the Eisenstein series $E_2(\tau) = 1 - 24 \sum_{r \geq 1} \sigma(r) q^r$.
Then
\[ \Phi_0 := \frac{\pi \langle v,v \rangle}{12 N v_e} \left[ \sum_{t=0}^{d} f_t(\tau) \frac{E_2(\tau)^{t+1}}{t+1} \right]_{q^0}. \]
\end{itemize}
\end{thm}

The theorem applies to cases where the depth condition
\[ k \geq 2d \]
which appeared in Theorem~\ref{thm:list of almost-holomorphic modular form} need not be satisfied. Hence $\mathrm{Lift}(F)$ may not be almost-holomorphic.
Under the depth condition, the expansion simplifies as follows:
\begin{cor} \label{cor:expansion of lift of almostholomorphic mod forms}
If in the situation of Theorem~\ref{thm:expansion of theta lift} 
we have the depth condition $k \geq 2d$ with $d>0$, then
the almost-holomorphic function 
$\mathrm{Lift}(F)$ has the expansion
\begin{align*} \mathrm{Lift}(F) 
&= \sum_{\substack{\mu \in L^{\vee} \\ \mu > 0}} \sum_{t=0}^d c^{(t)}(\mu, \mu^2 / 2) \sum_{r=0}^{t} 4^{-r} r! \binom{t}{r} \binom{t+r-k}{r}  \pi^{-r-t} \sum_{\delta=1}^{\infty} \delta^{k-t-r-1} e^{2\pi i \langle \delta \mu, z \rangle} \langle v, v \rangle^{-t} \langle \mu, v \rangle^{t-r}  \\ 
 &+ \frac{1}{2} \sum_{t < k/2} \frac{t!}{(2\pi)^{2t}} \binom{2t-k}{t} \zeta(2t+1-k) c^{(t)}(0, 0)\langle v, v \rangle^{-t} \\ &+ \delta_{\substack{k \in 2\mathbb{Z} \\ k > 0}} 2^{-k} \pi^{-k-1/2} c^{(k/2)}(0, 0) \Big( \sum_{h=0}^{k/2 - 1} \frac{k! \Gamma(1/2+k/2-h)}{(-4)^h h! (k-2h)!} \sum_{j=1}^{k/2 - h} \frac{1}{2j-1} \Big) \langle v, v \rangle^{-k/2} \\ 
 &+ \frac{(-1)^k}{2^{2k-1} \pi^{k-1}} \sum_{k/2 \ge t \ge k/2 - 1} \frac{(-4)^t t!}{(2+2t-k)!} 
\sum_{\substack{\beta \in \mathbb{Z}/N\BZ \\ \lambda \in K'}} c^{(t)}(0, \lambda, \beta/N; \lambda^2 / 2) \\ 
&\quad\quad\quad\quad\quad\quad \times B_{2+2t-k} \Big( \{\beta/N + \langle \lambda, \mathfrak{v} \rangle / (Nv_e) \} \Big)  (N v_e)^{1+2t-k} \langle v, v \rangle^{-t}.  \\  
\end{align*}
The expansion of its constant term is
\begin{align} \label{expansion of constant term of lift}
\ct(\mathrm{Lift}(F))
& = 
\frac{1}{2} \zeta(1-k) c^{(0)}(0,0) + 
\sum_{\substack{\mu \in L^{\vee} \\ \mu > 0}} c^{(0)}(\mu, \mu^2 / 2)  \sum_{\delta=1}^{\infty} \delta^{k-1} e^{2\pi i \langle \delta \mu, z \rangle}.
\end{align}
\end{cor}
\begin{proof}
We only need to argue that the constant term has the desired shape.
This either follows by Remark~\ref{rmk:FJ expansion is almost-hol} below which writes all terms explicitly in terms of the almost-holomorphic variables $\nu_j$, or can also be directly checked since only all terms have degree $\leq 0$ in $v$,
and the terms of degree $0$ give precisely rise to the constant terms.
\end{proof}

The last statement in particular implies that the expansion of $\ct(\Lift(F))$ can be fully expressed by the expansion of the constant term $\ct(F)$.
It also recovers the well-known formulas for the expansion of the lift of meromorphic modular forms (cf. \cite{Borcherds1998}, Theorem 14.3).

\begin{rmk} \label{rmk:FJ expansion is almost-hol}
The fact that $\Lift(F)$ is almost-holomorphic can be seen directly from the Fourier expansion. Recall that $\nu=\partial \log(v^2)$ and $v=\Im(z)$. Then we have
\[ \langle v, v \rangle^{-1} = -\langle \nu, \nu \rangle, \]
so $\langle v, v \rangle^{-1}$ is an almost-holomorphic funtion of depth $2$. Similarly,
$\langle \mu, v \rangle/\langle v,v \rangle$ for $\mu \in L_{\BQ}$ is almost-holomorphic of depth $1$ because it is of the form $\langle \nu, \mu' \rangle$ for a constant differential $\mu'$. 
Hence, in the above corollary, the first three terms in the expansion of $\mathrm{Lift}(F)$ are clearly almost-holomorphic.

The final term in Corollary \ref{cor:expansion of lift of almostholomorphic mod forms} is also almost-holomorphic. For odd $k$ this is clear by the above remarks. (Note that $v_e$ is of the form $\langle \mu, v \rangle$ for an appropriate vector $\mu$.) For even $k$ this is not trivial, since $v_e$ appears with a negative exponent and $v_e^{-1}$ is not almost-holomorphic. The necessary cancellation follows from an identity satisfied by the Fourier coefficients of quasimodular forms of depth one:
\begin{lem}
    Let $K$ be a positive-definite lattice and let $f(\tau) = f^{(0)}(\tau) + (2\pi y)^{-1} f^{(1)}(\tau)$ be an almost-holomorphic modular form of depth $1$ and weight $2 - r/2$ for the Weil representation attached to $U(N) \oplus K$. Then the Fourier coefficients $c^{(i)}(\alpha/N, \gamma, \beta/N; n)$, $\alpha,\beta \in \mathbb{Z}/N$, $\gamma \in K'/K$, $n \in \mathbb{Z} + Q(\gamma)$, $i=0,1$ of $f^{(i)}$ satisfy the identity \begin{equation}\label{eq:vanishing}\sum_{\substack{\beta \in \mathbb{Z}/N\mathbb{Z} \\ \lambda \in K'}} c^{(0)}(0, \lambda, \beta/N; \lambda^2/2) + 2 \sum_{\substack{\beta \in \mathbb{Z}/N\mathbb{Z} \\ \lambda \in K'}} c^{(1)}(0, \lambda, \beta/N; \lambda^2 / 2) \frac{\langle \lambda, v \rangle^2}{\langle v, v \rangle} = 0\end{equation} for every $v \in K \otimes \mathbb{R}$, $v \ne 0$.  
\end{lem}
\begin{proof}
Suppose first that $N=1$. For modular forms $F$ and $G$ for the Weil representation attached to $K$ and its dual, respectively, we write $$(F, G) := \sum_{\gamma \in K'/K} F_{\gamma}(\tau) G_{\gamma}(\tau)$$ and that if $F$ and $G$ have weights $k$ and $\ell$ then $(F, G)$ is a modular form of weight $k+\ell$ for $\mathrm{SL}_2(\mathbb{Z})$. The analogous statements hold for almost-holomorphic modular and quasimodular forms.

The vector-valued theta function $\theta_K(\tau) = \sum_{\lambda \in K'} q^{\langle \lambda, \lambda \rangle / 2} \mathfrak{e}_{\lambda + K}$ defines a modular form of weight $r/2$ for the dual Weil representation. Similarly, for any $v \in K \otimes \mathbb{R}$, the theta function with polynomial $\theta_{K;v} = \sum_{\lambda \in K'} \langle \lambda, v \rangle^2 q^{\langle \lambda, \lambda \rangle / 2} \mathfrak{e}_{\lambda + K}$ defines a quasimodular form of weight $r/2+2$ with almost-holomorphic completion $$\tilde \theta_{K;v}(\tau) = \sum_{\lambda \in K'} \langle \lambda, v \rangle^2 q^{\langle \lambda, \lambda \rangle / 2} \mathfrak{e}_{\lambda + K} - \frac{\langle v, v \rangle}{4\pi y} \sum_{\lambda \in K'} q^{\langle \lambda, \lambda \rangle / 2} \mathfrak{e}_{\lambda+K}.$$ By construction, the left-hand side of \eqref{eq:vanishing} is the constant Fourier coefficient of $(f^{(0)}, \theta_K) + \frac{2}{\langle v,v \rangle} (f^{(1)}, \theta_{K,v}).$
This is holomorphic in $\tau$. On the other hand, writing it as $$\Big( f - \frac{1}{2\pi y} f^{(1)}, \theta_K\Big) + \frac{2}{\langle v, v \rangle} \Big( f^{(1)}, \tilde \theta_{K;v} + \frac{\langle v,v\rangle}{4\pi y} \theta_K \Big) = [f, \theta_K] + \frac{2}{\langle v, v \rangle} [f^{(1)}, \tilde \theta_{K;v}]$$ shows that it is modular under $\mathrm{SL}_2(\mathbb{Z})$ with weight two. The claim follows because the constant Fourier coefficient of a weakly holomorphic modular form of weight two vanishes.

For general $N$, a similar argument shows that the sums $$\sum_{\lambda \in K'} c^{(0)}(\alpha/N, \lambda, \beta/N; \lambda^2 / 2) + 2 \sum_{\lambda \in K'} c^{(1)}(\alpha/N, \lambda, \beta/N; \lambda^2 / 2) \frac{\langle \lambda, v \rangle^2}{\langle v, v \rangle}$$ are the coefficients of $\mathfrak{e}_{(a/N, b/N)}$ of a vector-valued (weakly holomorphic) modular form of weight two for the Weil representation attached to the lattice $U(N)$. One can check that for any such vector-valued modular form, the sum over the components $\mathfrak{e}_{(0, \beta/N)}$ as $\beta$ runs through $\mathbb{Z}/N\mathbb{Z}$ is modular of level one.  Hence its constant term vanishes as before.
\end{proof}
When $k = 2\ell$ is even, we can apply this observation to the depth one quasimodular form $$\tilde F := \frac{(-2\pi)^{\ell-1}}{(\ell-1)!} (\LM)^{\ell-1} F = F^{(\ell-1)} + \frac{\ell}{2\pi y} F^{(\ell)}.$$ Then 
\begin{align*} 
&\quad \frac{(-1)^k}{2^{2k-1} \pi^{k-1}} \sum_{\ell \ge t \ge \ell-1} \frac{(-4)^t t!}{(2+2t-k)!} \sum_{\substack{\beta \in \mathbb{Z}/N\mathbb{Z} \\ \lambda \in K'}} c^{(t)}(0, \lambda, \beta/N; \lambda^2 / 2) \\ &\quad\quad\quad \times B_{2+2t-k}\Big( \{\beta/N + \langle \lambda, \mathfrak{v} \rangle / (Nv_e)\} \Big) (Nv_e)^{1+2t-k} \langle v, v \rangle^{-t} \\ &= \frac{(-1)^{\ell-1} (\ell-1)!}{2^{k-2}\langle v, v \rangle^{\ell} \pi^{k-1} Nv_e} \sum_{\substack{0 \le \beta < N \\ \lambda \in K'}} \Big(  \langle v, v \rangle c^{(\ell-1)}(0, \lambda, \beta/N; \lambda^2 / 2) 
 - 2\ell c^{(\ell)}(0, \lambda, \beta/N; \lambda^2 / 2) \Big[\\ &\quad\quad\quad \langle \lambda, \mathfrak{v} \rangle^2 + 2\beta \langle \lambda, \mathfrak{v}\rangle v_e + \beta^2 v_e^2 - \langle \lambda, \mathfrak{v}\rangle v_e - \beta v_e^2 + v_e^2 / 6 \Big]
 \Big)
 \\ &= \frac{(-1)^{\ell-1} (\ell-1)!}{2^{k-3} \langle v, v \rangle^{\ell} \pi^{k-1} N} \sum_{\substack{0 \le \beta < N \\ \lambda \in K'}} \Big(v_f c^{(\ell-1)}(0, \lambda, \beta/N; \lambda^2/2) - \frac{\ell}{N} c^{(\ell)}(0, \lambda, \beta/N; \lambda^2 / 2) \Big[ 2 \beta \langle \lambda, \mathfrak{v} \rangle + (\beta^2 - \beta + 1/6) v_e \Big] \Big),
\end{align*}
where in the last step, we wrote $\langle v, v \rangle = 2N v_e v_f - \langle \mathfrak{v}, \mathfrak{v} \rangle$ and applied the lemma. The remaining terms are almost-holomorphic.
\end{rmk}

\begin{example} \label{example:U+U}
We consider the theta lift for the lattice $M=U \oplus U$. The tube domain $\CC$ consists of pairs
$z = u+iv = (u_1+iv_1, u_2+iv_2)$ in $U \otimes \mathbb{C}$ with $v_1, v_2 > 0$, i.e. it is just $\mathbb{H} \times \mathbb{H}$. The positive cone with respect to the point $(0, 1)$ consists of vectors $\mu = (a,b) \in U$ with $\langle \mu, v \rangle = a v_2 + b v_1 > 0$ for all $v$ sufficiently close to $(0, 1)$; i.e. either $a > 0$, or $a=0$ and $b > 0$. Similarly the negative cone consists of $(a, b)$ with either $a < 0$, or $a=0$ and $b < 0$.

(1) Let $F$ be the constant function $1$, with depth $d=0$ and weight $k = 0$. The contribution of the first term in Theorem~\ref{thm:expansion of theta lift} is $$\sum_{\mu > 0} c^{(t)}(\mu, \mu^2 / 2) \sum_{\delta=1}^{\infty} \delta^{-1} e^{2\pi i \langle \delta, \mu, z \rangle} + \sum_{\mu < 0} c^{(t)}(\mu, \mu^2 / 2) \sum_{\delta = 1}^{\infty} \delta^{-1} e^{2\pi i \langle \delta \mu, \overline{z} \rangle}.$$ Since $F$ has no coefficients with $\mu^2 \ne 0$, we sum only over vectors $(a, 0)$ or $(b, 0)$ with $a,b \ne 0$. So the sum becomes \begin{align*} &\quad \sum_{\delta=1}^{\infty} \sum_{a=1}^{\infty} \delta^{-1} e^{2\pi i \delta a z_2} + \sum_{\delta=1}^{\infty} \sum_{b=1}^{\infty} \delta^{-1} e^{2\pi i \delta b z_1} + \sum_{a=1}^{\infty} \delta^{-1} e^{2\pi i \delta (-a \overline{z_2})} + \sum_{\delta=1}^{\infty} \sum_{b=1}^{\infty} \delta^{-1} e^{2\pi i \delta (-b \overline{z_1})} \\ &= -\sum_{a=1}^{\infty} \ln(1 - e^{2\pi i a z_2}) - \sum_{b=1}^{\infty} \ln(1 - e^{2\pi i bz_2}) - \sum_{a=1}^{\infty} \ln(1 - e^{-2\pi i a \overline{z_2}}) - \sum_{b=1}^{\infty} \ln(1 - e^{-2\pi i b \overline{z_1}}) \\ &= -\ln \Big| \prod_{a=1}^{\infty} (1 - e^{2\pi i az_2}) \cdot \prod_{b=1}^{\infty} (1 - e^{2\pi i bz_1}) \Big|^2.
\end{align*}
The sixth line in the formula contributes $\frac{1}{2}(2\gamma - \ln(2\pi) - \ln(2y_1 y_2))$.
The seventh line is $$\pi \cdot v_e B_2(0) = \frac{\pi}{6} y_1.$$
Finally, the constant term of $\frac{\pi}{3} F(\tau) E_2(\tau)$ is $\frac{\pi}{3}$ so the last line of the theorem is $$\Phi_0 = \frac{\pi}{3} \cdot \frac{\langle v, v \rangle}{4 v_e} = \frac{\pi}{3} \cdot \frac{2 y_1 y_2}{4 y_1} = \frac{\pi}{6} y_2.$$
Altogether we obtain $$\mathrm{Lift}(F) = C - \frac{1}{2} \ln(y_1 y_2) - \ln |\eta(z_1)^2 \eta(z_2)^2|$$ with the constant $C = \gamma - \frac{1}{2}\ln(4\pi)$.

(2) Let $F(\tau) = D_{\tau}(E_{10} / \Delta)$ where $D_{\tau} = \frac{1}{2 \pi i} \frac{d}{d \tau}$. This is a quasimodular form of weight $0$ whose modular completion is \begin{align*} F^*(\tau) &= F(\tau) + \frac{E_{10}/\Delta}{2\pi y} \\ &= -q^{-1} - 141444q - 17058560q^2 - ... \\ &\quad + \frac{1}{2\pi y} (q^{-1} - 240 - 141444q - 8529280q^2 - ...). \end{align*}
The main term contribution (the first two summands) is 
\begin{align*} &\quad \sum_{(a, b) > 0} c^{(0)}(ab) \sum_{\delta=1}^{\infty} \delta^{-1} (e^{2\pi i \delta (a z_2 + b z_1)} + e^{-2\pi i \delta (a \overline{z_2} + b \overline{z_1})}) \\ &+ \frac{1}{2\pi y_1 y_2} \sum_{(a, b) > 0} c^{(1)}(ab) (a y_2 + b y_1)  \sum_{\delta=1}^{\infty} \delta^{-2} (e^{2\pi i \delta (a z_2 + b z_1)} + e^{-2\pi i \delta (a \overline{z_2} + b \overline{z_1})}) \\ &+ \frac{1}{8\pi^2 y_1 y_2} \sum_{(a, b) > 0} c^{(1)}(ab) \sum_{\delta=1}^{\infty} \delta^{-3} (e^{2\pi i \delta (a z_2 + b z_1)} + e^{-2\pi i \delta (a \overline{z_2} + b \overline{z_1})}).\end{align*}
The third line of the formula contributes $$2 \zeta(3) \cdot (2\pi)^{-2} \cdot (-240) \cdot \frac{1}{2y_1 y_2} = -\frac{60 \zeta(3)}{\pi^2 y_1 y_2}.$$
The fourth and fifth line do not contribute because $k=0$.
The sixth line of the theorem contributes $0$ because $c^{(0)}(0) = 0$. The seventh line involves only the constant term $c^{(1)}(0) = -240$ and gives \begin{align*} &\quad 2\pi \cdot \frac{(-4)}{24} \cdot 240 \cdot y_1^3 B_4(0) \cdot \frac{1}{2y_1 y_2} = \frac{4\pi y_1^2}{3 y_2}.\end{align*}

We can write $F^* = 576 - \frac{5}{6} j - \frac{1}{6} \frac{E_{10}}{\Delta} E_2^*$ where $j$ is the modular $j$-function. The constant term of $$\frac{\pi}{3} \Big( 576 - \frac{5}{6} j \Big) E_2 + \frac{\pi}{6} \Big( -\frac{1}{6} \frac{E_2^2 E_{10}}{\Delta} \Big)$$ is $\frac{\pi}{3} \cdot (-24) + \frac{\pi}{6} \cdot (48) = 0$ and therefore $\Phi_0 = 0$.

Writing $E_{10}/\Delta = \sum_{n} d(n) q^n$, one obtains:
\begin{align*} \Lift\left( D_{\tau}(E_{10}/\Delta) \right) = 
&\quad \sum_{(a, b) > 0} ab \cdot d(ab) \sum_{\delta=1}^{\infty} \delta^{-1} (e^{2\pi i \delta (a z_2 + b z_1)} + e^{-2\pi i \delta (a \overline{z_2} + b \overline{z_1})}) \\ 
&+ \sum_{(a, b) > 0} d(ab) \left(\frac{a}{2\pi y_1} + \frac{b}{2\pi y_2} \right)  \sum_{\delta=1}^{\infty} \delta^{-2} (e^{2\pi i \delta (a z_2 + b z_1)} + e^{-2\pi i \delta (a \overline{z_2} + b \overline{z_1})}) \\ 
&+ \frac{1}{8\pi^2 y_1 y_2} \sum_{(a, b) > 0} d(ab) \sum_{\delta=1}^{\infty} \delta^{-3} (e^{2\pi i \delta (a z_2 + b z_1)} + e^{-2\pi i \delta (a \overline{z_2} + b \overline{z_1})}) \\
& -\frac{60 \zeta(3)}{\pi^2 y_1 y_2} + \frac{4\pi y_1^2}{3 y_2}.
\end{align*}
The function $\mathrm{Lift}(D(E_{10}/\Delta))$ is an example of a higher Green's function and the above expansion is a special case of \cite[Eqn.8.6]{Viazovska}.

Let us apply the orthogonal raising operator $\RO = \RO_{z_1} dz_1 + \RO_{z_2} dz_2$ to this function.
Since the lift is of weight zero, we simply have $\RO_{z_1} = 
\frac{d}{dz_1}$. We then get
\begin{align*}
\frac{1}{2 \pi i} \RO_{z_1}
\Lift\left( D_{\tau}(E_{10}/\Delta) \right)
= 
&\quad \sum_{(a, b) > 0} a b^2 \cdot d(ab) \sum_{\delta=1}^{\infty} (e^{2\pi i \delta (a z_2 + b z_1)} \\ 
&+ \sum_{(a, b) > 0} b \cdot d(ab) \left(\frac{a}{2\pi y_1} + \frac{b}{2\pi y_2} \right)  \sum_{\delta=1}^{\infty} \delta^{-1} e^{2\pi i \delta (a z_2 + b z_1)} \\
&+ \frac{1}{8 \pi^2 y_1^2} \sum_{(a, b) > 0} a \cdot d(ab) \sum_{\delta=1}^{\infty} \delta^{-2} (e^{2\pi i \delta (a z_2 + b z_1)} + e^{-2\pi i \delta (a \overline{z_2} + b \overline{z_1})}) \\
&+ \frac{1}{8\pi^2 y_1 y_2} \sum_{(a, b) > 0} b \cdot d(ab) \sum_{\delta=1}^{\infty} \delta^{-2} e^{2\pi i \delta (a z_2 + b z_1)} \\
&+ \frac{1}{32 \pi^2 y_1^2 y_2} \sum_{(a, b) > 0} d(ab) \sum_{\delta=1}^{\infty} \delta^{-3} (e^{2\pi i \delta (a z_2 + b z_1)} + e^{-2\pi i \delta (a \overline{z_2} + b \overline{z_1})}) \\
& - \frac{15 \zeta(3)}{\pi^3 y_1^2 y_3}
- \frac{2 y_1}{3 y_2}.
\end{align*}

Thus contrary to the holomorphic case of Proposition~\ref{prop:R on Borcherds lift}, we see that the derivative
$\RO(\Lift(D_{\tau}(E_{10}/\Delta)))$ is not
an almost-holomorphic modular form.
A similar expansion is obtained for
the lift of $D_{\tau}^2(E_{10}/\Delta)$
and again shows that it is not almost-holomorphic, in agreement with
Theorem~\ref{thm:list of almost-holomorphic modular form}.




(3)
Take the form $G_2^*(\tau) = G_2(\tau) + \frac{1}{8 \pi y}$ as input into the theta lift on $U \oplus U$. If $\mu = (a, b)$ then $c^{(0)}(\mu, \mu^2 / 2) = \sigma(ab)$ and $c^{(1)}(\mu, \mu^2 / 2) = 0$ for $a, b > 0$; and if $\mu = (a, 0)$ with $a \ge 0,$ then $c^{(0)}(\mu, \mu^2 / 2) = -1/24$ and $c^{(1)}(\mu, \mu^2 / 2) = \frac{1}{4}$. We have $\langle v, v \rangle /2 = y_1 y_2$ and $\langle \mu, v \rangle = a y_2$ for $\mu = (a, 0)$. We obtain
    \begin{align*}
        \mathrm{Lift}(G_2^*) = & \sum_{b=1}^{\infty} c^{(0)}(0) \sum_{\delta=1}^{\infty} \delta e^{2\pi i \delta b z_1} + \sum_{a=1}^{\infty} c^{(0)}(0) \sum_{\delta=1}^{\infty} \delta e^{2\pi i \delta a z_2}  +\sum_{a=1}^{\infty} \sum_{b=1}^{\infty} \sigma_1(ab) \sum_{\delta=1}^{\infty} \delta e^{2\pi i \delta (a z_2 + b z_1)} \\ &+ \frac{1}{2} \zeta(-1)  c^{(0)}(0) + 2^{-2} \pi^{-5/2}  c^{(1)}(0) \Gamma(3/2) \frac{1}{2 y_1 y_2} + \frac{1}{8\pi} c^{(0)}(0) y_1^{-1} \\
        & + \frac{1}{8\pi} \cdot \frac{(-4)}{2!} c^{(1)}(0) y_1 B_2 \frac{1}{2y_1 y_2} \\
        = & -\frac{1}{24} \sum_{n=1}^{\infty} \sigma_1(n) (e^{2\pi i n z_1} + e^{2\pi i n z_2}) + \sum_{b=1}^{\infty} \sum_{a=1}^{\infty} \sum_{d | \mathrm{gcd}(a, b)} \sigma_1(ab / d^2) e^{2\pi i (b z_1 + a z_2)} \\ & \quad + \left( -\frac{1}{24} + \frac{1}{8 \pi y_1} \right) \left( -\frac{1}{24} + \frac{1}{8 \pi y_2} \right) \\
        =& G_2^*(z_1) G_2^*(z_2),
        \end{align*}
    using the fact that $E_2^*$ is an eigenform of the Hecke operators to simplify $$\sum_{d | \mathrm{gcd}(a, b)} \sigma_1(ab/d^2) = \sigma_1(a) \sigma_1(b).$$
	

This behaviour is forced for Hecke eigenforms:
\begin{lemma}
If $f \in \AHMod_{k}(\SL_2(\BZ))$ is a normalized eigenform of depth $d$ with $k \geq 2d$,
then 
\[ \Lift(f) = f(z_1) f(z_2). \]
\end{lemma}
\begin{proof}
Let $\ct(f)= \sum_{n} c(n) q^n$ be the Fourier expansion of the constant term.
Since $\Lift(f)$ is almost-holomorphic, it is determined by its constant term which is given by
\begin{align*}
\ct(\Lift(f))=& \frac{1}{2} \zeta(1-k) c(0)
+ c(0) \sum_{n=1}^{\infty} \sigma_{k-1}(n) (e^{2\pi i n z_1} + e^{2\pi i n z_2}) \\
& + \sum_{a,b=1}^{\infty} \sum_{d | \mathrm{gcd}(a, b)} d^{k-1} c(ab/d^2) e^{2\pi i (b z_1 + a z_2)}.
\end{align*}
Since $f$ is a normalized eigenform, we have
\[ c(a)c(b) = \sum_{\delta | (a,b)} \delta^{k-1} c\left( \frac{ab}{\delta^2} \right). \]
This implies the claim if $c(0)=0$.
If $c(0) \neq 0$, then $f=G_k(\tau)$ (or its completion for $k=2$) so the claim holds as well.
\end{proof}
\end{example}


\subsection{Applications}
\label{subsec:theta lift for U Up N(-p)}
We now specialize to the case
\[ M = U \oplus U(p) \oplus N(-p), \]
for some prime $p$, where $N$ is a positive-definite lattice of rank $n-2$ and level\footnote{A lattice $M$ is of signature $b$ if $M^{\vee}(b)$ is even integral.} $1$, so a direct sum of copies of the $E_8$-lattice.
For coordinates $z \in L \otimes \BC$ we write $z=(\tau_1, \Fz, \tau_2) = \tau_1 e + \Fz + \tau_2 f$, where $e,f \in U(p)$ are the standard basis and $\Fz \in N \otimes \BC$.

\begin{prop} \label{prop:Gamma_0(p) to vector valued}
Let 
$f \in \AHMod_\kappa(\Gamma_0(p))$
be a modular form with Fricke involute
$$f^{\mathrm{Fr}} (\tau) := f \Big|_{\kappa} \begin{pmatrix} 0 & -1/\sqrt{p} \\ \sqrt{p} & 0 \end{pmatrix} = p^{-\kappa/2} \tau^{-\kappa} f \Big(-\frac{1}{p\tau} \Big)\ \in \AHMod_\kappa(\Gamma_0(p)).$$
For $j=0,\ldots, p-1$ let
\[
g_j(\tau) = p^{-1} \sum_{\ell=0}^{p-1} e^{-\frac{2 \pi i j \ell}{p}} f((\tau+\ell)/p).
\]
Then the vector-valued function
\begin{equation} 
F_{f}(\tau) = p^{(\kappa+\rk(N))/2} f^{\Fricke}(\tau) \mathfrak{e}_{(0,0)} + 
\sum_{\gamma} g_{\frac{p\langle \gamma, \gamma\rangle}{2}}(\tau) e_{\gamma}
\end{equation}
is an almost-holomorphic modular form with respect to the finite Weil representation of $\SL_2(\BZ)$ on $\BC[M'/M]$. 

If $\ct(f)= \sum_{n=0}^{\infty} d_n q^n$
and $\ct(f^{\mathrm{Fr}})=\sum_{n=0}^{\infty} c_n q^n$,
then
$$\ct(F_{f}(\tau)) = p^{(\kappa+\rk(N))/2} \sum_{n=0}^{\infty} c_n q^n \mathfrak{e}_{(0,0)} + 
\sum_{\substack{a,b \in \BZ/p \\ \alpha \in N}} 
\sum_{\substack{n \equiv a b - \frac{1}{2}(\alpha,\alpha)_N \\ (\text{mod }p)}} d_n q^{n/p} e_{(a/p,b/p,\alpha/p)}.$$
Moreover
\[
\LM(F_f(\tau)) = p \cdot F_{\LM(f)}.
\]
\end{prop}
\begin{proof}
For any even lattice $M$ even rank and level $p$ for some prime $p$ define the quadratic Nebentypus character 
$\chi_M : \Gamma_0(p) \longrightarrow \{\pm 1\}$ given by
$$\chi_M\Big( \begin{pmatrix} a & b \\ c & d \end{pmatrix} \Big) = \left( \frac{a}{|\mathrm{det}(M)|} \right)$$
where $\left( \frac{\cdot}{\cdot} \right)$ is the quadratic reciprocity symbol. Scheithauer \cite{S2009} then showed that for any modular form $h \in \AHMod_{\kappa}(\Gamma_0(p), \chi_M)$ the sum
$$F(\tau) := \sum_{A \in \Gamma_0(p) \backslash \SL_2(\BZ)} (h |_{\kappa} A(\tau)) \rho_M(A^{-1}) \mathfrak{e}_0$$ is a modular form with respect to the Weil representation of $\mathrm{SL}_2(\mathbb{Z})$ on $\mathbb{C}[M'/M]$ (the rank $M$ is even, so the Weil representation factors always through $\SL_2(\BZ)$).

We specialize now to
$M = U \oplus U(p) \oplus N(-p)$, where $\det(M) = p^{2+\rk(N)}$ and so the character $\chi_M$ is trivial: $\chi_M=1$.
Since the coset $\Gamma_0(p) \backslash \SL_2(\BZ)$ can be represented by the set $\{ \id \} \cup \{ S T^{\ell} | 0 \leq \ell \leq p-1 \}$, 
the vector-valued modular form is
\[
F(\tau) = h(\tau) e_{(0,0)} + p^{-\rk(N)/2-1} \sum_{\gamma} \sum_{\ell=0}^{p-1} \tilde{h}(\tau + \ell) e^{- 2 \pi i \ell \langle \gamma,\gamma \rangle/2} e_{\gamma}
\]
where $\tilde{h}(\tau) = \tau^{-\kappa} h(-1/\tau)$.
We now set $h(\tau) = f^{\mathrm{Fr}}(\tau)$
whch implies that $\tilde{h}(\tau) = p^{-\kappa/2} f(\tau/p)$, and we define $F_{f} := p^{(\kappa+\rk(N))/2} F$. This gives the first claim.
Moreover
\[ \ct(g_j(\tau)) = \sum_{n \equiv j \text{ mod } p} d_n q^{n/p} \]
which gives the Fourier expansion. For an almost-holomorphic function $g$ we have
\[ \LM(g(\tau/p)) = p \cdot \LM(g)(\tau/p) \]
which gives the last part.
\end{proof}

For $a,b>0$ and $\alpha \in N$ let us write
\[ q_1^{a} q_2^{b} \zeta^{\alpha} = e^{2 \pi i a \tau_1} e^{2 \pi i b \tau_2} e^{-2 \pi i (\alpha, \Fz)_N} = e^{2 \pi i (\frac{a}{p} e_1 + \frac{b}{p} f_1 + \frac{\alpha}{p}, \tau_2 e_1 + \tau_1 f_1 + \Fz)_{U(p) \oplus N(-p)}} \]
where $\tau_1, \tau_2 \in \BH$ and $\Fz \in N \otimes \BC$.
Assume that $f$ and hence $F$ satisfies the depth condition $k \geq 2d$, where $\kappa=k+1-\frac{n}{2}$. Then the constant term of the lift of $F_f$ is given by the formula:
\begin{align*}
\ct(\Lift(F_f))
& = \frac{1}{2} \zeta(1-k) p^{(\kappa+\rk(N))/2}  c_0 + 
p^{(\kappa+\rk(N))/2} 
\sum_{a,b>0,\alpha} \sum_{\delta \geq 1} \delta^{k-1} 
c_{p(ab - \frac{1}{2} (\alpha,\alpha)_N)} q_1^{p\delta a} q_2^{p \delta b} \zeta^{p\delta  \alpha} \\
& \quad + \sum_{a,b>0,\alpha} \sum_{\delta \geq 1} \delta^{k - 1} d_{ab - \frac{1}{2} (\alpha,\alpha)_N} q_1^{\delta a} q_2^{\delta b} \zeta^{\delta \alpha} \\
& = \frac{1}{2} p^{(\kappa+\rk(N))/2}  \zeta(1-k) c_0 + \sum_{\ell \geq 1} q_2^{\ell} \sum_{\substack{\tilde{\delta} | \ell \\ \tilde{\delta} \equiv 0 (\text{mod }p)}}
\frac{\tilde{\delta}^{k-1}}{p^{\kappa/2-1}}
\sum_{a, \alpha} c_{p(\frac{a \ell}{\tilde{\delta}} - \frac{(\alpha,\alpha)_N}{2})}
q_1^{\tilde{\delta} a} \zeta^{\tilde{\delta} \alpha} \\
& \quad + \sum_{\ell \geq 1} q_2^{\ell} \sum_{\delta | \ell} \sum_{a, \alpha}
\delta^{k-1}
d_{\frac{a \ell}{\delta} - \frac{(\alpha,\alpha)}{2}} q_1^{\delta a} \zeta^{\delta \alpha}.
\end{align*}
Assume that we have the condition
\begin{equation} \label{cofficieent condition}
\forall n : \quad 
c_{p n} + p^{\frac{\kappa}{2}-1} d_n = 0
\end{equation} 
Then the above simplifies to
\begin{align}
\ct(\Lift(F_f))
& = \frac{1}{2} p^{(\kappa+\rk(N))/2}  \zeta(1-k) c_0
+ \sum_{\ell \geq 1} q_2^{\ell} \sum_{\substack{\delta | \ell \\ \gcd(\delta,p) = 1}} \sum_{a, \alpha}
\delta^{k-1}
d_{\frac{a \ell}{\delta} - \frac{(\alpha,\alpha)}{2}} q_1^{\delta a} \zeta^{\delta \alpha} \\
& = \frac{1}{2} p^{(\kappa+\rk(N))/2} \zeta(1-k) c_0 
+ \sum_{\ell \geq 1} \sum_{a,\alpha} q_2^{\ell} q_1^{a} \zeta^{\alpha} \sum_{\substack{\delta | (a,\ell,\alpha) \\ \gcd(\delta,p) = 1}}
\delta^{k-1}
d_{\frac{a \ell}{\delta^2} - \frac{(\alpha,\alpha)}{2 \delta^2}}. \label{edg3r224r2}
\end{align}

\begin{prop} \label{prop:lift of Gamma_0(p) ahm}
Let $f \in \AHMod_{\kappa}(\Gamma_0(p))$ be a modular form satisfying $k \geq 2d$ where $\kappa=k+1-\frac{n}{2}$ and with Fourier expansions
 $\ct(f)= \sum_{n=0}^{\infty} d_n q^n$
and $\ct(f^{\mathrm{Fr}})=\sum_{n=0}^{\infty} c_n q^n$.
The lift $\Lift(F_f)$ is a modular form with respect to $O^{+}(M)$.
If moreover 
$c_{p n} + p^{\frac{\kappa}{2}-1} d_n = 0$ for all $n$,
then we have the expansion
\begin{equation} \label{Lift expansion jo} \ct(\Lift(F_f))
=
\frac{1}{2} p^{(\kappa+\rk(N))/2} \zeta(1-k) c_0 + \sum_{\ell \geq 1} \sum_{a,\alpha} q_2^{\ell} q_1^{a} \zeta^{\alpha} \sum_{\substack{\delta | (a,\ell,\alpha) \\ \gcd(\delta,p) = 1}}
\delta^{k-1}
d_{\frac{a \ell}{\delta^2} - \frac{(\alpha,\alpha)}{2 \delta^2}}.
\end{equation}
\end{prop}
\begin{proof}
It remains to show that $\Lift(F_f)$ is invariant under $O^{+}(M)$. This follows from \eqref{tr property}
since $F_f$ is invariant under $O^{+}(M)$ 
(the coefficients of $e_{\gamma}$ for $\gamma \neq 0$ only depend on $\langle \gamma, \gamma \rangle$).
\end{proof}
\begin{rmk}
For $f \in \AHMod_{\kappa}(\Gamma_0(p))$ let
\[ \Tr(f) := \sum_{\gamma \in \Gamma_0(p)\backslash \SL_2(\BZ)} f|_{\kappa}\gamma. \] Then $\Tr(f)=0$ if and only if the coefficient identity $c_{p n} + p^{\frac{\kappa}{2}-1} d_n = 0$ for all $n$ holds.
Indeed, since we have the coset representatives $\Gamma_0(p)\backslash \SL_2(\BZ) = \{ \id \} \sqcup \{ S T^{\ell} | \ell=0,\ldots,p-1 \}$,
this follows from a direct computation.
\end{rmk}

Define the holomorphic theta function of the lattice $N$ by
\[ \vartheta_{N}(\Fz, \tau) = \sum_{\alpha \in N} \zeta^{\alpha} q^{(\alpha,\alpha)_N/2}. \]
Moreover, recall that for a Jacobi form $f \in \Jac_{k,N}(\Gamma_0(p))$ with Fourier expansion
\[ f(\Fz,\tau) = \sum_{n} \sum_{\alpha \in N} c(n,\alpha) q^n \zeta^\alpha \]
the $\ell$-th Hecke operator is given by
\begin{equation} \label{Fourier expansion}
 (f|_{k,N} V_{\ell})(\Fz,\tau)
=
 \sum_{n} \sum_{\alpha \in N} q^n \zeta^{\alpha}
 \sum_{\substack{ \delta|(n,\alpha,\ell) \\ \mathrm{gcd}(\delta,p)=1 }}
 \delta^{k-1} c\left( \frac{\ell n}{\delta^2}, \frac{\alpha}{\delta} \right).
 \end{equation}
More generally, for any formal power series $f$, 
we define a formal Hecke operator
$(f|_{k,N} V_{\ell})$ by the right hand side of \eqref{Fourier expansion}.
Then \eqref{Lift expansion jo} can be rewritten as
\[ \ct(\Lift(F)) = \frac{1}{2} p^{(\kappa+\rk(N))/2}  \zeta(1-k) c_0 + \sum_{\ell \geq 1} \Big( (\ct(f) \vartheta_N(\Fz,\tau_1))|_{k,N}V_{\ell} \Big) q_2^{\ell}. \]

\begin{cor}
If $M=U \oplus U(p)$ and $f \in \AHMod_{k}(\Gamma_0(p))$ is a normalized Hecke eigenform with $\Tr(f)=0$ (e.g. a newform), satisfying $k \geq 2d$,
then
\[ \Lift(F_f) = f(\tau_1) f(\tau_2). \]
\end{cor}
\begin{proof}
Since $f$ is a Hecke eigenform we have
\[ \ct(f)|_{k} V_{\ell} = \ct( f|_{k} V_{\ell}) =  c(\ell) \ct(f) \]
where $\ct(f) = \sum_{d} c(d) q^d$.
Hence we find
\[ \ct(\Lift(F_f)) = \ct(f)(\tau_1) \ct(f)(\tau_2). \]
Since $\ct$ is an isomorphism, we conclude that
\[ \Lift(F_f) = f(\tau_1) f(\tau_2). \]
\end{proof}

\subsection{Proof}
In the rest of this section we discuss the proof of Theorem~\ref{thm:expansion of theta lift}.

The theta kernels $\Theta_k$ for $M$ can be expressed in terms of Poincar\'e series over theta kernels $\vartheta_k$ and $\theta$ for the sublattices $L$ and $K$, respectively. This is a reformulation of Theorem 7.1 of \cite{Borcherds1998}, applied twice.

For $\lambda \in L$, let $$\lambda_v = \frac{\langle \lambda, v \rangle}{\langle v, v \rangle} v$$ be the orthogonal projection to the positive-norm vector $v$ and write $\lambda_v^{\perp} = \lambda - \lambda_v$. Then $$\langle \lambda_v, \lambda_v \rangle = \frac{\langle \lambda, v \rangle^2}{\langle v, v \rangle}, \quad \langle \lambda_v^{\perp}, \lambda_v^{\perp} \rangle = \frac{\langle \lambda, \lambda \rangle \langle v, v \rangle - \langle \lambda, v \rangle^2}{\langle v, v \rangle}.$$

\begin{defn} For $k \in \mathbb{N}_0$, define the polynomial $$p_k(x; z) := \sum_{j=0}^{\lfloor k/2 \rfloor} \Big( -\frac{\langle v, v \rangle}{8\pi} \Big)^j \frac{k!}{j! (k - 2j)!} \cdot \langle x, v \rangle^{k - 2j}, \quad x \in L_{\mathbb{R}}$$
\end{defn}

In other words, $$p_k = \sum_{j=0}^{\infty} \frac{1}{j!} \Big( -\frac{\Delta}{4\pi} \Big)^j \langle x, v \rangle^k = e^{-\Delta / 4\pi} \langle x, v \rangle^k$$ where $\Delta=\sum_{i,j} g^{ij} \frac{\partial^2}{\partial z_i \partial z_j}$ is the Laplace operator.

The function $$f(x) := p_k(x, z) \cdot e^{\pi \langle x_v^{\perp}, x_v^{\perp} \rangle - \pi \langle x_v, x_v \rangle}$$ is then a Schwartz function that satisfies the Vign\'eras equation $$\omega(Y - X) f = -i \kappa f, \quad \kappa = k + 1 - n/2.$$
We denote the associated theta function by $\vartheta_k$. Explicitly, $$\vartheta_k(\tau, z) := \sum_{j=0}^{\lfloor k/2 \rfloor} \frac{k!}{j! (k-2j)!} y^{(n-1)/2 - j}  \Big( -\frac{\langle v, v \rangle}{8\pi} \Big)^j \sum_{\lambda \in L'} \langle \lambda, v \rangle^{k-2j} e^{\pi i \langle \lambda_v, \lambda_v \rangle \tau + \pi i \langle \lambda_v^{\perp}, \lambda_v^{\perp} \rangle \overline{\tau}} \mathfrak{e}_{\lambda + L}.$$

More generally, for $\alpha, \beta \in \mathbb{R}$, define \begin{align*} &\quad \vartheta_k\Big( \tau, z; \begin{pmatrix} \alpha \\ \beta \end{pmatrix} \Big) \\ &= \sum_{j=0}^{\lfloor k/2 \rfloor} \frac{k!}{j! (k-2j)!} y^{(n-1)/2 - j} \Big( -\frac{\langle v, v \rangle}{8\pi} \Big)^j \sum_{\lambda \in L'} \langle \lambda + \beta u, v \rangle^{k-2j} \\ &\quad\quad \times \exp \Big(\pi i \langle (\lambda + \beta u)_v, (\lambda + \beta u)_v \rangle \tau + \pi i \langle (\lambda + \beta u)_{v^{\perp}}, (\lambda + \beta u)_{v^{\perp}} \rangle \overline{\tau} - \pi i \alpha \beta \langle u, u \rangle - 2\pi i \alpha \langle \lambda, u \rangle \Big) \mathfrak{e}_{\lambda + L}, \end{align*} such that $$\vartheta_k(\gamma \cdot \tau, z; \gamma \zeta) = (c \tau + d)^{\kappa} \rho_L(\gamma) \vartheta_k(\tau, z; \zeta)$$ for every $\gamma = \begin{pmatrix} a & b \\ c & d \end{pmatrix} \in \mathrm{Mp}_2(\mathbb{Z})$ and every $\zeta = (\alpha, \beta)^T \in \mathbb{R}^2$.

\begin{defn} The theta function associated with $K$ is $$\theta_K(\tau) := y^{n/2 - 1} \sum_{\lambda \in K'} e^{\pi i \langle \lambda, \lambda \rangle_{K} \overline{\tau}} \mathfrak{e}_{\lambda + K}.$$
\end{defn}

Similarly to $\vartheta_k$, for $\alpha, \beta \in \mathbb{R}$ we define $$\theta_K\Big( \tau; \begin{pmatrix} \alpha \\ \beta \end{pmatrix} \Big) = y^{n/2 - 1} \sum_{\lambda \in K'} e^{\pi i \langle \lambda + \beta \mathfrak{v}/v_1, \lambda + \beta \mathfrak{v}/v_e \rangle \overline{\tau} - \pi i \alpha \beta \frac{\langle \mathfrak{v}, \mathfrak{v} \rangle}{v_e^2} - 2\pi i \alpha \frac{\langle \lambda, \mathfrak{v} \rangle}{v_e}} \mathfrak{e}_{\lambda + K}.$$

\begin{thm}\label{thm:theta_poincare} If $k \ge 1$, then the theta kernel $\Theta_k$ has the representation $$\Theta_k(\tau, Z) = \mathbb{P}_{\kappa, \rho_M}(\phi_k) := \sum_{\gamma \in \Gamma_{\infty} \backslash \Gamma} j(\gamma; \tau)^{-\kappa} \rho_M(\gamma) \phi(\gamma \cdot \tau)$$ as a Poincar\'e series with seed function \begin{align*} \phi_k(\tau) &= 2(-1)^k y^{-k} \Big( \frac{\langle v, v \rangle}{2Nv_e} \Big)^{k+1} \sum_{\beta \in \mathbb{Z}/N\mathbb{Z}} \sum_{\delta=1}^{\infty} \delta^k e^{-\frac{\pi \delta^2}{2y (N v_e)^2}  \langle v, v \rangle - 2\pi i \frac{\delta \beta}{N}} \mathfrak{e}_{(0, \beta / N)} \otimes \theta_K\Big( \tau; \begin{pmatrix} \delta/N \\ 0 \end{pmatrix} \Big) \\ &+ 2 i^k \left( \frac{\langle v, v \rangle}{2} \right)^{1/2} \sum_{j=0}^k \sum_{\delta=1}^{\infty} \frac{k!}{j!(k-j)!} \Big( \frac{\delta \langle v, v \rangle}{2y} \Big)^{k-j} e^{-\frac{\pi \delta^2 \langle v, v \rangle}{2y}} \vartheta_j\Big( \tau; z; \begin{pmatrix} \delta \\ 0 \end{pmatrix} \Big). \end{align*}
	
	For $k = 0$, we have $$\Theta_0(\tau, Z) = \mathbb{P}_{1-n/2, \rho_M}(\phi_0) +  \frac{\langle v, v \rangle}{2N v_e} \sum_{\beta \in \mathbb{Z}/N\mathbb{Z}} \mathfrak{e}_{(0, \beta / N)} \otimes \theta_K$$ with $\phi_0$ defined as above.
\end{thm}

\begin{proof}
This follows by applying Theorem 5.2 of \cite{Borcherds1998} twice, expressing the theta kernel for $M$ first in terms of theta functions attached to $L$ and finally in terms of theta functions attached to the definite lattice $K$.
\end{proof}

Now suppose $F = \sum_{t=0}^d F^{(t)} (2\pi y)^{-t}$ is an almost-holomorphic modular form with $$F^{(t)} = \sum_{a,b \in \mathbb{Z}/N} \sum_{\gamma \in K'/K} \sum_{n \in \mathbb{Q}} c^{(t)}(a/N, \gamma, b/N; n) q^n \mathfrak{e}_{\gamma}.$$  

Then $$\mathrm{Lift}(F) = \frac{i^{-k}}{2} \langle v, v \rangle^{-k} \int^{\mathrm{reg}}_{\mathrm{SL}_2(\mathbb{Z}) \backslash \mathbb{H}} [F, \Theta_k] \, \frac{\mathrm{d}x \, \mathrm{d}y}{y^2}.$$ The summands of the form $\mathbb{P}(\phi)$ in the representation of Theorem \ref{thm:theta_poincare} for $\Theta_k$ can be integrated using the (formal) identity $$\int_{\mathrm{SL}_2(\mathbb{Z}) \backslash \mathbb{H}}^{\mathrm{reg}} [F, \mathbb{P}(\phi)] \, \frac{\mathrm{d}x \, \mathrm{d}y}{y^2} = \int_{\Gamma_{\infty} \backslash \mathbb{H}}^{\mathrm{reg}} [F, \phi] \, \frac{\mathrm{d}x \, \mathrm{d}y}{y^2},$$ where $\Gamma_{\infty} = \{ \pm \begin{pmatrix} 1 & n \\ 0 & 1 \end{pmatrix}: \; n \in \mathbb{Z}\}$ is the stabilizer of $1$ in $\mathrm{SL}_2(\mathbb{Z})$ and $\Gamma_{\infty} \backslash \mathbb{H}$ is represented by the strip $-1/2 < x < 1/2, \; y > 0$. So we can write $$\mathrm{Lift}(F) = \Phi_0 + \Phi_1 + \Phi_2,$$ where $\Phi_j = \Phi_j(0)$ is defined as follows.

(i) $\Phi_0$ represents the integral against the part of $\Theta_k$ that is not a Poincar\'e series. So $\Phi_0 = 0$ unless $k = 0$, in which case
$$\Phi_0 = \frac{\langle v, v \rangle}{4Nv_e} \int^{\mathrm{reg}} \langle \sum_{\beta} F_{(0, *, \beta/N)}, \theta_K \rangle \, \frac{\mathrm{d}x \, \mathrm{d}y}{y^2}.$$ The regularized integral was computed in Theorem 9.2 of \cite{Borcherds1998} by writing the almost-holomorphic modular form for $\mathrm{SL}_2(\mathbb{Z})$ in the integral in the form $$\langle \sum_{\beta} F_{(0, *, \beta/N)}, \theta_K \rangle = \sum_{n=0}^d f_n(\tau) (E_2^*(\tau))^n$$ for some (uniquely determined) modular forms $f_n(\tau)$, and using the identity $$\int^{\mathrm{reg}} f_n(\tau) (E_2^*(\tau))^n \, \frac{\mathrm{d}x \, \mathrm{d}y}{y^2} = \text{constant term of} \; \frac{\pi}{3} f_n(\tau) \frac{E_2(\tau)^{n+1}}{n+1}.$$

(ii) $\Phi_1(s)$ represents the integral of $F$ against the definite theta functions $\theta_K$ attached to $K$:

\begin{align*} \Phi_1(s) &= \Big(-i \langle v, v \rangle\Big)^{-k} \Big( \frac{\langle v, v \rangle}{2Nv_e} \Big)^{k+1} \sum_{\beta \in \mathbb{Z}/N\mathbb{Z}} \sum_{\delta=1}^{\infty} \delta^k e^{-2\pi i \delta \beta / N} \\ &\quad\quad\quad\quad \times \int_0^{\infty} \int_{-1/2}^{1/2} F_\beta(\tau) \overline{\theta_K(\tau; (\delta/N, 0))} y^{-s}  e^{-\frac{\pi \delta^2}{2y (N v_e)^2}  \langle v, v \rangle} \frac{\mathrm{d}x \, \mathrm{d}y}{y^2} \\ &= \frac{\langle v, v \rangle}{2Nv_e} \cdot \Big(-\frac{i}{2N v_e} \Big)^k \sum_{t=0}^d (2\pi)^{-t} \sum_{\beta \in \mathbb{Z}/N\mathbb{Z}} \sum_{\lambda \in K'} c(0, \lambda+K, \beta/N; \lambda^2 / 2 ) \sum_{\delta = 1}^{\infty} \delta^k e^{- 2\pi i \frac{\delta \beta}{N} - 2\pi i \delta \frac{\langle \lambda, \mathfrak{v} \rangle}{N v_e}}  \\ &\quad \quad \times \int_0^{\infty} y^{-2-t-s}  e^{-\frac{\pi \delta^2}{2y (N v_e)^2} \langle v, v \rangle}  \, \mathrm{d}y \\ &= \frac{\langle v, v \rangle}{(2Nv_e)^{k+1}} (-i)^k \sum_{t=0}^d (2\pi)^{-t} \Gamma(1+s+t) \sum_{\beta \in \mathbb{Z}/N} \sum_{\delta=1}^{\infty} \delta^k e^{-2\pi i \frac{\delta \beta}{N} - 2\pi i \delta \frac{\langle \lambda, \mathfrak{v} \rangle}{N v_e}} \Big( \frac{2N^2 v_e^2}{\pi \delta^2 \langle v, v \rangle} \Big)^{1+s+t} \\ &\quad \quad \times \sum_{\lambda \in K'} c^{(t)}(0, \lambda, \beta/N; \lambda^2 / 2). \end{align*}

Here we use the identity $$\int_0^{\infty} y^{-2-s} e^{-C/y} \, \mathrm{d}y = C^{-1-s} \int_0^{\infty} u^s e^{-u} \, \mathrm{d}u = C^{-1-s} \Gamma(1+s)$$ for $C > 0$ and $\mathrm{Re}[s]$ sufficiently large.

The sum over $\delta$ (and $\beta$ and $\lambda$) can be calculated using the Hurwitz formula in the form \begin{align*} \sum_{n=1}^{\infty} \frac{e^{2\pi i n a}}{n^s} = \frac{\Gamma(1-s)}{(2\pi)^{1-s}} \Big( e^{\pi i (1-s)/2} \zeta(1-s; a) + e^{\pi i (s-1)/2} \zeta(1-s; 1-a) \Big), \end{align*} for $0 < a < 1$, where $\zeta(s; a) = \sum_{n=0}^{\infty} (n+a)^{-s}$ is the Hurwitz zeta function. So as long as $z$ stays in any open region for which $\frac{\langle \lambda, \mathfrak{v} \rangle}{v_e} \notin \mathbb{Z}$ for any $\lambda \in K' \backslash \{0\}$ with $c^{(t)}(0, \lambda, \beta/N; \lambda^2 / 2) \ne 0$, we have 
\begin{align*}
    &\quad \sum_{\beta, \lambda} \sum_{\delta=1}^{\infty} \delta^{k-2-2s-2t} e^{-2\pi i \delta \beta / N - 2\pi i \delta \frac{\langle \lambda, \mathfrak{v} \rangle}{N v_e}} c^{(t)}(0, \lambda, \beta/N; \lambda^2 / 2) \\ &= c^{(t)}(0, 0, 0; 0) \zeta(2+2t-k+2s) \\ &+(-1)^t (-i)^{k-1} \frac{\Gamma(k-1-2s-2t)}{(2\pi)^{k-1-2s-2t}} \sum_{(\beta, \lambda) \ne (0, 0)} c^{(t)}(0, \lambda, \beta/N; \lambda^2 / 2) \\&\quad \times \Big( e^{\pi i s} \zeta(k-1-2s-2t; \{ \beta/N + \langle \lambda, \mathfrak{v} \rangle / Nv_e \}) + (-1)^{k-1} e^{-\pi i s} \zeta(k-1-2s-2t; \{ -\beta/N - \langle \lambda, v \rangle / Nv_e \}) \Big),
\end{align*}
where if $a \in \mathbb{R}$ then we write $\{a\} \in (0, 1]$ for the number that makes $a - \{a\}$ an integer.

Now we consider the value at $s=0$ by cases. First suppose $k - 2t > 2$. In the above formula, the terms $\Gamma(k-1-2s-2t)$ and $\zeta(k-1-2s-2t; \{a\})$ remain holomorphic at $s=0$. Since $$c^{(t)}(0, \lambda, \beta/N; \lambda^2 / 2) = (-1)^k c^{(t)}(0, -\lambda, -\beta/N; \lambda^2 / 2),$$ the expressions $$e^{\pi i s} \zeta(k-1-2s-2t; \{ \beta/N + \langle \lambda, \mathfrak{v} \rangle / Nv_e \}) + (-1)^{k-1} e^{-\pi i s} \zeta(k-1-2s-2t; \{ -\beta/N - \langle \lambda, \mathfrak{v} \rangle / Nv_e \})$$ for the pairs $(\beta, \lambda)$ and $(-\beta, -\lambda)$ cancel themselves away at $s=0$. As for the remaining term, if $k$ is odd then $c^{(t)}(0, 0, 0;0) = 0$ (as is $c^{(t)}(0, 0, \beta/N; 0)$ if $\beta = N/2$ where $N$ is even), and if $k$ is even then $\zeta(2+2t-k)$ is a zeta value at a negative even integer and therefore $0$. Hence the entire sum is zero. \\

When $k - 2t = 2$ (and $k$ is even), the Hurwitz zeta function $\zeta(s;a)$ has a simple pole at $1$ with constant residue $1$. In this case we use the Laurent series $$\zeta(s; a) = \frac{1}{s - 1} + \psi(a) + O(s-1)$$ where $\psi(a) = \frac{\Gamma'(a)}{\Gamma(a)}$ is the digamma function, together with the reflection formula $$\psi(a) - \psi(1-a) = -\pi \cot(\pi a)$$ to obtain \begin{align*} &\quad \lim_{s \rightarrow 0} e^{\pi i s} \zeta(k-1-2s-t; \{\beta/N + \langle \lambda, \mathfrak{v} \rangle / Nv_e\}) - e^{-\pi i s} \zeta(k-1-2s-2t; \{ -\beta/N - \langle \lambda, \mathfrak{v} \rangle / Nv_e \}) \\ &= \lim_{s \rightarrow 0} \frac{e^{\pi i s} - e^{-\pi i s}}{-2s} + \psi(\{\beta/N + \langle \lambda, \mathfrak{v} \rangle / Nv_e\}) - \psi(1-\{\beta/N + \langle \lambda, \mathfrak{v} \rangle / Nv_e\}) \\ &= -\pi i -\pi \cot \Big( \pi \frac{\beta}{N} + \pi \frac{\langle \lambda, \mathfrak{v} \rangle}{N v_e} \Big). \end{align*} Hence if $k - 2t = 2$, \begin{align*}&\quad \lim_{s \rightarrow 0} \sum_{\beta, \lambda} \sum_{\delta=1}^{\infty} \delta^{k-2-2s-2t} e^{-2\pi i \delta \beta / N - 2\pi i \delta \frac{\langle \lambda, \mathfrak{v} \rangle}{N v_e}} c^{(t)}(0, \lambda, \beta/N; \lambda^2 / 2) \\ &= c^{(t)}(0, 0, 0; 0) \cdot \zeta(0) - \frac{1}{2} \sum_{(\beta, \lambda) \ne (0, 0)} c^{(t)}(0, \lambda, \beta/N; \lambda^2 / 2) + \frac{i}{2} \sum_{\beta, \lambda} c^{(t)}(0, \lambda, \beta/N; \lambda^2 / 2) \cot \Big( \pi \frac{\beta}{N} + \pi \frac{\langle \lambda, \mathfrak{v} \rangle}{N v_e} \Big) \\ &= -\frac{1}{2} \sum_{\beta, \lambda} c^{(t)}(0, \lambda, \beta/N; \lambda^2 / 2).
\end{align*}
The sum involving cotangent terms vanishes because corresponding to $(\beta, \lambda)$ and $(-\beta, -\lambda)$ cancel out. (N.B. Even if $N$ is even and $\lambda = 0$ and $\beta = N/2$, such that $(\beta, \lambda) = (-\beta, -\lambda)$, there is no additional contribution to the sum because $\cot(\pi/2) = 0$.)
\\

Finally consider the case $k - 2t < 2$. Then the series $$\sum_{\delta = 1}^{\infty} \delta^{k-2-2s-2t} e^{2\pi i a \delta}$$ converges at $s=0$ for any $a \notin \mathbb{Z}$, and using the Fourier series for the Bernoulli polynomials $$\sum_{\substack{\delta \in \mathbb{Z} \\ \delta \ne 0}} \delta^{-n} e^{2\pi i a \delta} = -(2\pi i)^n \frac{B_n(\{a\})}{n!}$$ together with $c^{(t)}(0, \lambda, \beta/N; \lambda^2 / 2) = (-1)^k c^{(t)}(0, -\lambda, -\beta/N; \lambda^2 / 2)$, we obtain \begin{align*} &\quad \sum_{\beta, \lambda} \sum_{\delta = 1}^{\infty} \delta^{k-2-2t} e^{-2\pi i \delta \beta / N - 2\pi i \delta \frac{\langle \lambda, \mathfrak{v} \rangle}{Nv_e}} c^{(t)}(0, \lambda, \beta/N; \lambda^2 / 2) \\ &= \zeta(2+2t-k) c^{(t)}(0, 0, 0; 0) -\frac{(2\pi i )^{2+2t-k}}{2 \cdot (2+2t-k)!} \sum_{(\beta, \lambda) \ne (0, 0)} c^{(t)}(0, \lambda, \beta/N; \lambda^2 / 2) B_{2+2t-k}\Big(\{\beta / N + \frac{\langle \lambda, \mathfrak{v} \rangle}{Nv_e}\}\Big).
\end{align*}
(When $2+2t-k = 1$ we have $c^{(t)}(0, 0, 0; 0) = 0$ because $k$ is odd, so the undefined term $\zeta(1)$ does not appear: this is a slight abuse of notation.) Since $\zeta(2+2t-k) = -\frac{(2\pi i)^{2+2t-k} B_{2+2t-k}}{2 \cdot (2+2t-k)!}$ for even $k$, and $c^{(t)}(0, 0, 0; 0) = 0$ for odd $k$, we can simply write \begin{align*} &\quad \sum_{\beta, \lambda} \sum_{\delta = 1}^{\infty} \delta^{k-2-2t} e^{-2\pi i \delta \beta / N - 2\pi i \delta \frac{\langle \lambda, \mathfrak{v} \rangle}{Nv_e}} c^{(t)}(0, \lambda, \beta/N; \lambda^2 / 2) \\ &= -\frac{(2\pi i )^{2+2t-k}}{2 \cdot (2+2t-k)!} \sum_{\beta, \lambda} c^{(t)}(0, \lambda, \beta/N; \lambda^2 / 2) B_{2+2t-k}\Big(\{\beta / N + \frac{\langle \lambda, \mathfrak{v} \rangle}{Nv_e}\}\Big).
\end{align*}

Note that this also describes the series in the case $k-2t=2$ because $B_0 = 1$. Altogether:

\begin{lemma} \begin{align*} \Phi_1(0) &= \frac{1}{2^{2k-1} N^{k-1} \pi^{k-1} v_e^{k-1}} (-1)^k \\ &\quad \times \sum_{t \ge k/2 - 1} v_e^{2t} \langle v, v \rangle^{-t} \frac{(-4)^t N^{2t} \cdot t!}{(2+2t-k)! } \cdot \sum_{\beta, \lambda} c^{(t)}(0, \lambda, \beta/N; \lambda^2 / 2) B_{2+2t-k} \Big( \{ \beta/N + \frac{\langle \lambda, \mathfrak{v} \rangle}{Nv_e} \} \Big).
\end{align*}
\end{lemma}

(iii) Finally, $\Phi_2(s)$ (the ``main part" of the theta integral) represents the integral of $F$ against the Lorentzian theta functions $\vartheta_j$ attached to $L$:

\begin{align*}
    \Phi_2(s) &= \frac{1}{\sqrt{2}} \langle v, v \rangle^{1/2 - k} \sum_{j=0}^k \sum_{\delta=1}^{\infty} \frac{k!}{j!(k-j)!} \Big( \frac{\langle v, v \rangle}{2Ny} \Big)^{k-j} e^{-\pi \delta^2 \frac{\delta \langle v, v \rangle}{2y}} \int_0^{\infty} \int_{-1/2}^{1/2} \Big[ F, \vartheta_j(\tau; z; (\delta, 0))] y^{-s} \, \frac{\mathrm{d}x\, \mathrm{d}y}{y^2} \\ &= \frac{1}{\sqrt{2}} \sum_{j=0}^k \frac{k!}{j! (k-j)! 2^{k-j}} \langle v, v \rangle^{1/2 - j} \sum_{h=0}^{\lfloor j/2 \rfloor} \frac{j!}{h! (j-2h)!} \Big( -\frac{\langle v, v \rangle}{8\pi} \Big)^h \\ &\quad \times \sum_{\mu \in L'} \sum_{t=0}^d (2\pi)^{-t} c^{(t)}(\mu, \mu^2 / 2) \langle \mu, v \rangle^{j-2h} \sum_{\delta=1}^{\infty} \delta^{k-j} e^{2\pi i \delta \langle \mu, u \rangle} \int_0^{\infty} y^{j-h-3/2-s-t} e^{-\frac{\pi \delta^2}{2y} \langle v, v \rangle - \frac{2\pi y}{\langle v, v \rangle} \langle \mu, v \rangle^2} \, \mathrm{d}y.
\end{align*}

First we compute the contribution of vectors $\mu \ne 0$ to this series. Using the integral representation of the Bessel $K$-function $$\int_0^{\infty} e^{-a/y - by} y^{s-1} \, \mathrm{d}y = 2(a/b)^{s/2} K_s(2 \sqrt{ab}), \quad a, b > 0$$ we have $$\int_0^{\infty} y^{j-h-3/2-s-t} e^{-\frac{\pi \delta^2}{2y} \langle v,v \rangle - \frac{2\pi y}{\langle v,v \rangle} \langle \mu, v \rangle^2} \, \mathrm{d}y = 2 \cdot \Big( \frac{\delta \langle v, v \rangle}{2|\langle \mu, v \rangle|} \Big)^{j-h-1/2-s-t} K_{j-h-1/2-s-t} \Big( 2\pi \delta |\langle \mu, v \rangle| \Big).$$
After inserting the series representation $$K_{j+1/2}(x) = K_{-j-1/2}(x) = \sqrt{\frac{\pi}{2x}} e^{-x} \sum_{n=0}^{\infty} \binom{j+n}{2n} \cdot \frac{(2n)!}{n!} (2x)^{-n}, \quad j \in \mathbb{N}_0, \; x > 0,$$ we can rewrite the value at $s=0$ in the summand attached to $\mu$ as follows:
\begin{align*} &\frac{1}{\sqrt{2}} \sum_{j=0}^k \frac{k!}{j! (k-j)! 2^{k-j}} \langle v, v \rangle^{1/2 - j} \sum_{h=0}^{\lfloor j/2 \rfloor} \frac{j!}{h! (j-2h)!} \Big( -\frac{\langle v, v \rangle}{8\pi} \Big)^{h} \\ &\quad \times \sum_{t=0}^d (2\pi )^{-t} c^{(t)}(\mu, \mu^2 / 2) \langle \mu, v \rangle^{j-2h} \sum_{\delta=1}^{\infty} \delta^{k-j} e^{2\pi i \delta \langle \mu, u \rangle} \cdot 2 \Big( \frac{\delta \langle v, v \rangle}{2|\langle \mu, v \rangle|} \Big)^{j-h-1/2-t} K_{j-h-1/2-t}\Big( 2\pi \delta |\langle \mu, v \rangle| \Big) \\ &=  \sum_{t=0}^d (2\pi)^{-t} c^{(t)}(\mu, \mu^2 / 2) \langle v, v\rangle^{-t} \sum_{\delta=1}^{\infty} \sum_{j=0}^k \sum_{h=0}^{\lfloor j/2 \rfloor} \sum_{n=0}^{\infty} \Big(-\frac{1}{8\pi}\Big)^h \frac{k! (2n)!}{(k-j)!h!(j-2h)! n!} \binom{j-h-t-1 + n}{2n} \\ &\quad \times 2^{h+t-k} (4\pi)^{-n} \delta^{k-h-1-t-n} \langle \mu, v \rangle^{j-2h} |\langle \mu, v \rangle|^{h+t-j-n} e^{2\pi i \delta \langle \mu, u \rangle - 2\pi \delta |\langle \mu, v \rangle|}
\end{align*}

We can rewrite $$\frac{k!}{(k-j)! (j-2h)!} = (2h)! \binom{k}{j} \binom{j}{2h}$$ to allow the sum over $h$ to run from $0$ to $\infty$. Then change variables to $r := h+n$ to rewrite the series over $h, n$ as \begin{align*} &\sum_{h=0}^{\infty} \sum_{n=0}^{\infty} \Big(-\frac{1}{8\pi} \Big)^h \frac{(2h)!}{h!} \binom{j}{2h} \frac{(2n)!}{n!} \binom{j-h-t-1+n}{2n} (4\pi)^{-n} \delta^{-h-n} 2^h \langle \mu, v \rangle^{-2h}|\langle \mu, v \rangle|^{h-n} \\ &= \sum_{r=0}^{\infty} (4\pi \delta |\langle \mu, v \rangle|)^{-r} \sum_{h+n=r} (-1)^n \frac{(2h)!(2n)!}{h!n!} \binom{j}{2h} \binom{t+r-j}{2n}.\end{align*}

The binomial sums $$\alpha(r,t,j) := \sum_{h+n=r} (-1)^h \frac{(2h)! (2n)!}{h!n!} \binom{j}{2h} \binom{t+r-j}{2n}$$ that appear in that series are well-defined for any $t,j, r \in \BZ$. If we sum over $j$ before summing over $r$ then the series simplifies significantly due to the following sums:

\begin{lem} \label{lemma:binomial identity} For any $k, r, t \in \mathbb{N}_0$, 
\begin{gather*}
\sum_{j=0}^k \binom{k}{j} \alpha(r, t, j) = 2^k \cdot r! \binom{t}{r} \binom{t + r - k}{r}; \\
\sum_{j=0}^k (-1)^j \binom{k}{j} \alpha(r, t, j) = 2^k \cdot r! \binom{t}{r} \binom{t+r-k}{r-k}.
\end{gather*}
In particular, the second sum is zero if $r < k$.
\end{lem}
Both parts of Lemma~\ref{lemma:binomial identity} can be derived by manipulating the following generating function:
\begin{lemma} For any $t,j,z$ (which can be taken to be polynomial variables), we have
\[ \sum_{r \geq 0} \frac{\alpha(r,t,j)}{r! \binom{t}{r}} z^r = \frac{(1-2z)^j}{(1-z)^{t+1}}. \]
\end{lemma}
In particular, $\alpha(r, t, j)$ vanishes if $r > t$, such that the quotient $\frac{\alpha(r, t, j)}{r! \binom{t}{r}}$ is well-defined.
\begin{proof}
For $j = 0, 1$, all terms in the defining sum for $\alpha(r, t, j)$ vanish except for $h=0$ and $n=r$. So $$\alpha(r, t, 0) = \frac{(2r)!}{r!} \binom{t + r}{2r}, \quad \alpha(r, t, 1) = \frac{(2r)!}{r!} \binom{t+r-1}{2r}$$ which yields  $$\frac{\alpha(r, t, 0)}{r! \binom{t}{r}} = \binom{t+r}{t}, \quad \frac{\alpha(r, t, 1)}{r! \binom{t}{r}} = \frac{(t-r)}{r} \binom{t+r-1}{t}.$$

In general, let $a$ be a polynomial variable and consider
\[ f_{a}(x) = \sum_{m \geq 0} \frac{(2m)!}{m!} \binom{a}{2m} x^m. \]
(These are the classical Hermite orthogonal polynomials up to a change of variables.) Then $f_a$ satisfies the recursion 
\begin{equation} \label{eq:poly_recursion} \forall a \in \BZ: \quad f_a(x) - f_{a-1}(x) - 2 (a-1) x f_{a-2}(x)=0. \end{equation}
By construction, $\alpha(r, t, j)$ is the coefficient of $x^r$ in $$f_{-j}(-x) \cdot f_{t+r-j}(x).$$
The recursion for $f_j(x)$ implies that $\alpha(r, t, j)$ satisfies
\[ \alpha(r,t,j) = \alpha(r,t-1,j-1) - 2 (j-1) \alpha(r-1,t-1,j-2). \]
If we let $\widetilde{\alpha}(r,t,j) := \frac{\alpha(r,t,j)}{r! \binom{t}{r}}$ (with the convention that this is zero if $r<0$) this is equivalent to
\begin{equation} \label{efsr3}
t \cdot \widetilde{\alpha}(r,t,j) = (t-r) \widetilde{\alpha}(r,t-1,j-1) - 2 (j-1) \widetilde{\alpha}(r-1,t-1,j-2)
\end{equation}

But the coefficients of $z^r$ in $\frac{(1 - 2z)^j}{(1 - z)^{t+1}}$ satisfy precisely the recursion (\ref{efsr3}). Since they coincide with $\widetilde{\alpha}(r, t, j)$ for $j \in \{0, 1\}$ it follows that they coincide for all $j$.

Indeed, set $\gamma(r,t,j) := \mathrm{Coeff}_{z^r}( \frac{(1-2z)^j}{(1-z)^{t+1}} )$ and
consider the generating series
\[ \Gamma(z,x,y) = \sum_{r \geq 0} \sum_{t,j \geq 0} \gamma(r,t,j) \frac{x^t}{t!} \frac{y^j}{j!} z^r. \]
The recursion for $\gamma(r,t,j)$ that we want to prove then translates to the differential equation
\begin{multline*}
x \frac{d}{dx} \Gamma(z,x,y) = \frac{x}{1-2z} \Gamma(z,x,y) - z \frac{d}{dz}\left( \frac{1-z}{1-2z} \Gamma(z,x,y) \right) - \left( y \frac{d}{dy} - 1 \right)\left[ 2 z \frac{1-z}{(1-2z)^2} \Gamma(z,x,y) \right].
\end{multline*}
The definition of $\gamma(r,t,j)$ gives us the expression
\[ \Gamma(z,x,y) = \sum_{r,t,j \geq 0} \mathrm{Coeff}_{z^r}\left( \frac{(1-2z)^j}{(1-z)^{t+1}} \right)
\frac{x^t}{t!} \frac{y^j}{j!} z^r = \frac{1}{1-z} e^{\frac{x}{1-z} + (1-2z)y}. \]
Thus the differential equation is immediately checked.
\end{proof}
\begin{rmk}
For $r \geq 0$ and $a,b$ general variables we set (with the notation of the above proof)
\[
\beta(r,a,b)=\sum_{m+n=r}(-1)^m \frac{(2m)!}{m!} \frac{(2n)!}{n!} \binom{a}{2m} \binom{b}{2n}
= \mathrm{Coeff}_{x^r}\left( f_a(-x) f_{b}(x) \right)
\]
With similar methods one can also prove the more symmetric relation
\[ \sum_{r \geq 0} \frac{\beta(r,a,b)}{r! \binom{a+b-r}{r}} z^r = (1-z)^a (1+z)^b \]
which relates the Hermite polynomials $f_a(x)$ with the classical Jacobi orthogonal polynomials (the $z^r$-coefficients of the right hand side).
\end{rmk}

Recall that

\begin{align*}
&\sum_{t=0}^d (2\pi)^{-t} c^{(t)}(\mu, \mu^2 / 2) \langle v, v\rangle^{-t} \sum_{\delta=1}^{\infty} \sum_{j=0}^k \sum_{h=0}^{\lfloor j/2 \rfloor} \sum_{n=0}^{\infty} \Big(-\frac{1}{8\pi}\Big)^h \frac{k! (2n)!}{(k-j)!h!(j-2h)! n!} \binom{j-h-t-1 + n}{2n} \\ &\quad \times 2^{h+t-j} (4\pi)^{-n} \delta^{k-h-1-t-n} \langle \mu, v \rangle^{j-2h} |\langle \mu, v \rangle|^{h+t-j-n} e^{2\pi i \delta \langle \mu, u \rangle - 2\pi \delta |\langle \mu, v \rangle|}
\\ &= 2^{-k}\sum_{t=0}^d \pi^{-t} c^{(t)}(\mu, \mu^2 / 2) \langle v, v \rangle^{-t} |\langle \mu, v \rangle|^t \\ &\quad \times\sum_{\delta=1}^{\infty} \sum_{j=0}^k \sum_{r=0}^{\infty} \binom{k}{j} (4\pi \delta |\langle \mu, v \rangle|)^{-r}  \delta^{k-t-1} \alpha(r, t, j) \mathrm{sgn}(\langle \mu, v \rangle)^j e^{2\pi i \delta \langle \mu, u \rangle - 2\pi \delta |\langle \mu, v \rangle|}.
\end{align*}

If $\langle \mu, v \rangle = |\langle \mu, v \rangle|$ then by the Lemma this simplifies to

$$\sum_{t=0}^d \pi^{-t} c^{(t)}(\mu, \mu^2 / 2) \langle v, v \rangle^{-t} \sum_{r=0}^{\infty} \langle \mu, v \rangle^{t-r} (4\pi)^{-r} r! \binom{t}{r} \binom{t+r-k}{r} \sum_{\delta=1}^{\infty} \delta^{k-t-r-1} e^{2\pi i \delta \langle \mu, z \rangle}$$
(note that the sum over $r$ terminates at $r=t$) while if $\langle \mu, v \rangle = -|\langle \mu, v \rangle|$ then it simplifies to $$\sum_{t=0}^d \pi^{-t} c^{(t)}(\mu, \mu^2 / 2) \langle v, v \rangle^{-t} \sum_{r=0}^{\infty} |\langle \mu, v \rangle|^{t-r} (4\pi)^{-r} r! \binom{t}{r} \binom{t+r-k}{r-k} \sum_{\delta=1}^{\infty} \delta^{k-t-r-1} e^{2\pi i \delta \langle \mu, \overline{z} \rangle}.$$

This accounts for the sum over $\mu \in L'$, $\mu \ne 0$ but we still have to consider the contribution of $\mu=0$. Instead of the $K$-Bessel function, we have the integral $$\int_0^{\infty} y^{j-h-3/2-s-t} e^{-\frac{\pi \delta^2}{2y} \langle v, v \rangle} \, \mathrm{d}y = \Big( \frac{\pi \delta^2 \langle v, v \rangle}{2} \Big)^{j-h-1/2-s-t} \Gamma(1/2+s+t+h-j).$$ All terms in the series over $h$ with $2h \ne j$ vanish due to the presence of $\langle \mu, v \rangle^{j-2h}$, so the $\mu=0$ contribution simplifies to
\begin{align*}
    &\frac{1}{\sqrt{2}} \sum_{h=0}^{\lfloor k/2 \rfloor} \frac{k!}{h!(k-2h)! 2^{k-2h}} \langle v, v \rangle^{1/2-2h} \Big(-\frac{\langle v, v \rangle}{8\pi}\Big)^h \\ &\quad \times \sum_{t=0}^d (2\pi)^{-t} c^{(t)}(0, 0) \sum_{\delta=1}^{\infty} \delta^{k-2h} \Big( \frac{\pi \delta^2 \langle v, v \rangle}{2} \Big)^{h-1/2-s-t} \Gamma(1/2 + s + t - 2h) \\ &= \frac{1}{2^{k-s} \sqrt{\pi}} \sum_{t=0}^d \pi^{-s-2t} \langle v, v \rangle^{-s-t} c^{(t)}(0, 0) \sum_{\delta=1}^{\infty} \delta^{k-1-2s-2t} \sum_{h=0}^{\lfloor k/2 \rfloor} \frac{k!}{(-4)^h h! (k-2h)!} \Gamma(1/2 + s + t - h)
\end{align*}


The value of the Gamma sum at $s=0$ is given in closed form by the following lemma:

\begin{lem} The identity $$\sum_{h=0}^{\lfloor k/2 \rfloor} \frac{k! (-4)^{-h}}{h! (k - 2h)!} \Gamma(1/2 + t - h) = \sqrt{\pi} \cdot t! \binom{2t-k}{t} 2^{k-1-2t} \cdot \begin{cases} 2: & k \le t; \\ 1: & k > t; \end{cases}$$ holds for any $k, t \in \mathbb{N}_0.$
\end{lem}
In particular, the sum is zero in the region $0 < t < k \le 2t$.
\begin{proof} 
Consider the generating functions $$f_{k, t}(X) = \sum_{h=0}^{\lfloor k/2 \rfloor} \frac{k! (-4)^{-h}}{h! (k-2h)!} \Gamma(1/2 + t - h) X^{h-t-1/2}.$$
Since for integers $n \geq 0$ one has $\Gamma(\frac{1}{2}-n) = \sqrt{\pi} (-4)^n n!/(2n)!$, at $t=0$ we have $$f_{k, 0} = \sqrt{\pi} \sum_{h=0}^{\lfloor k/2 \rfloor} \binom{k}{2h} X^{h-1/2} = \frac{\sqrt{\pi}}{2 \sqrt{X}} \Big( (1 + \sqrt{X})^k + (1 - \sqrt{X})^k \Big).$$ Moreover, the functional equation for $\Gamma$ implies $\frac{d}{dX} f_{k, t} = -f_{k, t+1}.$
By Taylor's theorem, we hence get 
\begin{align*} \sum_{t=0}^{\infty} \frac{Y^t}{t!} \sum_{h=0}^{\lfloor k/2 \rfloor} \frac{k! (-4)^{-h}}{h! (k-2h)!} \Gamma(1/2 + t - h) &= \sum_{t=0}^{\infty} \frac{f_{k, 0}^{(t)}(1)}{t!} (-Y)^t \\& = \frac{\sqrt{\pi}}{2 \sqrt{1 - Y}} \Big( (1 + \sqrt{1 - Y})^k + (1 - \sqrt{1-Y})^k \Big). \end{align*}
We find
\begin{align*}
\sum_{h=0}^{\lfloor k/2 \rfloor} \frac{k! (-4)^{-h}}{h! (k-2h)!} \Gamma(1/2 + t - h)
& = t! \sqrt{\pi} \oint_{\gamma_0} \frac{ (1 + \sqrt{1-Y})^k + (1 - \sqrt{(1-Y)})^k}{2 \sqrt{1-Y} \cdot Y^{t+1}} dY \\
& = -t! \sqrt{\pi} \oint_{\gamma_1} \frac{ (1 + x)^k + (1 - x)^k}{(1-x^2)^{t+1} } dx, \quad x=\sqrt{1-Y} \\
& = -t! \sqrt{\pi} \oint_{\gamma_1} \left[ \frac{(1+x)^{k-t-1}}{(1-x)^{t+1}} + \frac{ (1+x)^{-(t+1)}}{(1-x)^{t+1-k}} \right] dx \\
& = - t! \sqrt{\pi} \cdot 2^{k-2t-1} \left[ \binom{2t-k}{t} + \binom{2t-k}{t-k} \right]
\end{align*}
with the convention that $\binom{n}{k} = 0$ for $k<0$
and where $\gamma_0, \gamma_1$ are small loops around $0,1$ respectively and in the last step we used that for all $a \in \BZ$ and arbitrary $b$ we have
\[ \oint_{\gamma_1} \frac{(1+x)^{b}}{(1-x)^{a}} dx = - \oint_{\gamma_0} \frac{(2-x)^b}{x^a} dx = (-1)^a 2^{b-a+1} \binom{b}{a-1}. \qedhere \]
\end{proof}

The sum over $\delta$ is $\zeta(2s+2t+1-k)$ and it extends holomorphically to $s=0$ if $t \ne k/2$. Using the lemma we obtain 
\begin{align*}
&\lim_{s \rightarrow 0} \frac{1}{2^{k-s} \sqrt{\pi}} \sum_{2t \ne k} \pi^{-s-2t} \langle v, v \rangle^{-s-t} c^{(t)}(0, 0) \sum_{\delta=1}^{\infty} \delta^{k-1-2s-2t} \sum_{h=0}^{\lfloor k/2 \rfloor} \frac{k!}{(-4)^h h! (k-2h)!} \Gamma(1/2 + s + t - h) \\ &= \sum_{\substack{t \ge k \\ (t, k) \ne (0, 0)}} \frac{(2t-k)!}{(2\pi)^{2t} (t-k)!} \langle v, v \rangle^{-t} c^{(t)}(0, 0) \zeta(2t+1-k) \\ &\quad+ \frac{1}{2} \sum_{2t < k}  \frac{t!}{(2\pi)^{2t}} \binom{2t-k}{t} \langle v, v \rangle^{-t} c^{(t)}(0, 0) \zeta(2t+1-k).
\end{align*}

For $t$ such that $2t+1-k = 1$ (which can only occur if $k = 2t$ is even), $\zeta(2s+2t+1-k)$ has a simple pole at $s = 0$ with residue $1/2$. If $k > 0$ then the Gamma sum also vanishes, so the contribution of such terms is 
\begin{align*} &\quad \frac{1}{2^{k+1} \sqrt{\pi}} \pi^{-k} \langle v, v \rangle^{-k/2} c^{(k/2)}(0, 0) \sum_{h=0}^{k/2} \frac{k!}{(-4)^h h! (k-2h)!} \Gamma'(1/2 + k/2 - h) \\ &= \frac{1}{2^k \sqrt{\pi}} \pi^{-k} \langle v, v \rangle^{-k/2} c^{(k/2)}(0, 0) \sum_{h=0}^{k/2} \frac{k! \Gamma(1/2+k/2-h)}{(-4)^h h! (k-2h)!}\sum_{j=1}^{k/2 - h} \frac{1}{2j-1}.\end{align*}
Here we use the digamma value $$\frac{\Gamma'(1/2 + N)}{\Gamma(1/2 + N)} = \psi(1/2+N) = -\gamma - \ln(4) + 2 \sum_{j=1}^N \frac{1}{2j-1}$$ where $\gamma = -\Gamma'(1)$ is the Euler--Mascheroni constant, and the fact that the Gamma sum vanishes.


Otherwise, if $k = 0$ and $t = 0$, the Gamma sum does not vanish and the entire expression has a pole, and to carry out the regularized theta lift we have to take the constant term in the Laurent series about $s=0$. We have the Laurent series about $0$, $$\zeta(2s+1) = \frac{1}{2s} + \gamma + O(s)$$ and $$\Big( \frac{\pi \langle v, v \rangle}{2} \Big)^{-s} = 1 - \ln \Big( \frac{\pi \langle v, v \rangle}{2} \Big) s + O(s^2)$$ and $$\Gamma(1/2+s) = \sqrt{\pi} \cdot \Big( 1- (\gamma + \ln(4)) s + O(s^2) \Big).$$ Therefore the constant coefficient in the Laurent series of $$\frac{1}{2^{-s} \sqrt{\pi}} \pi^{-s} \langle v, v \rangle^{-s} c^{(0)}(0, 0) \zeta(2s+1) \Gamma(1/2+s)$$ about $s=0$ is $$\frac{1}{2} c^{(0)}(0, 0) \cdot \Big( 2\gamma - \ln(2\pi) - \ln (\langle v, v \rangle) \Big).$$

\section{Fourier-Jacobi expansion of quasimodular forms}
\label{sec:FJ expansion and quasimodular forms}
In this section we consider the Fourier-Jacobi expansions 
of orthogonal quasimodular forms.
Many parts of this section are inspired by \cite{CG} which discusses the Fourier-Jacobi expansion of usual orthogonal modular forms
(and which is based on earlier work of Gritsenko).
We also refer to \cite[Appendix A]{Enriques} for a discussion of quasi-Jacobi forms and to \cite{Ma} for the Fourier-Jacobi expansions of vector-valued orthogonal modular forms.

\subsection{Jacobi forms}
Let $N$ be an even integral positive semidefinite lattice of rank $l$ with a fixed basis $b_1,\ldots, b_{l}$ and inner product denoted by $(x,y)_N$ for $x,y \in N$.
The real Heisenberg group is
\[
H(N_{\BR}) = \{ [x,y; \kappa] | x,y \in N \otimes \BR, \kappa \in \BR \}
\]
with group law
\[
[x,y;\kappa] \cdot [x',y';\kappa'] = [x+x', y+y'; \kappa+\kappa'+ \frac{1}{2} ( (x,y')_N - (x',y)_N) ].
\]
The group $\SL_2(\BR)$ acts from the left on $H(N_{\BR})$ by
\[
A \cdot [x,y;\kappa] = [(x,y) \cdot A^{-1}; \kappa] := [ dx - cy, -bx + ay ; \kappa ].
\]
The real Jacobi group is the associated semi-direct product
\[ \Gamma^J(N,\BR) = \SL_2(\BR) \ltimes H(N_{\BR}). \]
The multiplication is $(A,\zeta) \cdot (A',\zeta') = (A A', (A')^{-1} \zeta \cdot \zeta')$. Concretely,
\[ (A,[x,y;\kappa]) \cdot (A', [x',y'; \kappa'])
=
(A A', [(\tilde{x}, \tilde{y}) + (x',y'); \kappa + \kappa' + \frac{1}{2} ( (\tilde{x}, y')_N- (\tilde{y}, x')_N) ]) \]
where $(\tilde{x}, \tilde{y}) = (x,y) A'$.

Let $\Fz = (\Fz_1, \ldots, \Fz_l) = \sum_i \Fz_i b_i \in \BC^l \cong N_{\BC}$,
let $m \geq 0$ be a rational number and $k \in \BZ$.
\begin{defn}
For any $A = \binom{a\ b}{c\ d} \in \SL_2(\BR)$, $\zeta=[(x, y),\kappa] \in H(N_{\BR})$ and smooth function $$F = F(\Fz,\tau) : \BC^l \times \BH \to \BC,$$ the Jacobi slash operator is defined by
\begin{multline*}
	(F|_{k,m} (A, \zeta))(\Fz,\tau) :=
    F\left( \frac{\Fz + x \tau + y}{c \tau + d}, \frac{a \tau + b}{c \tau + d} \right)
    (c \tau+d)^{-k} \\
       \times e\left( m \left( - \frac{1}{2} \frac{ c (\Fz+x \tau + y, \Fz+x \tau + y)_N }{c \tau + d} + \frac{1}{2} (x,x)_N \tau + (x, \Fz)_N + \frac{1}{2} (x,y)_N + \kappa \right) \right)
\end{multline*}
We abbreviate $F|_{k,m} (A, \zeta)$ by $F|_{k,m} A$ if $\zeta=0$, and by $F|_{m} \zeta$ if $A=\id$.
\end{defn}

Note that $F|_{k,m} (A,\zeta) = F|_{k,m} A |_{m} \zeta$.

As explained in  \cite[Lem.1.2]{Ziegler},
the Jacobi group 
$\Gamma^J(N,\BR)$ acts on
the space of smooth functions on $\BC^l \times \BH$
via the slash operator.
The real analytic functions $y=\mathrm{Im}(\tau)$ and
\[ 
\alpha = (\alpha_1,\ldots,\alpha_l), \quad \alpha_i(x,\tau) = \frac{\mathrm{Im}(\Fz_i)}{y}
\]
have the transformation properties
\begin{gather}
\alpha\left( \frac{\Fz + X \tau + Y}{c \tau+d}, \frac{a \tau+b}{c \tau+d} \right)  	= 	( c \tau + d) \cdot \alpha(\Fz,\tau) + X (c \tau + d) - c (\Fz+X \tau + Y), \label{tr laws for alpha} \\
	y^{-1}\left( \frac{a \tau+b}{c \tau+d} \right)
	= (c \tau+d)^2 y^{-1}(\tau) - 2ic (c \tau+d) \label{tr law for 1/y}
    \end{gather}
for all $\gamma = \binom{a\ b}{c\ d} \in \mathrm{SL}_2^{+}(\BR)$
and $x,y \in N_{\BR}$.

A function
$F : \BC^l \times \BH \to \BC$ 
is  \emph{almost holomorphic} if it is of the form
\[ F(\Fz, \tau) = \sum_{i \geq 0} \sum_{j = (j_1, \ldots, j_l) \in (\BZ_{\geq 0})^l}
F_{i, j}(\Fz,\tau) \alpha^j y^{-i}, \quad \quad \alpha^j = \alpha_1^{j_1} \cdots \alpha_m^{j_m} \]
such that each of the finitely many non-zero $F_{i, j}(z,\tau)$ is holomorphic
on $\BC^l \times \BH$.
The transformation laws for $\alpha$ and $y^{-1}$ above show that the slash operator
preserves (almost) holomorphic functions.

The integral Heisenberg group is 
\[ H(N) = \{ [x,y;\kappa] | x,y \in N, \kappa \in \frac{1}{2} \BZ \text{ with } \kappa+\frac{1}{2}(x,y)_N \in \BZ \}.\] Let $\Gamma \subset \Gamma^J(N,\BZ) := \SL_2(\BZ) \ltimes H(N)$ be a finite index subgroup.

\begin{defn} \label{defn:quasi jacobi forms main text}
(i) A weakly \emph{almost holomorphic Jacobi form} of weight $k$ and index $m$
for $\Gamma$ is an almost-holomorphic function $F(x,\tau) : \BC^l \times \BH \to \BC$ satisfying
$F|_{k,m} \gamma = F$ for all $\gamma \in \Gamma$. \\
(ii) A weakly quasi-Jacobi form of weight $k$ and index $m$ 
for the group $\Gamma$ is the coefficient of $\alpha^0 \beta^0$
of a weakly almost-holomorphic Jacobi form of the same kind.\\
(iii) The vector spaces of weakly almost-holomorphic Jacobi forms and
quasi-Jacobi forms of weight $k$ and index $m$ for $\Gamma$
will be denoted by $\AHJ_{k,m}(\Gamma), \QJac_{k,m}(\Gamma)$ respectively.
\end{defn}

\begin{rmk}
Quasi-Jacobi forms are defined by additionally requiring a cusp condition, see \cite[Appendix A]{Enriques}.
We do not need this condition here, so we omit this discussion.
\end{rmk}

Taking the constant term of an almost-holomorphic Jacobi form
defines an isomorphism \cite{RES}:
\begin{equation} \mathrm{ct} : \AHJ_{k,m}(\Gamma) \xrightarrow{\cong}
\QJac_{k,m}(\Gamma), \quad 
\sum_{i,j} \phi_{i,j} \alpha^i \nu^j \mapsto \phi_{0,0}. \label{constant term map} \end{equation}

For $x = (x_1, \ldots, x_l) \in \BZ^m$ we have the lowering operators
\[
{\LJ} 
= -2 i y 
\left( y \frac{\partial}{\partial \overline{\tau}} + \sum_{i=1}^{l} \mathrm{Im}(\Fz_i) \frac{\partial}{\partial \overline{\Fz}_i} \right),
\quad
\xi_{x} = -2 i y \sum_i x_i \frac{\partial}{\partial \overline{\Fz}_i}.
\]
In particular, $\LJ(y^k) = k y^{k+1}$, $\LJ(\alpha_i)=0$ and $\xi_{x}(\alpha_i) = x_i$, $\xi_x(y)=0$.
The $\LJ$ and $\xi_x$ are well-known to intertwine the Jacobi slash operator and hence act on the space of almost-holomorphic Jacobi forms, lowering the weight by $2$ and $1$ respectively, while preserving the index, see \cite{BRR}
or \cite{RES}\footnote{In the notation of \cite[Sec.1.3.3]{RES} we have $\frac{d}{d \nu} =-8 \pi L^{\Jac}$ and $\frac{d}{d \alpha_i} = \xi_{e_i} = -2 i y \frac{\partial}{\partial \overline{\Fz}_i}$.}. 

Since $\ct$ is an isomorphism, we also define
\begin{gather*}
\LJ := \ct \circ \LJ \circ \ct^{-1} : \QJac_{k,m}(\Gamma) \to \QJac_{k-2,m}(\Gamma) \\
\xi_x := \ct \circ \xi_x \circ \ct^{-1} : \QJac_{k,m}(\Gamma) \to \QJac_{k-1,m}(\Gamma).
\end{gather*}

\subsection{Fourier-Jacobi expansion of an almost-holomorphic modular form}
We return to our fixed lattice $M$ of signature $(2,n)$.
We assume that $M$ has the direct sum decomposition
\[ M = U \oplus L(r), \quad L = U_1 \oplus N(-1) \]
where $U,U_1$ are two copies of the hyperbolic lattice and $N$ is a positive definite lattice. 
Let $e,f \in U$ and $e_1,f_1 \in U(1)$ be standard bases with intersection form $\binom{0\ 1}{1\ 0}$.
We write $(-,-)_M$ for the intersection pairing on $M$, and similarly for $L,N$.
For example, $(e_1, f_1)_M = r$ and $(e_1,f_1)_L = 1$.
Write $z = (\omega, \Fz, \tau)$ for the element $\omega e_1 + \Fz + \tau f_1$, as well as $\tilde{q} = e(\omega)$ and $q = e(\tau)$.
We also choose a basis $b_1,\ldots, b_{n-2}$ of $N$, and deine the coordinates $\Fz_i$ by $\Fz=\sum_i \Fz_i b_i$.
The tube domain is 
\[ \CC = \{ (\omega, \Fz, \tau) \in \BH \times N_{\BC} \times \BH\, | \, 2 \Im(\tau) \Im(\omega) - (\Fz, \Fz)_N > 0 \}. \]
The map $\widetilde{\varphi} : \CC \to \A(\CD)$ to the affine cone is given by
\[ \widetilde{\varphi}(\omega,\Fz, \tau) = 
\left( r (-\omega \tau + (\Fz,\Fz)_N/2), \omega, \Fz, \tau, 1\right). \]

\begin{lemma}
The real Jacobi group $\Gamma^J(N,\BR)$ is embedded as a subgroup of $O(M_{\BR})$ by
\[
A \mapsto \{ A \}_r  := \begin{pmatrix}
	A^{(r)} & 0 & 0 \\
	0 & I_{n_0} & 0 \\
	0 & 0 & A
\end{pmatrix},
\]
for all $A = \begin{pmatrix} a & b \\ c & d \end{pmatrix} \in \SL_2(\BR)$,
where
$A^{(r)} = \begin{pmatrix} a & -rb \\ -c/r & d \end{pmatrix}$,
and
\[
[x,y;\kappa] \mapsto
\begin{pmatrix}
1 & & r y^{t} S & \frac{r}{2} (x,y)_N - r \kappa & \frac{r}{2} (y,y)_N \\
& 1 & x^t S & \frac{1}{2} (x,x)_N & \frac{1}{2} (x,y)_N + \kappa \\
& & 1 & x & y \\
& & & 1 & \\
& & & & 1
\end{pmatrix}
\]
for all $[x,y;\kappa] \in H(N)$,
where the matrix is with respect to the basis $e,e_1,b_1,\ldots,b_{n-2},f_1,f$) and 
$S$ is the Gram matrix of $N$.
\end{lemma}
\begin{proof}
This can be found for $r=1$ in \cite[Sec.1]{CG}
and the general case is similar. We remark that the embedding of the Heisenberg group is equivalent to
the assignment
\[ [x,0;0] \mapsto \exp\left(\frac{x \wedge_M e_1}{r}\right),\quad [0,y;0] \mapsto \exp\left(y \wedge_M e\right), \quad [0,0;\kappa] \mapsto \exp\left(\kappa e_1 \wedge_M e\right). \]
where $v \wedge_M w$ stands for the Eichler transvection with respect to the lattice $M$.
\end{proof}

Let $\Gamma \subset O^{+}(M)$ be a finite-index subgroup and let $\tilde{q} = e^{2 \pi i \omega}$.
We assume that
\[ \exp(e_1 \wedge e) \in \Gamma. \]
which implies\footnote{This follows since $\exp(e_1 \wedge e) \cdot (\omega, \Fz, \tau) = (\omega+1, \Fz, \tau)$ and $J(\exp(e_1 \wedge_M e))=1$ and $\lambda(\exp(e_1 \wedge_M e))=\id$.}
that an almost-holomorphic modular form $F \in \AHMod_{k,s}(\Gamma)$ has Fourier expansion
\[ F(\omega, \Fz, \tau, \overline{\omega}, \overline{\Fz}, \overline{\tau}) = \sum_{m = 0}^{\infty} \tilde{q}^m F_m(\Fz, \tau, \overline{\omega}, \overline{\Fz}, \overline{\tau}). \]
For clarity we emphasize here the dependence on the non-holomorphic variables. In particular, note that even though $F_m$ does not depend on $\omega$ it still depends on $\overline{\omega}$.

Our goal is now to relate the Fourier coefficients $F_m$ to almost-holomorphic Jacobi forms.
To do this, we need to modify $F_m$ in two ways: we need to restrict its dependence on the non-holomorphic variables and  to turn the tensor-valued function $F_m$ into a scalar-valued function.

For the first step, consider the differential form
\[ \nu = \nu_{\omega} d \omega + \nu_{\tau} d \tau + \sum_i \nu_{\Fz_i} d \Fz_i = \partial \log( \Im(\mathrm{z})^2)  =
\frac{\mathrm{Im}(\tau) d \omega + \mathrm{Im}(\omega) d \tau - ( d \Fz, \mathrm{Im}(\Fz))_N}{
i ( 2 \mathrm{Im}(\omega) \mathrm{Im}(\tau) - (\Im(\Fz), \Im(\Fz))_N )} \]
where we used
\begin{gather*}
\mathrm{Im}(z)^2 = 2 r \mathrm{Im}(\omega) \mathrm{Im}(\tau) - r (\Im(\Fz), \Im(\Fz))_N \\
\partial( \mathrm{Im}(z)^2 ) = \frac{r}{i} ( \mathrm{Im}(\tau) d \omega + \mathrm{Im}(\omega) d \tau - ( d \Fz, \mathrm{Im}(\Fz))_N ).
\end{gather*}
If $f$ is an almost-holomorphic function, it is a polynomial in the $\nu_{\omega}, \nu_{\Fz_i}, \nu_{\tau}$ with holomorphic functions as coefficients.
We set
\begin{equation} \label{Flimit} f^{\mathrm{lim}} := f|_{\nu_{\omega}=\nu_{\Fz}=0, \nu_{\tau} = 1/2 i \mathrm{Im}(\tau)}.\end{equation}
The notation contains the superscript $\mathrm{lim}$ since these specializations may be seen as taking the limit of $\nu_{\omega}, \nu_{\Fz}, \nu_{\tau}$ as $\mathrm{Im}(\omega) \to \infty$.

Secondly, recall from Section~\ref{subsec:multidiff as multilinear maps} that an almost-holomorphic modular form $F$ of rank $s$ can be viewed as a linear map $F : \CC \to (L_{\BC}^{\vee})^{\otimes s}$. Essentially, the multidifferential can be evaluated on tangent vectors in $\CC$, which then can be identified with elements of $L_{\BC}$.
Hence we consider a fixed element
\[ v = v_1 \otimes \ldots \otimes v_s \in L_{\BR}^{\otimes s}. \]
To obtain the weight of the Jacobi form, we assume that for all $j$ we have
\[ v_j \in \{ e_1, f_1 \} \cup N_{\BR}. \]
We then define a weight shift by
\[ \wt(v) := |\{ j : v_j = f_1 \} | - | \{ j : v_j = e_1 \} |. \]
We will evaluate $F$ on the vector
\[
v[\Fz,\tau] :=
\exp\left( \frac{\mathrm{Im}(\Fz)}{\mathrm{Im}(\tau)} \wedge_L e_1 \right) v
\ \in L_{\BC}^{\otimes s}.
\]
where we use the convention $g(v_1 \otimes \cdots \otimes v_s) = gv_1 \otimes \cdots \otimes gv_s$ for $g \in O(L_{\BC})$ and the Eichler transvection as defined in Section~\ref{subsec:Eichler transvection}.
With this we can state the main result of this section:
\begin{thm} \label{thm:FJ of almost-holomorphic}
	Let $F \in \AHMod_{k,s}(\Gamma)$ be an almost-holomorphic modular form for $\Gamma$ and let $$F = \sum_{m} F_m \tilde{q}^m$$ be its Fourier expansion in the variable $\omega$. Then each $F_m^{\mathrm{lim}}(v[\Fz,\tau])$ is a weakly almost-holomorphic Jacobi form for $\Gamma \cap \Gamma^J(N,\BR)$ of weight $k+\wt(v)$ and index $m$.

In particular, if $\Gamma = O^{+}(M)$, then
$F_m^{\mathrm{lim}}(v[\Fz,\tau])$ is a weakly almost-holomorphic Jacobi form for the group $\Gamma_0(r) \ltimes H(N)$.
\end{thm}

\begin{rmk}
The theorem could also be formulated as saying that $F_m^{\mathrm{lim}}$ defines a vector-valued almost-holomorphic Jacobi form (depending non-holomorphically only on $1/y$). Here we trade the vector-valued property with a scalar-valued modularity by paying the price of introducing the extra non-holomorphic dependence on $\Im(\Fz)$. The scalar-valued case is what we will use in applications in Gromov-Witten theory. \qed
\end{rmk}

The rest of this section will be the proof.
First we determine the automorphy factors and the action on the tube domain of elements of $\Gamma^J(N,\BR)$:
\begin{lemma} \label{lemma:automorphy factors for Gamma J}
For any $A = \begin{pmatrix} a & b \\ c & d \end{pmatrix} \in \SL_2(\BR)$ and $[x,y;\kappa] \in H(N_{\BR})$ and $z=(\omega,\Fz,\tau)$ we have
\[ (A, [x,y;\kappa]) \cdot z = \left(\omega - \frac{1}{2} \frac{c( \Fz, \Fz)_N}{c \tau + d} + \frac{1}{2} (x,x)_N \tau + (x,\Fz)_N + \frac{1}{2} (x,y)_N + \kappa, \frac{\Fz + x \tau + y}{c \tau + d}, \frac{a \tau + b}{c \tau + d} \right). \]
Moreover, $J((A, [x,y;\kappa]),z) = (c \tau + d)$ and
\[
\lambda( (A,[x,y;\kappa]), z)
=
\begin{pmatrix} c \tau + d & & \\ & I_{n-2} & \\ & & (c \tau + d)^{-1} \end{pmatrix} \exp\left( \frac{ - c}{c \tau + d} \left( \Fz + x \tau + y \right) \wedge_L e_1 \right) \exp(x \wedge_L e_1).
\]
\end{lemma}
\begin{proof}
	For any $x,y \in N_{\BR}$ and $A \in \SL_2(\BR)$ we have
	\begin{enumerate}
    \item[(i)] $\exp( \kappa e_1 \wedge_M e) \cdot (\omega,\Fz,\tau) = (\omega+\kappa, \Fz, \tau)$, $J(\exp(\kappa e_1 \wedge_M e))=1$, $\lambda(\exp( \kappa e_1 \wedge_M e)) = \id$,
		\item[(ii)] $\exp(y \wedge_M e) \cdot (\omega, \Fz, \tau) = (\omega, \Fz + y, \tau)$,
		$J(\exp(y \wedge_M e)) = 1$,
		$\lambda(\exp(y \wedge_M e)) = \id$.
		\item[(iii)] $\exp\left(\frac{1}{r} x \wedge_M e_1\right) \cdot (\omega, \Fz, \tau) =
		(\omega + (\Fz,x)_N + \frac{1}{2} (x,x)_N \tau, \Fz + \tau x, \tau)$,
		$J(\exp\left(\frac{1}{r} x \wedge_M e_1\right)) = 1$,\\
		$\lambda(\exp\left(\frac{1}{r} x \wedge_M e_1\right)) = \exp(x \wedge_L e_1)$.
		\item[(iv)] 
        \[ \{ A \}_r \cdot (\omega, \Fz, \tau) = \left( \omega - \frac{1}{2} \frac{c (\Fz, \Fz)_N}{c \tau + d}, \frac{\Fz}{c \tau + d}, \frac{a \tau + b}{c \tau + d} \right), \]
        and
        \begin{align*}
       		J(\{A\}_r) & = c \tau + d \\
		\lambda( \{ A \}_r, z ) & = 
		\exp\left( - c \Fz \wedge_L e_1 \right)
		\begin{pmatrix} c \tau + d & & \\ & I_{n-2} & \\ & & (c \tau + d)^{-1} \end{pmatrix} \\
& = \begin{pmatrix} c \tau + d & & \\ & I_{n-2} & \\ & & (c \tau + d)^{-1} \end{pmatrix} \exp\left( \frac{ - c \Fz \wedge_L e_1 }{c \tau + d} \right).
\end{align*}
\end{enumerate}
Part (i) is Example~\ref{example:Automorphy factors} (a). The rest are simple computations.
Since
\[ (A, [x,y;\kappa]) = A \cdot [x,0;0] \cdot [0,y;0] \cdot [0,0,\kappa-\frac{1}{2} (x,y)_N] \]
the claim follows by composition and the automorphy properties of $J$ and $\lambda$.
\end{proof}

Consider the row vector $(\nu_{\omega}, (\nu_{\Fz_i})_{i=1}^{n-2},\nu_{\tau})$, identified with $$\nu=\nu_{\omega} d \omega + \sum_i \nu_{\Fz_i} d \Fz_i + \nu_{\tau} d\tau = \partial \log( \Im(\mathrm{z})^2).$$
\begin{lemma} \label{lemma:nu transform for imega tau}
For any $A \in \SL_2(\BR)$ and $[x,y;\kappa] \in H(N_{\BR})$ and $z=(\omega,\Fz,\tau)$ we have
\begin{multline*}
\nu( (A, [x,y;\kappa]) \cdot z)
=
(\nu_{\omega}, (c \tau + d) \nu_{\Fz_i} + ( c (y+\Fz) - d \cdot x, b_i)_N \nu_{\omega}, \\
(c \tau +d)^2 \nu_{\tau} - c (c \tau + d)  - (c \tau + d)^2 \sum_i x_i \nu_{\Fz_i} + c (c \tau + d) \sum_i (\Fz_i + x_i \tau + y_i) \nu_{\Fz_i} \\
+ \big( \frac{1}{2} (x,x)_N (c \tau + d)^2 - c (c \tau + d) (x, \Fz + x \tau + y)_N \nu_{\omega}
+ \frac{1}{2} c^2 (\Fz + x \tau + y, \Fz + x \tau + y)_N \big) \nu_{\omega}.
\end{multline*}
\end{lemma}
\begin{proof}
By Lemma~\ref{lemma:nu transformation} we have
\[ \nu(\gamma \cdot z) = J(\gamma,z) \nu(z) \lambda(\gamma,z)^{-1} - \partial(J(\gamma,z)) \lambda(\gamma,z)^{-1}. \]
The claim follows from Lemma~\ref{lemma:automorphy factors for Gamma J} and using that for $x=(x_1,\ldots,x_{n-2}) \in N$ we have
\[
\exp(x \wedge_L e_1)^{-1} = \exp(- x \wedge_L e_1 ) =
\begin{pmatrix}
1 & -(x,b_1)_N & \cdots & -(x,b_{n-2})_N & \frac{1}{2} (x,x)_N \\
& 1 & & & -x_1 \\
& & \ddots & & \vdots \\
& & & 1 & -x_{n-2} \\
& & & & 1
\end{pmatrix}
\]
\end{proof}

We are ready to start the proof of Theorem~\ref{thm:FJ of almost-holomorphic}.

\begin{proof}[Proof of Theorem~\ref{thm:FJ of almost-holomorphic}]
For $A = \binom{a \ b}{c\ d} \in \SL_2(\BR)$ and $[X,Y;\kappa] \in H(N_{\BR})$ we have
\begin{small}
\begin{equation} \label{dfswt34t3322}
\begin{aligned}
& v\left[ \frac{\Fz + x \tau + y}{c \tau + d}, \frac{a \tau + b}{c \tau + d} \right] \\ 
& = \exp\left( \Big( (c \tau + d) \frac{\Im(\Fz)}{\Im(\tau)} + x (c \tau + d) - c (\Fz + x \tau + y) \Big) \wedge_L e_1 \right) v \\
& = 
\begin{pmatrix} c \tau + d & & \\ & I_{n-2} & \\ & & (c \tau + d)^{-1} \end{pmatrix}
\exp\left( \Big( \frac{\Im(\Fz)}{\Im(\tau)} + x - \frac{c(\Fz + x \tau + y)}{c \tau + d} \Big) \wedge_L e_1 \right) 
\begin{pmatrix} (c \tau + d)^{-1} & & \\ & I_{n-2} & \\ & & (c \tau + d) \end{pmatrix}
v \\
& = \lambda( (A,[x,y;\kappa]), z) (c \tau + d)^{\wt(v)} v[\Fz, \tau].
\end{aligned}
\end{equation}
\end{small}
where we used \eqref{tr laws for alpha} in the first line.

Let $F \in \AHMod_{k,s}(\Gamma)$ be an almost-holomorphic modular form. We view $F$ as a polynomial in $\nu_{\omega}, \nu_{\Fz_i}, \nu_{\tau}$ and apply the Fourier expansion to its holomorphic coefficients. Hence we may write
\[ F = \sum_{m} \tilde{q}^{m} F_m( \Fz_i, \tau, \nu_{\omega}, \nu_{\Fz_i}, \nu_{\tau} ). \]
For every $\gamma=(A, [X,Y;\kappa]) \in \Gamma \cap \Gamma^J(N,\BR)$ we then have
\begin{align*}
& F = F|_{k,s} \gamma = \left( \sum_m F_m \tilde{q}^m \right)|_{k,s} \gamma \\
& = \sum_{m} e\left( 
m \left( \omega - \frac{1}{2} \frac{c( \Fz, \Fz)_N}{c \tau + d} + \frac{1}{2} (x,x)_N \tau + (x,\Fz)_N + \frac{1}{2} (x,y)_N + \kappa \right) \right)
(c \tau + d)^{-k} \\
& \quad \quad \quad \cdot F_m \left( \frac{\Fz + x \tau + y}{c \tau + d}, \frac{a \tau + b}{c \tau + d},
\nu(\gamma \cdot z) \right) \lambda(\gamma,z)^{\otimes s}.
\end{align*}
We now set $\nu_{\omega} = \nu_{\Fz_i}=0$ and $\nu_{\tau} = 1/2i \Im(\tau)$ and
take the coefficient of $\tilde{q}^m$. This yields:
\begin{align*}
F_m^{\mathrm{lim}}
& =
e\left( 
m \left( - \frac{1}{2} \frac{c( \Fz, \Fz)_N}{c \tau + d} + \frac{1}{2} (x,x)_N \tau + (x,\Fz)_N + \frac{1}{2} (x,y)_N + \kappa \right) \right)
(c \tau + d)^{-k} \\
& \quad \quad \cdot F_m \left( \frac{\Fz + x \tau + y}{c \tau + d}, \frac{a \tau + b}{c \tau + d},
0,0,(c \tau + d) \frac{1}{2i \Im(\tau)} - c (c \tau + d) \right) \lambda(\gamma,z)^{\otimes s}.
\end{align*}
where we have used Lemma~\ref{lemma:nu transform for imega tau}.
We now apply this to $v[\Fz, \tau]$ and use \eqref{dfswt34t3322} to get
\begin{align*}
& F_m^{\mathrm{lim}}(v[\Fz, \tau])\\
& =
e\left( 
m \left( - \frac{1}{2} \frac{c( \Fz, \Fz)_N}{c \tau + d} + \frac{1}{2} (x,x)_N \tau + (x,\Fz)_N + \frac{1}{2} (x,y)_N + \kappa \right) \right)
(c \tau + d)^{-k - \wt(v)} \\
& \quad \quad \cdot F_m \left( \frac{\Fz + x \tau + y}{c \tau + d}, \frac{a \tau + b}{c \tau + d},
0,0,(c \tau + d) \frac{1}{2i \Im(\tau)} - c (c \tau + d) \right)\left(
v\left[ \frac{\Fz + x \tau + y}{c \tau + d}, \frac{a \tau + b}{c \tau + d} \right] \right) \\
& = F_m^{\mathrm{lim}}(v[\Fz, \tau]) |_{k,m} (A, [x,y;\kappa]).
\end{align*}
In other words, the function $f=F_m^{\mathrm{lim}}(v[\Fz, \tau])$ 
satisfies $f|_{k+\wt(v),m} \gamma = f$ for all $\gamma \in \Gamma \cap \Gamma^J(N,\BR)$ which was what we wanted to show.

The last claim follows because $O^{+}(M) \cap \Gamma^J(N,\BR)$
contains
$\Gamma_0(r) \ltimes H(N)$.
\end{proof}
\begin{rmk}
The key in the above proof is that the transformation of $\nu$ considered in Lemma~\ref{lemma:nu transform for imega tau} matches the transformation of $1/2 i \Im(\tau)$ given in \eqref{tr law for 1/y} after setting $\nu_{\omega} = \nu_{\Fz_i}=0$ and $\nu_{\tau} = 1/2 i \Im(\tau)$.
\end{rmk}
We also discuss the interaction of the lowering operators.
\begin{prop} \label{prop:lowering in FJ expansion}
\begin{align}
\LJ( F_m^{\mathrm{lim}}(v[\Fz,\tau]) )
&=
\frac{i}{r} \cdot \LO(F)_m^{\mathrm{\lim}}( v[\Fz,\tau] \otimes e_1 )  \\
\xi_\mu(F_m^{\mathrm{lim}}(v[\Fz,\tau])) &= F_m^{\mathrm{lim}}\Big(((m \wedge_L e_1)v)[\Fz,\tau]\Big) 
\end{align}
\end{prop}
\begin{proof}
Let $F = \sum_I F_I dz_I$ be almost-holomorphic viewed as a polynomial in $\nu_{\omega}, \nu_{\Fz_i}, \nu_{\tau}$ with coefficients holomorphic functions. By \eqref{L on nu} we have
\[ \LO(F) = \sum_{I} \frac{r}{2} \frac{\partial F_I}{\partial \nu_{\tau}} dz_I d\omega + \ldots \]
where $\ldots$ stands for terms involving $d \tau$ and $d\Fz_i$. Thus for any $\rho \in L_{\BC}^{\otimes s}$ we get
\[ \LO(F)(\rho \otimes e_1) = \sum_{I} \frac{r}{2} \frac{\partial F_I(\rho)}{\partial \nu_{\tau}} dz_I, \]
where we note that $e_1 = \frac{\partial}{\partial \omega}$ since $z=\omega e_1 + \Fz + \tau f_1$.

For an almost-holomorphic function $f : \CC \to \BC$ we have
\[ \LJ( f^{\mathrm{lim}} ) = \frac{i}{2} \left( \frac{\partial f}{\partial \nu_{\tau}} \right)^{\mathrm{lim}} \]
(to check this, we may assume that $f = \nu_{\tau}^{k}$.)

Hence we get
\begin{align*}
\LJ( F_m^{\mathrm{lim}}(v[\Fz,\tau]) )
& = \LJ( F_m^{\mathrm{lim}} )(v[\Fz, \tau]) \quad (\text{since } \LJ( \mathrm{Im}(\Fz)/\Im(\tau))=0)\\
& = \frac{i}{2} \left(\frac{\partial F_m}{\partial \nu_{\tau}}\right)^{\mathrm{lim}}(v[\Fz,\tau]) \\
& = \frac{i}{2} \left(\frac{\partial F}{\partial \nu_{\tau}}\right)_m^{\mathrm{lim}}(v[\Fz,\tau]) \\
& = \frac{i}{r} \left( \LO(F)(- \otimes e_1)\right)_m^{\mathrm{lim}}(v[\Fz,\tau]) \\
& = \frac{i}{r} \LO(F)_m^{\mathrm{lim}}(v[\Fz,\tau] \otimes e_1).
\end{align*}

As for the second claim, we have
\[ \xi_x \exp\left( \frac{\Im(\Fz)}{\Im(\tau)} \wedge_L e_1 \right) v
= \exp\left( \frac{\Im(\Fz)}{\Im(\tau)} \wedge_L e_1 \right) (x \wedge_L e_1)v,
\]
which implies the claim because $F_m^{\mathrm{lim}}$ does not depend on $\overline{\Fz}$.
\end{proof}

Now we can describe the Fourier--Jacobi expansion of a quasimodular form.
Let $f \in \QMod_{k,s}(\Gamma)$ be a quasimodular form, viewed again as a function $f : \CC \to (L_{\BC}^{\vee})^{s}$. Let $f = \sum_{m} f_m \tilde{q}^m$ be its Fourier expansion. 
Consider a fixed element
\[ v = v_1 \otimes \ldots \otimes v_s \in L_{\BR}^{\otimes s}. \]

\begin{cor} \label{thm:FJ expansion of quasimodular forms}
The function $f_{m}(v) : \BH \times N_{\BC} \to \BC$ is a weakly quasi-Jacobi form for $\Gamma \cap \Gamma^J(N,\BR)$ of weight
	$k+\wt(v)$ and index $m$. It satisfies
	\begin{gather*}
    		\LJ f_{m}(v) = \frac{i}{r} \cdot \LO(f)_m(v \otimes e_1) \\
		\xi_{\mu} f_{m}(v) = f_{m}((\mu \wedge e_1)v).
	\end{gather*}
\end{cor}
\begin{proof}
If $F$ is the unique almost-holomorphic modular form such that $\ct(F)=f$, then
$F_m^{\mathrm{lim}}(v[\Fz,\tau])$ is the unique almost-holomorphic Jacobi form such that
$\ct(F_m^{\mathrm{lim}}(v[\Fz,\tau]))$. The claim follows.
\end{proof}

\begin{example}
Consider the lattice $M=U\oplus U_1(r)$ as in Section~\ref{subsec:example U+U}.
To compare with the notation there, we write $(\omega,\tau) = \omega e_1 + \tau f_1 = z_1 e_1 + z_2 f_1 = (z_1, z_2)$, and we have $dz_1 = d \omega = e_1^{\ast}$, $dz_2=f_1^{\ast}$, where $e_1^{\ast}, f_1^{\ast}$ stand for the dual vectors of $e_1, f_1$. Moreover,
\[ \nu = \frac{1}{2i}\left( \frac{dz_1}{y_1} + \frac{dz_2}{y_2} \right) \]
so $\nu_{\omega} = 1/2i \Im(\omega)$, $\nu_{\tau} = 1/2i \Im(\tau)$.
Hence $f^{\lim}$ is obtained by formally setting $\Im(\omega)^{-1} = 0$.

If $f : \CC \to \BC$ is an almost-holomorphic modular form then by Section~\ref{subsec:example U+U} we have
\[
\LO(f) = \frac{r}{i}\left( \LM_{z_2}(f) dz_1 + \LM_{z_1}(f) dz_2  \right)
= \frac{1}{i} \left( \LM_{\tau}(f) e_1^{\ast} + \LM_{\omega}(f) f_1^{\ast}  \right),
\]
Pairing with $e_1$ singles out the first term, so we get
\[ \LO(f)(e_1) = \frac{r}{i} \LM_{\tau}(f). \]
Taking the $m$-th Fourier coefficient and setting $\Im(\omega)^{-1}=0$ then recovers Proposition~\ref{prop:lowering in FJ expansion}.
\end{example}

\part{Gromov-Witten theory}

\section{Basic definitions}
\label{sec:basic definitions}
Let $X$ be a smooth projective variety and let $\beta \in H_2(X,\BZ)$ be a curve class modulo torsion.\footnote{In this paper, cohomology and homology will always be taken modulo torsion and we denote by $H_k(X,\BZ),H^k(X,\BZ)$ homology and cohomology modulo torsion.}
Let $\Mbar_{g,n}(X,\beta)$ be the moduli space of stable maps $f :  C \to X$ from connected nodal curves of genus $g$ of degree $f_{\ast}[c]=\beta$ with $n$ markings $p_1,\ldots,p_n \in C$ which are disjoint from the nodes.
The moduli space $\Mbar_{g,n}(X,\beta)$ is a proper Deligne-Mumford stack and carries a virtual fundamental class 
\[ [ \Mbar_{g,n}(X,\beta) ]^{\vir} \in H_{2\mathrm{vd}_{\beta}}(\Mbar_{g,n}(X,\beta)) \]
which is of virtual dimension
\[ \mathrm{vd}_{\beta} = \int_{\beta} c_1(X) + (\dim(X)-3)(1-g)+n. \]

The Gromov--Witten invariants of $X$ are defined by integrating natural cohomology classes against the virtual class.
Let $L_i$ be the cotangent line bundle on $\Mbar_{g,n}(X,\beta)$ with fiber $T_{p_i}^{\ast} C$ over the point $(f,p_1,\ldots,p_n)$.
Let $\psi_i=c_1(L_i)$ be its first Chern class.
Further, let
\[ \ev_i : \Mbar_{g,n}(X,\beta) \to X, \quad \ev_i((f,p_1,\ldots,p_n)) = f(p_i) \]
be the natural evaluation maps at the markings.
For $k_i \geq 0$ and cohomology classes $\gamma_i \in H^{\ast}(X)$, the Gromov--Witten invariants of $X$ are defined by the integral
\begin{equation} \label{GW invariants}
\left\langle \tau_{k_1}(\gamma_1) \cdots \tau_{k_n}(\gamma_n) \right\rangle^{X}_{g,\beta}
=
\int_{[ \Mbar_{g,n}(X,\beta) ]^{\vir}} \prod_{i=1}^{n} \ev_i^{\ast}(\gamma_i) \psi_i^{k_i}.
\end{equation}

A more refined definition can be obtained by considering Gromov--Witten classes. If $2g-2+n > 0$, the moduli space of stable maps $\Mbar_{g,n}$ is well-defined. Let
\[ \tau : \Mbar_{g,n}(X,\beta) \to \Mbar_{g,n} \]
be the forgetful map. We define the Gromov--Witten classes by
\[
\CC_{g,\beta}^X(\gamma_1, \ldots,\gamma_n)
=
\pi_{\ast}\left( 
\prod_{i=1}^{n} \ev_i^{\ast}(\gamma_i) \cap
 [ \Mbar_{g,n}(X,\beta) ]^{\vir} \right) \quad \in H_{\ast}(\Mbar_{g,n}).
\]
The Gromov--Witten invariants of $X$ can be recovered from its Gromov--Witten classes by integration against the tautological classes of the moduli space of stable curves.

\section{Enriques surfaces}
\label{sec:Enriques surfaces}

In this section, we state our main conjecture that the Gromov--Witten series of Enriques surfaces are quasimodular forms for the orthogonal group. We also state a conjecture for the implied modularity of the Gromov--Witten potentials of the Enriques Calabi--Yau threefold in~\ref{conj:Enriques CY3}.

\subsection{Background}
An Enriques surface $Y$ is a smooth projective surface satisfying \[ \omega_Y^{\otimes 2} \cong \CO_Y, \quad \omega_Y \neq \CO_Y,
\quad H^1(Y,\CO_Y)=0. \]
Equivalently, $Y$ is the quotient of a K3 surface $X$ by a fixed-point free involution. We let $p : X \to Y$ be the quotient map and let $G = \BZ_2$ be the covering group.
The second cohomology of an Enriques surface is given by
\[ H^2(Y,\BZ) \cong U \oplus E_8(-1) \]
and it embeds isomorphically into the fixed part of the K3 cohomology
\[ p^{\ast} : H^2(Y,\BZ) \xrightarrow{\cong} H^2(X,\BZ)^G \cong U(2) \oplus E_8(-2). \]

The orthogonal complement of this embedding is isomorphic to the {\em Enriques lattice}
\[ (H^2(X,\BZ)^G)^{\perp} \cong M, \quad M := U \oplus U(2) \oplus E_8(-2). \]
The complement of the Heegner divisor $\CH_{-2}$ in the period domain $\CD_M$ of the lattice $M$ is then the moduli space of marked Enriques surfaces.
The moduli space of Enriques surfaces is the quotient
by the discriminant kernel of $M$.
For us the relevant group is
\[ \Gamma := O^{+}(M) \]
which is the image of the group of $G$-equivariant monodromies of the K3 surface $X$ in the lattice $M$, see \cite[Prop.3.3]{Enriques}.
The Borcherds--Enriques modular form $\Phi_4$ is a modular form on $M$ of weight $4$ for the group $\Gamma$ with a character of order two \cite{BorcherdsEnriques}. The form $\Phi_4$ vanishes to first order along $\CH_{-2}$.
Hence the assumptions of Section~\ref{subsec:log ahol modular forms}
are satisfied, and we can and will consider below logarithmic  orthogonal quasimodular forms $f \in \QMod^{\log}_{k,s}(\Gamma)$ with respect to $\Gamma$ for the pair $(\CD_{M},\CH_{-2})$.

Consider the splitting
\[ M = U \oplus L(2), \quad L=U \oplus E_8(-1). \]
As explained in Section~\ref{subsec:Fourier expansion of quasimodular forms}, every quasimodular form $f \in \QMod^{\log}_{k,s}(\Gamma)$ has a Fourier expansion
\[ f = \sum_{\beta \in L} c_{f}(\beta) Q^{\beta} \]
where $Q^{\beta} = \exp(2 \pi i (z,\beta)_L)$, $\beta$ runs over positive classes (i.e. lies in the connected component of the closure of the positive cone corresponding to the tube domain), and $c_f(\beta) \in (L^{\vee})^{\otimes s}$.
We will identify such quasimodular forms with their associated Fourier expansion
\[ f = \sum_{\beta \in L} c_{f}(\beta) Q^{\beta} \in (L^{\vee})^{\otimes s} \otimes \BC[ L ]^{\wedge} \]
where $\BC[ L ]^{\wedge}$ is the group ring of $L$ completed along
the cone of positive classes.

For example, the Borcherds-Enriques form is given by the Borcherds product
\[
\Phi_4
=
\prod_{\beta > 0} \left( \frac{1- Q^{\beta} }{1+ Q^{\beta} } \right)^{8 a\left( (\beta,\beta)_L /2 \right)} \]
where the coefficients $a(n)$ are defined by
	\[
	\sum_{n \geq 0} a(n) q^n = \prod_{n \geq 1} \frac{(1+q^n)^8}{(1-q^n)^8} 
	=
	\frac{\eta(2 \tau)^8}{\eta(\tau)^{16}} = 
	1 + 16q + 144q^{2} + 960q^{3} + 5264q^{4} + \ldots .
	\]

\begin{rmk} \label{rmk:autoequivalences Enriques}
We briefly remark here that the lattice $M$ and the group $O^+(M)$ also appear naturally in the geometry of Enriques surface in another context.
Recall that the Mukai lattice of $X$ is the group $H^{\ast}(X,\BZ)$ endowed with the pairing $(r,D,n) \cdot (r',D',n') = D D'-rn'-r'n$
and it is given by
\[ H^{\ast}(X,\BZ) \cong U^{\oplus 4} \oplus E_8(-2). \]
The full invariant lattice is
		\[ H^{\ast}(X,\BZ)^{G} = U \oplus U(2) \oplus E_8(-2) \cong M \]
and the anti-invariant lattice is
		\[ (H^{\ast}(X,\BZ)^{G})^{\perp} = U \oplus U(2) \oplus E_8(-2) \cong M. \]
The derived monodromy group of the Enriques surface  $\DMon(Y)$ is generated by the cohomology action of $G$-equivariant auto-equivalences and monodromy operators on the Mukai lattice, see \cite{Enriques}.
As shown in \cite{Enriques}, the restriction of the derived monodromy group of $Y$ onto the invariant part is then precisely
\[
\mathrm{DMon}(Y)|_{H^{\ast}(X,\BZ)^G} = O^{+}( H^{\ast}(X,\BZ)^G ) \cong O^{+}(M) \]
and its restriction onto the anti-invariant part is
what we have seen previously:
\[
\mathrm{DMon}(Y)|_{(H^{\ast}(X,\BZ)^G)^{\perp}} =
\mathrm{\Mon}(Y)|_{(H^2(X,\BZ)^G)^{\perp}} \cong O^{+}(M). \]
The dual relationship
\[ (H^{\ast}(X,\BZ)^{G}, \mathrm{DMon}(Y)|_{H^{\ast}(X,\BZ)^G}) \leftrightarrow
(H^{\ast}(X,\BZ)^G)^{\perp}, \mathrm{DMon}(Y)|_{(H^{\ast}(X,\BZ)^G)^{\perp}}) \]
obtained by interchanging the invariant and anti-invariant parts reflects the fact that the Enriques surface is self-mirror in the sense of mirror symmetry for lattice-polarized K3 surfaces \cite{Dolgachev}.
\end{rmk}

\begin{rmk}
Orthogonal modular forms for $(M,O^{+}(M))$ have been studied before from a slightly different viewpoint.
One has that 
\[ M^{\vee}(2) \cong U(2) \oplus U \oplus E_8(-1) \cong U \oplus U \oplus D_{8}(-1). \]
Moreover, $O^{+}(M) = O^+(U^2 \oplus D_8(-1))$.
It follows that $\Mod_{k,s}(\Gamma) \cong \Mod_{k,s}( O^{+}(U^2 \oplus D_8(-1)))$.
In the scalar-valued case (where $s=0$) these spaces are well-understood and the ring $\oplus_k \Mod_{k,0}( O^{+}(U^2 \oplus D_8(-1)))$
was shown by Vinberg \cite{Vinberg2018} to be a a free polynomial algebra in
$11$ generators; see also entry 12 of \cite[Table 1]{WW}.
However, very little is known about the vector-valued case.
\end{rmk}

\subsection{Gromov-Witten theory of Enriques surfaces}
Recall from Section~\ref{sec:basic definitions}
that for $2g-2+n>0$ 
we have the Gromov-Witten classes
\[ \CC^Y_{g,\beta}(\gamma_1,\ldots,\gamma_n)=
\tau_{\ast}\left(  \prod_i \ev_i^{\ast}(\gamma_i) \cap [\Mbar_{g,n}(Y,\beta)]^{\vir} \right) \in H_{\ast}(\Mbar_{g,n}) \]
where $\tau : \Mbar_{g,n}(Y,\beta) \to \Mbar_{g,n}$
is the forgetful morphism.

We consider the formal series of Gromov-Witten classes
\[
\CC_{g}^Y(\gamma_1, \ldots, \gamma_n)
= 
\sum_{\beta \geq 0} \CC_{g,\beta}(\gamma_1, \ldots, \gamma_n) Q^{\beta}
\]
where $\beta$ runs over all effective (or zero) classes in $H^2(Y,\BZ)$.
The variable $Q^{\beta}$ can be viewed as the canonical basis element of the (completed) group ring $\BC[ H^2(Y,\BZ) ]^{\vee}$ of the lattice $H^2(Y,\BZ)$, where the completion is taken with respect to the cone of effective curve classes.

The series $\CC_g(-,\ldots,-)$ then may be viewed as an element in
\[
(H^{\ast}(Y,\BC)^{\otimes n})^{\vee} \otimes H_{\ast}(\Mbar_{g,n}) \otimes \BC[ H^2(Y,\BZ) ]^{\wedge}.
\]
For any decomposition $n=a+s+b$ we write
\[ \CC_g|_{(H^0)^{\otimes a} \otimes (H^2)^{\otimes s} \otimes (H^4)^{\otimes b}} \]
for the restriction of $\CC_g$ onto the subspace
\[ H^0(Y,\BQ)^{\otimes a} \otimes H^2(Y,\BQ)^{\otimes s} \otimes H^4(Y,\BQ)^{\otimes b} \subset H^{\ast}(Y,\BZ)^{\otimes n}. \]

\subsection{Modularity conjecture} \label{subsec:modularity conjecture}
We are now ready to state the main conjecture.
Fix throughout a lattice isomorphism
\[ L \xrightarrow{\cong} H^2(Y,\BZ) \]
where $L=U \oplus E_8(-1)$ and let $M=U \oplus L(2)$ as before.

\begin{conj} \label{conj:main} The cycle $\CC_g$ is a 
cycle-valued logarithmic quasimodular form for the pair $(\CD_M,\CH_{-2})$ defined by the Enriques lattice $M$ with respect to the group $\Gamma=O^{+}(M)$.

More precisely, for any $n=a+b+s$ we have:
	\[ \CC_g|_{(H^0)^{\otimes a} \otimes (H^4)^{\otimes b} \otimes (H^2)^{\otimes s}}
	\in H^{\ast}(\Mbar_{g,n}) \otimes \QMod^{\log}_{2g-2+n-a+b,s}(\Gamma). \]
\end{conj}
\vspace{8pt}

The conjecture implies that each $\CC_g(\gamma_1,\ldots, \gamma_n)$ is a component of a logarithmic vector-valued orthogonal quasimodular form for group $\Gamma$.
If there are no insertions from $H^2(Y)$, we obtain a scalar-valued quasimodular form (with values in cohomology classes on $\Mbar_{g,n}$).
Integrating against a class $\taut \in H^{\ast}(\Mbar_{g,n})$ then yields:
\[ F_{g}( \taut ; 1^a, \pt^b )  \in \QMod^{\log}_{2g-2+2b}(\Gamma). \]
where we let
\[
	F_g(\taut ; \tau_{k_1}(\gamma_1), \ldots, \tau_{k_n}(\gamma_n)) = \sum_{\beta \in H^2(Y,\BZ)_{>=0}} \int_{\Mbar_{g,n}} \taut \cdot \prod_{i=1}^{n} \psi_i^{k_i} \cap \CC_{g}(\gamma_1,\ldots,\gamma_n).
	\]

\begin{rmk}
There is an interesting numerical coincidence here.
The complex cohomological degree of 
$\CC_g|_{(H^0)^{\otimes a} \otimes (H^4)^{\otimes b} \otimes (H^2)^{\otimes s}}$
in $H^{\ast}(\Mbar_{g,n})$ is
precisely its weight as an orthogonal quasimodular form.
	Indeed, the class $\CC_{g}(\gamma_1, \ldots, \gamma_n)$ has 
	complex degree $2g-2+\sum_i \deg_{\BC}(\gamma_i)$ in $\Mbar_{g,n}$.
\end{rmk}

\subsection{Raising and lowering operators}
We consider how the raising and lowering operators interact with the Gromov-Witten classes $\CC_g(\gamma_1,\ldots,\gamma_n)$.

Recall that the raising operator $\RO : \QMod_{k,s}(\Gamma) \to \QMod_{k+1,s+1}(\Gamma)$ on quasimodular forms
can be identified with a differential operator $\partial$. Concretely, 
if $e_i$ is a fixed basis of $L$, then for any $\mu = \sum_i \mu_i e_i \in L$ we have
\[ \RO_{\mu} = \sum_{i} \mu_i \frac{d}{d z_i}. \]
By the divisor equation in Gromov-Witten theory we find:

\begin{lemma} \label{lemma:Raising operator on geometry} For any $\mu \in L \cong H^2(Y,\BZ)$ we have
	\[ \frac{1}{2 \pi i } \RO_{\mu} \CC_g(\gamma_1, \ldots, \gamma_n) =
	p_{\ast} \CC_g(\gamma_1, \ldots, \gamma_n, \mu)
	\]
	where $p : \Mbar_{g,n+1} \to \Mbar_{g,n}$ is the forgetful morphism.
\end{lemma}

The lowering operator is more interesting:
For $\lambda \in H^2(Y,\BZ)$ let
\[ U_{\lambda} = \lambda \otimes 1 + 1 \otimes \lambda \in H^{\ast}(Y \times Y). \]
Let $U_{\lambda} = \sum_i \delta_i \otimes \delta_i^{\vee}$ denote its K\"unneth decomposition
and let 
\[ U_{\lambda}(\gamma) = \sum_i \left( \int_{Y} \gamma \cdot \delta_i\right) \delta_i^{\vee} \]
be the action of $U_{\lambda}$ as a correspondence 
on a class $\gamma \in H^{\ast}(Y)$.

We define the class
\begin{align*} 
	\LO_{\lambda}^{\GW} \CC_g(\gamma_1, \ldots, \gamma_n)
	& :=
	\iota_{\ast} \CC_{g-1}(\gamma_1, \ldots, \gamma_n, U_{\lambda}) \\
	& + \sum_{\substack{g = g_1 + g_2,i \\ \{ 1, \ldots, n \} = A \sqcup B }} 
	j_{\ast}( \CC_{g_1}(\gamma_A, \delta_i ) \boxtimes \CC_{g_2}(\gamma_B, \delta_i^{\vee} ) )  \\
	& - 2 \psi_i \CC_{g}( \gamma_1, \ldots, U_{\lambda}(\gamma_i), \ldots, \gamma_n).
\end{align*}
where 
$\iota : \Mbar_{g-1,n+2} \to \Mbar_{g,n}$ and $j : \Mbar_{g_1, |A|+1} \times \Mbar_{g_2, |B|+1}$ are the natural gluing morphisms.

The following is our second main conjecture
which expresses the result of applying the lowering operator applied to $\CC_g$ in terms of the geometric class $\LO_{\lambda}^{\GW}$.

\begin{conj}[Holomorphic Anomaly Equation] \label{conj:HAE}
Assume that Conjecture~\ref{conj:main} holds.
	For every $\lambda \in L \cong H^2(Y,\BZ)$ we have
	\[ -4 \pi i (\LO \CC_g)(\gamma_1, \ldots, \gamma_n, \lambda) 
	=
	\LO_{\lambda}^{\GW} \CC_g(\gamma_1, \ldots, \gamma_n). \]
\end{conj}
\vspace{8pt}

In the remainder of this section, we will collect some evidence for the two conjetures above.
We check the compatibility with the commutation relations,
prove the conjectures for the example of Hodge integrals,
and consider Fourier-Jacobi expansions.

\subsection{Checking the commutation relation}
In Lemma~\ref{lemma:commutation relation}
we have seen the commutation relation:
\begin{equation*} \label{commutation relnX}
	\big( [\LO_{\lambda}, \RO_{\mu}] F \big)(\gamma_1, \ldots, \gamma_s)
	=
	\frac{k}{2} (\lambda, \mu)_M F(\gamma_1, \ldots, \gamma_s)
	+ \frac{1}{2} \sum_{m=1}^{s} F( \ldots, (\mu \wedge_M \lambda)(\gamma_i), \ldots )
\end{equation*}
where $F$ is an almost-holomorphic modular form of weight $k$ and rank $s$ expanded in the tube domain and $\lambda, \mu \in L$.
Since the pairing of $M$ differs here from the one on $L$ by a factor of $2$, the above may be rewritten as:
\begin{equation} \label{commutation reln}
	\big( [\LO_{\lambda}, \RO_{\mu}] F \big)(\gamma_1, \ldots, \gamma_s)
	=
	k (\lambda, \mu) F(\gamma_1, \ldots, \gamma_s)
	+ \sum_{m=1}^{s} F( \ldots, (\mu \wedge_L \lambda)(\gamma_i), \ldots )
\end{equation}

We have geometric expressions for $\LO_{\lambda}, \RO_{\mu}$ acting on $\CC_g$.
In the following lemma we test that these are compatible with this commutation relation:

\begin{lemma}
	Conjecture~\ref{conj:HAE} is compatible with the commutation relations \eqref{commutation reln}.
\end{lemma}
\begin{proof}
The same argument as in \cite[Proof of Corollary 1]{HAE}
gives that
\begin{multline*} 
-\frac{1}{2} [\LO_{\lambda}^{\GW}, \frac{1}{2 \pi i} \RO_{\mu}] \CC_g(\gamma_1, \ldots, \gamma_n) \\
		=
		(2g-2+n) (\lambda, \mu)_L \CC_g(\gamma_1, \ldots, \gamma_n)
		- \sum_{i=1}^{n} \CC_g(\gamma_1, \ldots, \gamma_{i-1}, [ U_{\lambda},e_{\mu} ] \gamma_i, \gamma_{i+1}, \ldots, \gamma_n )  \\
=
		\left(2g-2 + \sum_{i} \deg_{\BC}\gamma_i \right) (\lambda, \mu)_L \CC_g(\gamma_1, \ldots, \gamma_n)
		+ \sum_{i=1}^{n} \CC_g(\gamma_1, \ldots, (\mu \wedge_L \lambda) \gamma_i, 
		\ldots, \gamma_n )     
        \end{multline*}
	where $e_{\mu}(\gamma) = \mu \cup \gamma$. 
    This implies the claim.
\end{proof}

\subsection{Example: Hodge integrals}
The most fundamental series of Gromov-Witten invariants of the Enriques surface is the series
\[ F_g := F_g( (-1)^{g-1} \lambda_{g-1} ) = \int_{\Mbar{g}} \CC_{g}() \cdot (-1)^{g-1} \lambda_{g-1} \]
where $\lambda_i = c_i(\BE) \in H^{2i}(\Mbar_{g})$ is the Chern class of the Hodge bundle. 

These were recently evaluated in \cite{Enriques} as follows:
\begin{thm}[{\cite[Theorem 1]{Enriques}}] For any $g \geq 1$ we have
\[ F_g = 2 (-1)^{g-1}\sum_{\beta > 0} Q^{\beta} \sum_{\substack{\textup{odd } k | \beta}} k^{2g-3} \omega_g\left( \frac{\beta^2}{2 k^2} \right) \]
where the coefficients $\omega_g(n)$ are defined by
\begin{align*}
	\sum_{g \geq 0} \sum_{n \geq 0} \omega_g(n) z^{2g-2} q^n = 
	\prod_{\substack{m \geq 1\\ m \text{ odd}}} \frac{ 1 }{ (1-e^z q^m)^2 (1- e^{-z} q^m)^2 (1-q^m)^{4} }
	\prod_{m \geq 1} \frac{1}{(1-q^m)^8}.
\end{align*}
\end{thm}

We have the following check:
\begin{prop} \label{prop:Hodge integral} For $g \geq 2$ we have
\[	F_g \in \QMod^{\log}_{2g-2}(\Gamma). \]
Moreover, the holomorphic anomaly equation (Conjecture~\ref{conj:HAE}) holds.
\end{prop}

To give the proof of the proposition we need some preparation.

For $g \geq 1$ define the quasimodular forms
\begin{align*}
f_g(\tau) & := \left[ \frac{\Theta(z, 2 \tau)^2}{\Theta(z,\tau)^2} \right]_{z^{2g-2}} \cdot \frac{\eta(2 \tau)^8}{\eta(\tau)^{16}}, \\
\widetilde{f}_g(\tau) & := 2^{-g-3} \left[ \frac{\Theta(z, \tau)^2}{\Theta(2z,2\tau)^2} \right]_{z^{2g-2}} \cdot \frac{\eta(\tau)^8}{\eta(2\tau)^{16}},
\end{align*}
where
\begin{equation} \label{defn:theta funciton}
\Theta(z) = \frac{\vartheta_1(z,\tau)}{\eta^3(\tau)} := (e^{z/2} - e^{-z/2}) \prod_{m \geq 1} \frac{(1-e^z q^m) (1-e^{-z} q^m)}{(1-q^m)^2} \end{equation}
is the normalized theta function.
\begin{lemma} \label{lemma:f_g 1}
Let $f_g^{\ast}, \widetilde{f}_g^{\ast}$ be the almost-holomorphic completions of $f_g, \widetilde{f}_g$. Then
$f_g$ is the Fricke involute of $\widetilde{f}_g$.
\end{lemma}
\begin{proof}
Let $\varphi(z,\tau) = \Theta(z,\tau)^2$ and define 
\begin{align*}
F(z,\tau) &:= \frac{\varphi(z, 2 \tau)}{\varphi(z,\tau)} \cdot \frac{\eta(2 \tau)^8}{\eta(\tau)^{16}} \\
\widetilde{F}(z,\tau) &:= \frac{1}{2^6} \frac{\varphi(z,\tau)}{\varphi(2z, 2 \tau)} \frac{\eta(\tau)^8}{\eta(2 \tau)^{16}}.
\end{align*}
Since $\varphi(z,\tau) = \Theta(z,\tau)^2$ is a Jacobi form of index $1$ and weight $-2$, it satisfies
\[ \varphi\left( \frac{z}{\tau}, -\frac{1}{\tau} \right) = \tau^{-2} e^{2 \pi i \frac{z^2}{\tau}} \varphi(z,\tau). \]
Moreover, we have $\eta^8(-1/\tau) = \tau^4 \eta(\tau)$. Hence we have the transformation
\[ e^{2 \pi i \frac{z^2}{\tau}} \tau^4 F\left( \frac{z}{\tau}, -\frac{1}{2 \tau}\right)
= \widetilde{F}(z,\tau) \]
Taking $z^{2g-2}$-coefficients, we find that
\[ 2^{\frac{6-2g}{2}} \tau^{6-2g} f_g(-1/2 \tau) + (\ldots) = \widetilde{f}_g(\tau) \]
where $(\ldots)$ stands for terms involving quasimodular forms of lower weight.
It follows that the Fricke involute of $f_g^{\ast}(\tau)$ is $\widetilde{f}_g^{\ast}(\tau)$.
\end{proof}
\begin{lemma} \label{lemma:f_g 2}
Write $\widetilde{f}_g = \sum_n c_n q^n$ and $f_g = \sum_n d_n q^n$. Then $c_{2n} + 2^{g-4} d_n = 0$.
\end{lemma}
\begin{proof}
We need to show that
\[ \widetilde{f}_g(\tau) + \widetilde{f}_g(\tau + 1) = 2^{g-4} f_g(2 \tau). \]
By rewriting this in terms of generating series, this follows from the identity
\[ \varphi(z,2 \tau) \eta(2 \tau)^8 =
-\frac{1}{4} \left[ \varphi\left(\frac{z}{2}, \frac{\tau}{2}\right) \eta\left( \frac{\tau}{2} \right)^8 + e^{\frac{2\pi i}{3}} \varphi\left( \frac{z}{2},\frac{\tau+1}{2} \right) \eta\left(\frac{\tau+1}{2}\right)^8 \right]. \qedhere
\]
\end{proof}
\begin{proof}[Proof of Proposition~\ref{prop:Hodge integral}]
Recall that we have the expansion
\[ \Theta(z) = z \exp\left(-2\sum_{k\geq 2} G_k(\tau) \frac{z^k}{k!} \right) \]
where $G_k$ are the weight $k$ Eisenstein series defined to be zero for odd $k$, and for even $k$ by
\[ G_k(\tau) = - \frac{B_k}{2 \cdot k} + \sum_{n \geq 1} \sum_{d|n} d^{k-1} q^n \]
Since $\frac{d}{dG_2} G_2(2 \tau) = \frac{1}{2}$ 
(e.g. \cite[Lemma.2.20]{HilbHAE}) it follows that
\[ \frac{\Theta(z,2 \tau)^2}{\Theta(z,\tau)^2} = \exp\left( z^2 G_2(\tau) + \ldots \right) \]
where $(...)$ only involves modular forms.
Hence
\[ \frac{d}{dG_2} f_g = f_{g-1}. \]
Thus $f_g$ is a quasimodular form of depth $g-1$ and weight $2g-2$ for the group $\Gamma_0(2)$.

Moreover, observe that
\[ \sum_{n} \omega_g(n) q^n  = f_g(\tau). \]

By Proposition~\ref{prop:lift of Gamma_0(p) ahm}
and Lemmas~\ref{lemma:f_g 1} and~\ref{lemma:f_g 2}
we thus get that $\Lift(F_{f^{\ast}_g})$ (with notation as in the proposition) is an almost-holomorphic modular form for $O^{+}(M)$ with constant term
\[ \ct(\Lift(F_{f^{\ast}_g}))
=
\sum_{\beta > 0} Q^{\beta} \sum_{\substack{\delta | \beta \\ \gcd(\delta,2) = 1}}
\delta^{2g-3}
\omega_{g}\left( \frac{(\beta,\beta)_L}{2 \delta^2} \right).
\]
where we used that for $\beta = (\ell,a,\alpha) \in L = U \oplus E_8(-1)$ we have $Q^{\beta} = q_1^a q_2^{\ell} \zeta^{\alpha}$.
Hence
\[ F_g = 2(-1)^{g-1} \cdot \ct(\Lift(F_{f^{\ast}_g})). \]
This proves the first claim.

For the holomorphic anomaly equation, let $F_g^{\ast}$ be the completion of $F_g$. Then we compute:
\begin{align*}
-4 \pi i \LO(F_g^{\ast} )
= &
-4 \pi i \cdot 2 (-1)^{g-1} \LO( \mathrm{Lift}(F_{f_g^{\ast}}) \\
=&  -2 i \cdot 2 (-1)^{g-1} \RO \mathrm{Lift}\left( \LM( F_{f_g}^{\ast} ) \right) \\
=& -4 i \cdot 2 (-1)^{g-1}\RO \mathrm{Lift}\left( F_{\LM(f_{g}^{\ast})} \right) = (\ast)
\end{align*}
where we used Theorem~\ref{thm:L on theta lift}
and the second claim in Proposition~\ref{prop:Gamma_0(p) to vector valued}. Then we use
\[ \LM(f_{g}) = -\frac{1}{8 \pi} \frac{d}{d G_2} f_g 
= -\frac{1}{8 \pi} f_{g-1} \]
to get
\begin{align*}
(\ast) & = \frac{1}{2 \pi i} \cdot 2 (-1)^{g-1}\RO \mathrm{Lift}\left( F_{f_{g-1}^{\ast}} \right) \\
& = - \frac{1}{2 \pi i} \RO F_{g-1}
\end{align*}
But this gives exactly the right hand side of the holomorphic anomaly equation by the computation done in 
\cite[Sec.4.4.4]{Enriques}.
\end{proof}

For $g=1$, the series $F_g$ is not defined as written so Conjecture~\ref{conj:main} does not apply.
However, for $\lambda \in H^2(Y,\BQ)$ we can consider the series
\[ F_{1,\lambda} := \int_{\Mbar_{1,1}} \CC_1(\lambda). \]
Then the evaluation in \cite{Enriques} gives that
\[
F_{1,\lambda} = -\frac{1}{8} \frac{\frac{d}{d \lambda} \Phi_4}{\Phi_4}
\]
where $\Phi_4$ is the Borcherds-Enriques form
and $\frac{d}{d \lambda} Q^{\beta} = (\beta,\lambda)_L Q^{\beta}$.
Thus we have
\[ F_{1, \star} = -\frac{1}{8} \frac{\partial(\Phi_4)}{\Phi_4} \]
and hence $F_{1, \star}$ lies in $\QMod_{1,1}^{\log}(\Gamma)$.

\subsection{Fourier-Jacobi series}
Consider the decomposition $L = U \oplus E_8(-1)$ and let $e,f$ be effective generators of $U$ with intersection form $\binom{0\ 1}{1\ 0}$. If $Y$ is generic, then $e,f$ are just the classes of two rigid smooth ellitpic curves intersecting in a point and the linear systems $|2e|$ and $|2f|$ define elliptic fibrations
for which they are half-fibers.
We can then
consider the partial series
\[
F_{g,k}( \taut ; \tau_{k_1}(\gamma_1) \cdots \tau_{k_n}(\gamma_n) )
=
\sum_{d \geq 0} \sum_{\alpha \in E_8(-1)}
q^d \zeta^{\alpha} 
\left\langle \taut ; \tau_{k_1}(\gamma_1) \cdots \tau_{k_n}(\gamma_n) \right\rangle^{Y}_{g,ke+df+\alpha}.
\]

Then by \cite[Theorem 4.3]{Enriques} the series $F_{g,k}( \taut; \tau_{m_1}(\gamma_1) \cdots \tau_{m_n}(\gamma_n) )$ is a quasi-Jacobi form
for $\Gamma_0(2) \ltimes H(N)$ of index $k$, with pole of order $\leq k$ at cusps. More precisely\footnote{The definition of quasi-Jacobi form in \cite{Enriques} is stronger than ours, it requires boundedness at cusps.},
\begin{equation} F_{g,k}(  \taut; \tau_{m_1}(\gamma_1) \cdots \tau_{m_n}(\gamma_n) ) \in 
	\left(\frac{ \eta^{8}(2 \tau) }{\eta^{16}(\tau)} \right)^k
	\QJac_{\frac{1}{2} k Q_{E_8}}(\Gamma_0(2)). \label{dfsdf333} \end{equation}
This matches perfectly with Conjecture~\ref{conj:main}.
Indeed, by Corollary~\ref{thm:FJ expansion of quasimodular forms}
the Fourier-Jacobi coefficients of orthogonal quasimodular forms 
for $\Gamma$ are quasi-Jacobi forms for $\Gamma_0(2) \ltimes H(N)$.

Moreover, the holomorphic anomaly equation 
proven in \cite{Enriques} for the $F_{g,k}$
is compatible with
the formula for the lowering operator applied to Fourier-Jaobi coefficients of orthogonal moduar forms, as established in Corollary~\ref{thm:FJ expansion of quasimodular forms}. 
Hence Conjecture~\ref{conj:HAE} is also compatible with taking Fourier expansions.

Thus we see that Conjectures~\ref{conj:main} and~\ref{conj:HAE}
hold on the level of Fourier-Jacobi expansions.
In our view this is the strongest evidence
that these conjectures holds for arbitrary insertions.

\begin{rmk}
We remark that \eqref{dfsdf333} is not satisfied by every logarithmic quasimodular form: A function $f \in \QMod^{\log}_{k,s}(\Gamma)$ satisfies \eqref{dfsdf333} if and only if every of its constant terms $[f]_{\tilde{q}^0}(v)$ lies in $\QMod(\Gamma_0(2))$.
Indeed, this follows from the expansion $\Phi_4 = \eta(2 \tau)^8/\eta(\tau)^6 + \sum_{\ell} a_{\ell} \tilde{q}^{\ell}$
where $a_{\ell}$ are holomorphic Jacobi forms bounded at the cusps.
This observation allows one to work with a slightly smaller subspace in practice, but is not extremely useful. For example,
$\Mod^{\log}_{4,0} \cong \Mod_{20}(\Gamma)$ has dimension $22$, and the forms satisfying \eqref{dfsdf333} have dimension $16$.
\end{rmk}

\subsection{The Enriques Calabi-Yau threefold}
\label{subsec:Enriques CY3}
Let $E$ be an elliptic curve,
let $X \to Y$ as before the K3 cover of the Enriques surface $Y$ and let $\iota : X \to X$ be the covering involution.
The Enrique Calabi-Yau threefold is the quotient
\[ Q = (S \times E)/\BZ_2 \]
where $\BZ_2$ acts by $(s,e) \mapsto (\iota(s), -e)$.
Modulo torsion, there is an isomorphism
\[ H_2(Q,\BZ) = H_2(Y,\BZ) \oplus \BZ [E'] \]
where $E'$ is the fiber of the elliptic fibration $\tilde{\pi} : Q \to Y$
and the embedding $H_2(Y,\BZ) \subset H_2(Q,\BZ)$ is given by pushforward along one of the $4$-sections of $\tilde{\pi}$.
We refer to \cite{MP,Enriques} for more details.

Consider the generating series
of Gromov-Witten invariants
for curve classes of degree $m$ over the base of the K3 fibration $\pi : Q \to [E/\BZ_2]$,
\[ F^Q_{g,m} = \sum_{\beta \geq 0} 
\langle 1 \rangle^{Q}_{g,(\beta,m)}
Q^{\beta}. \]

As before identify $H^2(Y,\BZ) \cong L = U \oplus E_8(-1)$ and $Q^{\beta} = \exp(2 \pi i (z,\beta))$.

By a degeneration argument as in \cite{MP},
the Conjectures~\ref{conj:main} and \ref{conj:HAE}
imply the following modularity statement for these relative potentials:
\begin{conj} \label{conj:Enriques CY3}
\label{conj:Enriques CY3} For $g > 1$ or $m>0$, the series $F^Q_{g,m}$ is a logarithmic quasimodular form
for the pair $(\CD_M,\CH_{-2})$ for the lattice $M=U\oplus U(2) \oplus E_8(-2)$ and group $\Gamma=O^{+}(M)$,
\[ F^Q_{g,m} \in \QMod^{\log}_{2g-2}(\Gamma). \]
It satisfies the holomorphic anomaly equation
\[ -4 \pi i \LO_{\lambda} F^Q_{g,m} = (2m-1) \frac{d}{d \lambda} F^Q_{g-1,m}
+ 2 \sum_{\substack{g=g_1+g_2 \\ m = m_1 + m_2}} m_1 F^Q_{g_1,m_1} \frac{d}{d \lambda} F^Q_{g_2,m_2}. \]
\end{conj}

\section{Bi-elliptic surfaces} \label{sec:bi-elliptic surfaces}
\subsection{Background}
An bi-elliptic (or hyperelliptic) surface is a minimal smooth projective surface $S$ of Kodaira dimension zero with $q(S)=1$ and $p_g(S)>0$.
They are constructed as quotients $(E \times F)/G$ of the product of two elliptic curves $E \times F$ by a finite subgroup $G$ of $F$, which acts on $F$ by translations and by $E$ on automorphisms, not all translations.
In this way, they are analogues of the Enriques surface, but where the covering is an abelian surface rather than a K3.
There are 7 deformation types of bielliptic surfaces, as classified by Bagnera-De Franchis \cite[Chapter 12]{CilibertoBook}. For simplicity {\em we will always assume below that $G$ is cyclic,} in which case $S$ is called of {\em split type}. We then have that 
\[ G=\BZ/N \BZ, \quad \text{ for } \quad N \in \{ 2,3,4, 6\} \]
and $G$ acts by translation with an $N$-torsion point on $F$ and by multiplication by an $N$-th root on $E$ (for $N=3,4,6$ we must choose a special $E$). 
Two elliptic fibrations exist on $S$. The first is induced by the projection from $E \times F$ to the second factor,
\[ \pi : S \to F/G \]
and it has a section $e=[0_E \times F]/G$. All fibers are isomorphic to $E$.
We let $f \in H^2(S,\BZ)$ be the class of a fiber.
The Neron-Severi group of $S$ modulo torsion is then $N(S) = \BZ e \oplus \BZ f$ with intersection form $\binom{0\ 1}{1\ 0}$. For the second, the projection to the first factor of $E \times F$ induces the fibration
\[ S \to \p^1 = E/G \]
which has fibers of multiplicity $N$ of class $N e$, and the $N$-section $f$.

\subsection{Modularity}
As before we consider the generating series of Gromov-Witten classes
\begin{equation} 
\label{CCg bielliptic} \CC_g^{S}(\gamma_1, \ldots, \gamma_n) = \sum_{a,b \geq 0} \CC_{g, a e + b f}^{S}(\gamma_1, \ldots, \gamma_n) q_1^a q_2^b. \end{equation}
For elliptic fibrations with sections, a modularity for fiberwise (or relative) Gromov-Witten potentials was conjectured in \cite{HAE}.
Applied to the fibration $S \to F/G$, this leads to the prediction that the $q_1$-coefficients of \eqref{CCg bielliptic} are quasimodular forms for $\SL_2(\BZ)$. In the other direction, a slight generalization of this conjecture for elliptic fibrations without sections (compare with the Enriques case \cite{Enriques}) leads to the prediction that the $q_2$-coefficients of \eqref{CCg bielliptic} are quasimodular for $\Gamma_0(N)$.
The holomorphic anomaly equations can be found likewise.
Altogether one has the following conjecture:

\begin{conj} \label{conj:bielliptic surface} Let $S$ be a bi-elliptic surface $(E \times F)/\BZ_N$ of split type. Then for $3g-3+n>0$ we have
\[ \CC_g^{S}(\gamma_1, \ldots, \gamma_n)
\in H^{\ast}(\Mbar_{g,n}) \otimes
\QMod(\Gamma_1(N)) \otimes \QMod(\SL_2(\BZ)) \]
Moreover, for every $i \in \{ 1,2 \}$ we have the holomorphic anomaly equation:
\begin{align*}
\frac{d}{dG_2(q_i)} 
\CC_g^{S}(\gamma_1, \ldots, \gamma_n)
& = 
	\iota_{\ast} \CC_{g-1}^S(\gamma_1, \ldots, \gamma_n, U_{\lambda_i}) \\
	& + \sum_{\substack{g = g_1 + g_2, \ell \\ \{ 1, \ldots, n \} = A \sqcup B }} 
	j_{\ast}( \CC_{g_1}(\gamma_A, \delta_\ell ) \boxtimes \CC_{g_2}(\gamma_B, \delta_\ell^{\vee} ) )  \\
	& - 2 \psi_i \CC_{g}( \gamma_1, \ldots, U_{\lambda_i}(\gamma_i), \ldots, \gamma_n).
\end{align*}
where $U_{\lambda_i} = \lambda_i \otimes 1 + 1 \otimes \lambda_i$ where
\[ \lambda_i = \begin{cases} e & \text{ if } i=1 \\ f & \text{ if } i=2 \end{cases}
\]
and as before we wrote $U_{\lambda_i} = \sum_{\ell} \delta_{\ell} \otimes \delta_{\ell}^{\vee}$ for the K\"unneth decomposition.
\end{conj}

\begin{rmk}
The bi-weight of the quasimodular form $\CC_g^{S}(\gamma_1, \ldots, \gamma_n)$ is determined from the above conjectures along the lines of \cite[Sec.3.1]{RES} as follows:
Let $\wt_{\lambda} \in \End H^{\ast}(S)$ be the operator
$\wt_{\lambda} = [U_{\lambda}, \lambda \cdot ]$, that is
$\wt_{\lambda}(\gamma) = U_{\lambda}(\lambda \gamma) - \lambda \cdot U_{\lambda}(\gamma)$. The operators $\wt_{e}, \wt_f$ commute and can be simultaneously diagonalized. Assume that each $\gamma_i\in H^{\ast}(S)$ is a common eigenvector with eigenvalues $a_i$ for $\wt_e$ and $b_i$ for $\wt_f$. Then
$\CC_g^{S}(\gamma_1, \ldots, \gamma_n)
\in 
\QMod_{k_1}(\Gamma_0(N)) \otimes \QMod_{k_2}(\SL_2(\BZ))
$
where
\begin{align*}
k_1 &= 2g-2+n + \sum_i a_i \\
k_2 & = 2g-2+n + \sum_i b_i.
\end{align*}
\qed
\end{rmk}

Evidence for this conjecture can already be found in genus $1$.
If we let $F_1 = \sum_{a,b} \langle 1 \rangle^{S}_{1,a e + b f} q_1^1 q_2^b$, then a basic degeneration argument shows
\[
F_1 =
\begin{cases}
-4 \log \eta(q_1) + 2 \log\eta(q_1^2) & \text{ if } N=2 \\
-3 \log \eta(q_1) + \log \eta(q_1^3)  & \text{ if } N=3 \\
-2 \log \eta(q_1) - \log \eta(q_1^2) + \log \eta(q_1^4)  & \text{ if } N=4 \\
- \log \eta(q_1) - \log \eta(q_1^2) - \log \eta(q_1^3) + \log \eta(q_1^6)
 & \text{ if } N=6
 \end{cases}
\]
It follows that
\[ \CC^S_1(e) = q\frac{d}{dq} F_1 \]
is quasimodular of the desired properties.
More evidence in higher genus has been obtained by T. Blomme in \cite[Thm 5.1]{Blomme}.
In particular he proved that\footnote{The statement in \cite[Thm 5.1]{Blomme} claims only modularity for $\Gamma_1(N)$ in the first variable, but this is equivalent to $\Gamma_0(N)$ since $\Gamma_0(N)$ is generated by
$\Gamma_1(N)$ and $-I_2$ for $N=2,3,4,6$, see \cite[Remark 5.2]{Blomme}.}
\[ \CC_{g}(\pt^{g-1}) \in \QMod(\Gamma_0(N)) \otimes \QMod(\SL_2(\BZ)). \]

We expect that Conjecture~\ref{conj:bielliptic surface} can be proven in general by a degeneration argument parallel to \cite{HAE}. 

\subsection{Orthogonal quasimodular forms}
The above conjectured modularity can be related to orthogonal quasimodular forms. To make this precise we only need to decide on the lattice.
Anticipating more general cases (such as abelian-surface fibered geometries)
let us argue conceptually what the correct lattice and group should be, parallel to the case of Enriques surfaces.
The mirror symmetry perspective is not well-developed for bielliptic surfaces.\footnote{The authors are not aware of a reference.
From the viewpoint of mirror symmetry of lattice-polarized abelian surfaces, bi-elliptic surfaces should be mirror to themselves. However, this doesn't match with a standard interpretation of mirror symmetry: for $N=3,4,6$, the elliptic curve $E$ is special and hence rigid, and thus the moduli space of bi-elliptic surfaces is $1$-dimensional. However, $h^{1,1}(S)$ is of rank $2$ and we have two K\"ahler parameters. Hence it is not clear how to think about $F_g$ as functions on the moduli space of bi-elliptic surfaces.}
Instead we gain insight from considering the group of auto-equivalences of the bielliptic surfaces as in Remark \ref{rmk:autoequivalences Enriques}.
For that, we recall a result of Potter that describes the group of derived autoequivalences of $S$:

Let $K(S)$ be the numerical Grothendieck $K$-group of the category of coherent sheaves on $S$. It is isomorphic via the Chern character to the even cohomology:
$\ch: K(S) \cong H^{2 \ast}(S,\BZ)$.
Let $\pi : E \times F \to S$ be the quotient map.
Consider the lattice
\[ \Lambda = \left\{ (r,a[E]+b[F],s) \in \bigoplus_{i=0}^{2} H^{2i}(E \times F,\BZ) \middle| r,a,b,s \in \BZ \right\}. \]
We define the Mukai pairing on $\Lambda$ as usual by
$(r,D,s) \cdot (r',D',s') = D \cdot D' - r s' - r' s$, such that
\[ \Lambda \cong U \oplus U. \]
We have the sublattice
\[
\Lambda_S := \pi^{\ast} K(S) = \left\{ (r,D,s) \in \Lambda \, \middle| \, D = a [E] + b [F] \text{ and } s,b \equiv 0 (\text{mod }N) \right\}. \]

Let $M_2(\BZ)$ be the set of integral $2 \times 2$-matrices
with quadratic form given by $2 \det$. Consider the isomorphism 
$\Lambda = U \oplus U \to M_2(\BZ)$ given by $(r,a[E]+b[F],s) \mapsto \binom{b\ r}{s\ a}$.
Recall the homomorphism
\[ \rho : \SL_2(\BZ) \times \SL_2(\BZ) \to O^{+}(U \times U) \]
defined by $\rho(A,B) \cdot M = A M B^{t}$.
Let 
$\Gamma \subset O(U \times U)$
be the image of $\SL_2(\BZ) \times \Gamma_0(N)$ under $\rho$.
Then $\Gamma$ preserves
    the sublattice $\pi^{\ast}K(S)$.
We view $\Gamma$ as a subgroup of $\GL(K(S))$.

\begin{thm}[Potter \cite{Potter}] $\Gamma = \mathrm{Im}\left( \Aut D^b(S) \to O(K(S)) \right)$. \end{thm}

Parallel to the case of Enriques surfaces it is thus natural to conjecture
that the Gromov-Witten series of a bielliptic surface
are given by orthogonal quasimodular forms
for $M=U\oplus U$ and $\Gamma = \rho(\SL_2(\BZ) \times \Gamma_0(N))$. Indeed, comparing with Section~\ref{subsec:example U+U} this is what our earlier predictions do:

\vspace{5pt}
\noindent
\textbf{Observation:}
If Conjecture~\ref{conj:bielliptic surface} holds, then (when restricted to even cohomology classes)
$\CC_g^S(-)$ is a cycle-valued quasimodular form for $M=U\oplus U$ and $\Gamma = \rho(\SL_2(\BZ) \times \Gamma_0(N))$,
that is precisely
\[ \CC^S_g|_{(H^0)^{\otimes a} \otimes (H^4)^{\otimes b} \otimes (H^2)^{\otimes s}} \in H^{\ast}(\Mbar_{g,n}) \otimes \QMod_{2g-2+n-a+b,s}(\Gamma) \]
for any $n=a+b+s$ and fixed identification $H^2(S,\BZ) \cong U$.

\section{Lattice polarized K3 and abelian fibrations}
In this last section we consider the modularity for relative Gromov-Witten potentials of K3 and abelian fibrations $\pi : X \to B$.
We restrict ourselves to families over $1$-dimensional bases and fiber curve classes (i.e. those which are contracted by $\pi$).
For nodal K3 fibrations, we show that the generating series of Gromov-Witten invariants have a natural completion which is {\em not} almost-holomorphic. We also consider the case of the abelian-fibered Banana Calabi-Yau threefold 
with a similar result.

\subsection{$1$-parameter families of $L$-polarized K3 surfaces}
Let $\Lambda_{K3} := U^3 \oplus E_8(-1)^2$ be the K3 lattice and let $L \subset \Lambda_{K3}$ be a primitive sublattice containing a distinguished vector of positive norm $v \in L$.
Then following \cite[1.2.1]{PT} a $1$-parameter family of $L$-polarized K3 surfaces consists of a smooth projective threefold $X$, an embedding map $j : L \to \Pic(X)$ and a map $\pi : X \to C$ to a smooth proper curve $C$ such that
\begin{enumerate}
\item[(i)] the fibers $(X_t, j_t : L \to \Pic(X_t))$ are $L$-polarized K3 surfaces for every $t \in C$,
\item[(ii)] the line bundle $j(v)$ defines a quasi-polarization of all fibers of $\pi$.
\end{enumerate}
We refer to \cite{PT} for the precise definition of $L$-polarized K3 surface and quasi-polarization.

The family $\pi : X \to C$ defines a morphism $\iota_{\pi} : C \to \CM_{L}$ to the moduli space of $L$-polarized K3 surfaces.
Intersecting with the Noether-Lefschetz divisors defines the generating series of Noether-Lefschetz numbers, see \cite{MPGWNL} and \cite[1.2.2]{PT}:
\[ \Phi^{\pi}(q) = \sum_{n} \sum_{\gamma \in L^{\vee}/L} a_{n,\gamma} q^n e_{\gamma} \]
By a result of Borcherds \cite{Borcherds1999}, $\Phi^{\pi}(q)$ is a modular form for $\mathrm{Mp}_2(\BZ)$ of weight 
$(22- \mathrm{rk}(L))/2$ for the Weil representation $\rho_M$ on $M^{\vee}/M \cong (L^{\vee}/L)(-1)$ where $M=L^{\perp}$ is the orthogonal complement in $\Lambda_{\text{K3}}$.

\subsection{Gromov-Witten invariants}
Let $L_1,\ldots,L_{\ell}$ be an integral basis of $L$.
Let $d = (d_1,\ldots, d_{\ell})$ be a vector of integers.
The Gromov-Witten invariants of $X$ in fiber curve classes of degree $d$ are defined by
\[ N_{g,d}^{X} := \sum_{\substack{\beta \in H_2(X,\BZ) \\ \pi_{\ast} \beta = 0 \\ \beta \cdot c_1(L_i) = d_i }}
\left\langle 1 \right\rangle^X_{g,\beta}. \]

For $g \geq 0$, consider the following quasimodular forms
(with $\Theta$ as defined in \eqref{defn:theta funciton}):
\[ f_g(\tau) := \left[ 
-\frac{1}{\Delta(\tau) \Theta^2(z,\tau)}
\right]_{z^{2g-2}} \quad \in \frac{1}{\Delta(\tau)} \QMod_{2g}.
\]
We set
\[
\Psi_{\pi,g}(q) := \Phi^{\pi}(q) \cdot f_g(\tau)
\]
which is of weight 
\[ 11- \frac{\rk(L)}{2} + (2g-12) = 2g-1-\rk(L)/2. \]
Consider the Fourier expansion of $\Psi_{\pi,g}$ which we write as follows:
\[
\Psi_{\pi,g}(q) = \sum_{\gamma \in L^{\vee}/L}
\Psi_{\pi,g}\big[ \gamma, n \big] \, q^n e_{\gamma}.
\]

\begin{thm}[Pandharipande-Thomas \cite{PT}] \label{thm:eval K3 fibrations}
The degree $d=(d_1,\ldots, d_{\ell}) \in L^{\vee}$ Gromov-Witten invariants of $X$ in fiber curve classes are
	\[
	N_{g,d}^{X} = (-1)^{g-1} \sum_{k|\gcd(d)} k^{2g-3} 
	\Psi_g\left[ \gamma(d/k), \frac{(d,d)_{L^{\vee}}}{2k^2} \right]
	\]
	where $\gamma(d) \in L^{\vee}/L$ is defined as the image of any element $\beta \in \Lambda_{K3}$ satisfying $L_i \cdot \beta_i = d_i$ for all $i$
    under the natural morphism $\Lambda_{K3} \to L^{\vee} \to L^{\vee}/L$.
\end{thm}

In the theorem above, if $e_i$ is a basis of $L$ with matrix $(e_i,e_j)_L = h_{ij}$ and inverse $h^{ij}$, then we have identified $d \in L^{\vee}$ with the element $d = \sum_{i,j} d_i h^{ij} e_j$, such that $(d,d)_{L^{\vee}} = \sum_{i,j} h^{ij} d_i d_j$.

\begin{proof}
By \cite{PT}, the Hodge integrals in the Gromov-Witten theory of a K3 surface $S$ are evaluated by the Katz-Klemm-Vafa formula:
\[
\int_{[ \Mbar_{g}(S,\beta) ]^{\text{red}}}(-1)^{g} \lambda_g
=
(-1)^{g-1} \sum_{k|\beta} k^{2g-3} \omega_g\left( \frac{\beta^2}{2k^2} \right)
\]
where $\beta \in H_2(S,\BZ)$ and the $\omega_g(n)$ are defined by
\[ f_g(\tau) = \sum_{n} \omega_g(n) q^n. \]
It follows from the Gromov-Witten/Noether-Lefschetz relation proven in \cite{MPGWNL} (see also \cite{KMPS}) that if
	we define $N_{g,d}^{X,\text{mc}}$ by
	$N_{g,d}^{X} = \sum_{k|d} k^{2g-3} N_{g,d/k}^{X,\text{mc}}$ then 
we get
	\[
	N_{g,d}^{X,\text{mc}} = \sum_{n} (-1)^{g-1} \omega_g(n) \cdot\Phi^{\pi}\left[ -\frac{1}{2} \frac{ \det \binom{h\ d^{t}}{d\ 2n} }{\det(g)}, \gamma(d) \right].
	\]
	The claim follows by observing that this translates to
	\[N_{g,d}^{X,\text{mc}} = \Psi_{\pi,g}\left[ d h^{-1} d^t/2, \gamma(d) \right].
	\]
\end{proof}

Written in generating series form, the theorem says that
\[
F_{g} :=
\sum_{d=(d_1,\ldots,d_{\ell})>0} N_{g,d} Q^d
=
\sum_{\substack{\lambda \in L^{\vee} \\ \lambda>0}} \sum_{k \geq 1} k^{2g-3} \Psi_{\pi,g}[ \lambda, (\lambda,\lambda)/2 ] Q^{\lambda},
\]
where $Q^{\lambda} = e^{2 \pi i (z,\lambda)_L}$ for a formal variable $z \in L \otimes \BC$. 
Comparing with \eqref{expansion of constant term of lift},  the right hand side for $g \geq 2$ looks exactly like the Fourier expansion of the constant term of the lift of the almost-holomorphic completion of $\Psi_{\pi,g}$.
This lift has weight $\wt=2g-2$.
However,
$f_g$ is of depth $g$, and thus the depth-weight inequality
$\wt(f_g) \geq 2\cdot \text{depth}(f_g)$
of Theorem~\ref{thm:list of almost-holomorphic modular form} does not hold. Hence the natural completion of $F_g$,\
given by the lift of $\Psi_{\pi,g}^{\ast}$, is not almost-holomorphic.\footnote{Unfortunately we can not rule out that there is some other completion of $F_g$ that is almost-holomorphic. However, we believe that this is unlikely.}

\subsection{An example: the STU model}
To make the discussion more concrete, we specialize to the
well-studied \cite{KMPS} STU threefold. This is a particular Calabi-Yau threefold $X$ with an elliptic fibration 
\[ p : X \to \p^1 \times \p^1 \]
and section 
\[ \sigma : \p^1 \times \p^1 \to X. \]
Let $p_i : \p^1 \times \p^1 \to \p^1$ for $i=1,2$ be the projections to the factors.
The composition
\[ \pi = p_1 \circ p : X \to \p^1 \]
realizes $X$ as a $1$-parameter family of K3 surfaces.

By Lemma~\ref{lemma:intersections}(a) below, the fibration $\pi$ is $U$-polarized by the divisor classes
\[
\CL_1 = \omega_2 + k \omega_1, \quad \CL_2 = \Sigma + \omega_2 + \ell \omega_1
\]
for any $k,\ell \in \BQ$,
where $\Sigma = [ \sigma(\p^1 \times \p^1) ]$, $\omega_1 = p^{\ast} p_1^{\ast}(\omega)$, and $\omega_2 = p^{\ast} p_2^{\ast}(\omega)$,
with $\omega \in H^2(\p^1)$ the point class,
and we have identified line bundles with their images in $H^2(X,\BZ)$.

\begin{lemma} \label{lemma:intersections} We have $\Sigma^3 = 8$,\ $\Sigma^2 \cdot \omega_i = -2$,\ $\Sigma \omega_1 \omega_2 = 1$,\ $\omega_i^2=0$. Hence
	\begin{enumerate}
		\item[(a)] $\CL_1^2 \omega_1 = 0, \CL_1 \CL_2 \omega_1 = 1, \CL_2^2 \omega_1 = 0$
		\item[(b)] $\CL_1^3 = 0,  \CL_1^2 \CL_2 = 2k, \CL_1 \CL_2^2 = 2 \ell - 2, \CL_2^3 = 2$.
	\end{enumerate}
\end{lemma}

\begin{rmk}
For $k=\ell=1$, the divisors $\CL_1, \CL_2$ also define a $U$-polarization
for the K3 fibration $p_2 \circ p : X \to \p^1$, hence this seems the natural choice.
Moreover, the canonical (in the sense of \cite{RES}) divisor associated to the elliptic fibration $p$ is
$W = \Sigma - \frac{1}{2} c_1(N_{\sigma}) = \Sigma + \omega_1 + \omega_2$
so from this point of view it is also natural to choose $\ell=1$ above. 
Hence from now on we set $k=\ell=1$. \qed
\end{rmk}


The generating series of Noether-Lefschetz numbers of the family $\pi : X \to \p^1$ was computed in \cite{KMPS} to be
$\Phi^{\pi} = -2 E_{10}(q) = -2 ( 1 - 264 q + ... )$.
Define coefficients $c(n)$ by $\sum_{n} c(n) q^n = E_{10} / \Delta(q)$.
Then by Theorem~\ref{thm:eval K3 fibrations} we have
for $(d_1, d_2) \neq 0$:
\[
N_{0,(d_1,d_2)} q_1^{d_1} q_2^{d_2}
=
-2 \sum_{k|(d_1, d_2)} \frac{1}{k^3} c(d_1 d_2/k^2)
\]

Then we obtain the evaluation (with $\sum_n b(n) q^n = D_{\tau}^2( E_{10} / \Delta )$)
\begin{equation} \label{Dfsd3t311}
\begin{aligned}
\sum_{(d_1,d_2) > 0} \langle \tau_0(\CL_1)^2 \tau_0(\CL_2)^2 \rangle^{X}_{0,(d_1,d_2)}=
\sum_{\substack{ (d_1, d_2) \in \BZ^2_{\neq 0} \\ d_1 \geq 0, d_2 \geq -d_1}} d_1^2 d_2^2 N_{0,(d_1,d_2)} q_1^{d_1} q_2^{d_2} =
\sum_{(d_1, d_2)>0} k\cdot b(d_1 d_2/k^2) q_1^{d_1} q_2^{d_2}.
\end{aligned}
\end{equation}
The natural completion of this series is
the lift of the (completion of) $D_{\tau}^2(E_{10}/\Delta)$
which is not almost-holomorphic by Theorem~\ref{thm:list of almost-holomorphic modular form}, see also Section~\ref{example:U+U}. More strongly, by using that the algebra of quasimodular forms for $U\oplus U$ are given by $(\QMod \otimes \QMod)^{\BZ_2}$, one can show here directly that \eqref{Dfsd3t311} is not quasimodular (when we allow possible poles along the diagonal and cusps).

\subsection{The Banana Calabi-Yau threefold}
Let $\pi_R: R \to \p^1$ be a rational elliptic surface with $12$ nodal fibers and with fixed section $B \subset R$. The banana Calabi-Yau threefold $X$ is the small resolution of the fibered product $R \times_{\p^1} R$ obtained by blowing up the diagonal divisor,
\[ X = \mathrm{Bl}_{\Delta}(R \times_{\p^1} R). \]
Let $\pi_i : X \to R$ for $i=1,2$ be the projections to the factors.
The fibration
\[ \pi := \pi_R \circ \pi_1 = \pi_R \circ \pi_2 : X \to \p^1 \]
is an abelian surface fibration over $\p^1$
whose smooth fibers are squares of elliptic curves. It has $12$ singular fibers, called banana fibers  \cite{Banana}, which are obtained as blow-ups of the squares of the $12$ singular fibers of $\pi_R$.
There exists a lattice polarization of the fibers of $\pi$ by
\[ L = \BZ \omega_1 \oplus \BZ \omega_2 \oplus \BZ \omega_3 \subset \Pic(X) \]
where
\[ \omega_1 = \pi_1^{\ast}(B), \quad \omega_2 = \pi_2^{\ast}(B), \quad \omega_3 = [ \widetilde{\Delta} ] \]
where $\widetilde{\Delta}$ is the proper transform of the diagonal.
The intersection form of these divisor classes restricted to a generic fiber
is given by the matrix
\[
\begin{pmatrix}
0 & 1 & 1 \\
1 & 0 & 1 \\
1 & 1 & 0
\end{pmatrix}
\]
Hence $\pi$ is polarized by the lattice $U \oplus (-2)$.
Dually, the sublattice
$H_2(X,\BZ)_{\pi} \subset H_2(X,\BZ)$
 generated by effective classes that are contracted under $\pi$
 is
\[ H_2(X,\BZ)_{\pi} = \BZ [C_1] \oplus \BZ[C_2] \oplus \BZ[C_3] \]
where $C_i$ are the banana curves (of any singular fiber).
There is a natural isomorphism
\[ H_2(X,\BZ)_{\pi} \cong L^{\ast} \]
and the $C_i$ are the dual basis of the $\omega_i$.

The (unstable) genus $0$ Gromov-Witten fiber potential is
\[
F_0 = 
\sum_{(d_1,d_2,d_3)>0}
N_{0, d_1 C_1 + d_2 C_2 + d_3 C_3}
e^{2 \pi i (z_1 d_1 + z_2 d_2 + z_3 d_3)}
\]
where the sum is over non-negative integers $d_i$ not all zero.

By Bryan--Pietromonaco \cite[App.A.4]{Banana} it is evaluated to be
\[
F_0 = 12 \sum_{(d_1,d_2,d_3)>0} \sum_{k \geq 1} k^{-3} a( d_1 d_2 + d_1 d_3 + d_2 d_3 - d_1^2 - d_2^2 - d_3^2 ) e^{2 \pi i k (d_1 z_1 + d_2 z_2 + d_3 z_3)}
\]
where $a(n)$ are the Fourier coefficients of the following weight $-5/2$ modular form for $\Gamma(2)$:
\[
f(\tau) = \sum_{n\geq -1/2} a(n) q^n = \frac{16}{\theta_3(\tau) \theta_2(\tau)^4} = q^{-1/2} - 2 + 8 q^{3/2} - 12 q^2 + 39 q^{7/2} \pm \ldots
\]
where $\theta_3(\tau) = \sum_{n \in \BZ} q^{n^2/2}$ and $\theta_2(\tau) = \sum_{n \in \BZ + \frac{1}{2}} q^{n^2/2}$.

Using the isomorphism $H_2(X,\BZ)_{\pi} \cong L^{\ast}$,
we can rewrite this expression for $F_0$ as
\[ F_0 = \sum_{\substack{\lambda \in L^{\vee} \\ \lambda>0}}
a(\lambda^2) \sum_{k \geq 1} k^{-3}e^{2 \pi i k (d_1 z_1 + d_2 z_2 + d_3 z_3)}. \]
In other words, $F_0$ is formally the expansion of the theta lift (in this case, the Shimura lift) of $f$ for the lattice $M=U \oplus L$. (Of course, the theta lift is not well-defined in this case since the weight of $f$ is too low.)
For $a,b,c\in \{ 1,2, 3 \}$, 
consider the (stable) genus 0 potentials
\[ F_0(\tau_0(\omega_a) \tau_0(\omega_b) \tau_0(\omega_c))
= \int_{X} \omega_a \omega_b \omega_c + \frac{1}{(2 \pi i)^3}\frac{d}{dz_a} \frac{d}{dz_b} \frac{d}{dz_b} F_0. \]
A certain linear combination of these have as a natural completion the form $\RO( \Lift( \RM f))$. However, this completion is not almost-holomorphic by Theorem~\ref{thm:list of almost-holomorphic modular form}.

\bibliographystyle{plainnat}
\bibliofont
\bibliography{references}{}

\end{document}